\documentclass[11pt,english]{amsart}
\usepackage[T1]{fontenc}
\usepackage[utf8]{inputenc}
\usepackage{geometry}
\usepackage{amsmath}
\usepackage{amsthm}
\usepackage{amssymb}
\usepackage{color}
\usepackage{verbatim}

\theoremstyle{plain}
\newtheorem{prop}{\protect\propositionname}
\theoremstyle{plain}
\newtheorem{lem}{\protect\lemmaname}
\theoremstyle{plain}
\newtheorem{thm}{\protect\theoremname}
\theoremstyle{plain}
\newtheorem{defn}{\protect\definitionname}
\theoremstyle{plain}
\newtheorem{cor}{\protect\corollaryname}

\DeclareMathOperator{\esssup}{ess sup}

\DeclareMathOperator{\supp}{supp}

\makeatother

\usepackage{babel}
\providecommand{\lemmaname}{Lemma}
\providecommand{\propositionname}{Proposition}
\providecommand{\theoremname}{Theorem}
\providecommand{\definitionname}{Definition}
\providecommand{\corollaryname}{Corollary}

\title{Matrix weighted estimates on spaces of homogeneous type}
\author[G. Claro]{Guido Claro}
\address{(G. Claro) Escuela de Producción, Tecnología y Medio Ambiente. Universidad Nacional de Río Negro, Sede Andina, EL Bolsón, Río Negro, Argentina}
\email{gclaro@unrn.edu.ar }
\author[P. A. Muller]{Pamela A. Muller}
\address{(P. A. Muller) Departamento de Matemática. Universidad Nacional del Sur. Bahía Blanca, Argentina}
\email{pamela.muller@uns.edu.ar}
\author[L. Nowak]{Luis Nowak}
\address{(L. Nowak) Departamento de matemática, FaEA UNComa, IITCI (CONICET-UNCo), Neuquén, Argentina}
\email{luis.nowak@faea.uncoma.edu.ar}
\author[A. Perini]{Alejandra Perini}
\address{(A. Perini) Departamento de matemática, FaEA UNComa, IITCI (CONICET-UNCo), Neuquén, Argentina}
\email{alejandra.perini@faea.uncoma.edu.ar}
\author[I. P. Rivera-R\'{\i}os]{Israel P. Rivera-R\'{\i}os}
\address{(I. P. Rivera-Ríos) Departamento de Análisis Matemático, Estadística e Investigación Operativa y Matemática Aplicada. Facultad de Ciencias, Universidad de Málaga. Málaga, España}
\email{israelpriverarios@uma.es}

\usepackage{scalerel}
\usepackage[usestackEOL]{stackengine}

\def\dashint{\,\ThisStyle{\ensurestackMath{%
            \stackinset{c}{.2\LMpt}{c}{.5\LMpt}{\SavedStyle-}{\SavedStyle\phantom{\int}}}%
        \setbox0=\hbox{$\SavedStyle\int\,$}\kern-\wd0}\int}

\usepackage{hyperref}
\begin{document}
\begin{abstract}
In this paper matrix quantitative weighted estimates on spaces of homogeneous type, such as endpoint estimates, strong type estimates are provided. To that end we extend some earlier results on convex body domination due to Nazarov, Petermichl, Treil and Volberg \cite{NPTV} to this setting. We also provide a $T(1)$ alike convex body domination result analogous to the one provided by Lerner and Ombrosi in \cite{LO}, and an application to vector valued extensions of Petermichl operators.
\end{abstract}
\maketitle

\section{Introduction and Main Results}

Given a linear operator $T$ that transforms scalar valued functions into scalar valued functions, a reasonable question to address is how to extend such an operator to
a vector valued one.

Several kinds of extensions have been studied in the literature. For example. If $\mathcal{F}_i \subseteq \{f:X \longrightarrow \mathbb{R} \}, i=1,2$ and $T: \mathcal{F}_1 \longrightarrow  \mathcal{F}_2$ we can consider, for  $\{f_j\}_{j=1}^\infty \subseteq \mathcal{F}_1$ and $1\leq q<\infty$ the new operator $\vec{T}$ given by 
\[\vec{T}\vec{f}(x) = \left(\sum_{j=1}^\infty|Tf_j(x)|^q\right)^{\frac{1}{q}}.\]
The study of this kind of extensions, that can be traced back to the thirties with works of authors such as Bochner, Marcinkiewickz, Paley and Zygmund, is a source of classical questions. See for instance \cite{CMO} for some recent generalizations and some further background in that direction. This kind of extensions have also been studied within the realm of the theory of weights.  See for instance works such as \cite{FS, AJ, RRT}.

Another possibility would be the following. Let $\{\vec{e_{j}}\}_{j=1}^{n}$ be an orthonormal
basis of $\mathbb{R}^{n}$ and $\vec{f}:X\rightarrow\mathbb{R}^{n}$ we can 
define 
\[
\vec{T}\vec{f}(x)=\sum_{j=1}^{n}T(\langle \vec{f},\vec{e_{j}}\rangle)(x)\vec{e_{j}},
\]
where $\langle \vec{f},\vec{e_{j}}\rangle$ stands for the usual inner of the vectors $\vec{f}$ and $\vec{e_{j}}$ product in $\mathbb{R}^{n}$.
For this kind of extension studying $\|\vec{T}\vec{f}\|_{L^p(w)}$ with scalar weights would be trivial, provided we are not interested on the precise dependence on $n$. In the sequel, we will use the same $T$ to denote the extensions $\vec{T}$.

Another possibility would be to consider $n\times n$ order matrices as weights. That kind of extension was considered for first by Treil and Volberg in \cite{TV}. We say that a matrix function $W:X\rightarrow\mathbb{R}^n\times\mathbb{R}^n$ is a weight if it is positive-definite and symmetric a.e. Given a matrix weight $W$ and $1\leq p<\infty$ we define 
\[ \|\vec{f}\|_{L^p(W)}=\left(\int|W^\frac{1}{p}\vec{f}|^p\right)^\frac{1}{p}.\]
Note that in the case $W(x)=I_n$ where $I_n$ stands for the identity matrix, we shall drop $W$ from the notation and write just
\[ \|\vec{f}\|_{L^p}=\left(\int|\vec{f}|^p\right)^\frac{1}{p}.\]

As we were saying Treil and Volberg introduced in \cite{TV} the study of inequalities in this setting. They proved that  
\[\|H\vec{f}\|_{L^2(W)}\lesssim\|\vec{f}\|_{L^2(W)}\]
where $H$ is the Hilbert transform if and only if, $W\in A_2$, namely if 
\[[W]_{A_2}=\sup_{I}\frac{1}{|I|}\int_I\frac{1}{|I|}\int_I|W^{-\frac{1}{2}}(y)W^{\frac{1}{2}}(x)|_{op}^2dydx<\infty\]
where $I$ stand for intervals in $\mathbb{R}$.

Since that work the theory was developed with the contributions of a number of authors.  For instance, a weighted version of the Hardy-Littlewood maximal was introduced by Christ and Goldberg in \cite{CG} for $p=2$. The theory for singular integrals  was extended in \cite{NT,V} for every $p>1$. Also Goldberg \cite{G} provided estimates for every $p$ for weighted maximal functions and for weighted maximal Calder\'on-Zygmund operators. 

In the last years quantitative matrix weighted estimates have attracted the interest of a number of authors in the area. The first sharp estimate, for the Christ-Goldberg maximal function, was provided by Isralowitz, Kwon and Pott in \cite{IKP}. In that paper bounds for sparse operators were obtained as well. Not much later, sparse domination technology was developed in this setting. Nazarov, Petermichl, Treil and Volberg \cite{NPTV} showed that suitable convex bodies were a good replacement for averages. Relying upon that result they proved that for every  Calder\'on-Zygmund operator $T$, 
\begin{equation}\label{eq:T32}
\|T\vec{f}\|_{L^2(W)}\lesssim[W]_{A_2}^{\frac{3}{2}}\|\vec{f}\|_{L^2(W)}.
\end{equation}
Hence the $A_2$ Theorem remained an open question up until very recently when Domelevo, Petermichl, Treil and Volberg \cite{DPTV} have shown for the Hilbert transform the dependence on the $A_2$ constant displayed in \eqref{eq:T32} is sharp. 

In that context, a number of papers have been devoted to the study of matrix weighted estimates. For instance, extrapolation and factorization with matrix weights were developed by Bownik and Cruz-Uribe \cite{BCU}, the estimate \eqref{eq:T32} was extended for $p\not=2$ in \cite{CUIM} where some two weight bump estimates were studied as well, extensions to the multiparameter and the multilinear setting were provided in \cite{DKPSiG} and  \cite{KN} respectively, endpoint estimates were introduced in \cite{CUIMPRR} and some of them were shown to be sharp in \cite{LeLiORR}. Also recently results for commutators were obtained in \cite{IPT}. To sum up this has been a quite active field of research in the recent years.

A common point of the aforementioned works is that they had $\mathbb{R}^d$ as an underlying space on which the scalar functions were defined. Our purpose in this work is to replace $\mathbb{R}^d$ by a suitable class of spaces of homogeneous type. The cornerstone result of this paper will be a convex body domination result that generalizes the convex body domination result in \cite{NPTV} to spaces of homogeneous type under the more flexible hypothesis in \cite{LO,L}. Let us provide a few definitions before presenting that result.

We say that $(X,d,\mu)$ is a space of homogeneous type if $X$ is
a set equiped with a quasi-metric $d$ and a doubling Borel measure
$\mu$. Recall that $d$ is a quasi-metric if the usual triangle inequality
holds but with an additional constant, namely if there exists a constant
$c_{d}\geq1$ such that
\[
d(x,y)\leq c_{d}\left(d(x,z)+d(z,y)\right)\qquad x,y,z\in X
\]
Since $\mu$ satisfies the doubling property, there exists $c_{\mu}\geq1$
such that 
\[
\mu(B(x,2\rho))\leq c_{\mu}\mu(B(x,\rho))\qquad x\in X,\quad\rho>0
\]
where $B(x,\rho):=\{y\in X\,:\,d(x,y)<\rho\}$. We shall assume additionally
that all balls $B$ are Borel sets and that $0<\mu(B)<\infty$. 

Since $\mu$ is a Borel measure defined on the Borel $\sigma$-algebra
of the quasi-metric space $(X,d)$ the Lebesgue differentiation theorem
holds and hence continuous functions with bounded support are dense
in $L^{p}(X)$ for every $1\leq p<\infty$. 

If $g:X\rightarrow\mathbb{R}$, for $1\leq s<\infty$ we shall call
\[
\langle g\rangle_{s,Q}=\left(\frac{1}{\mu(Q)}\int_{Q}|g|^{s}d\mu\right)^{\frac{1}{s}}
\]
and for $\vec{f}:X\rightarrow\mathbb{R}^{n}$
\[
\langle\langle \vec{f}\rangle\rangle_{s,Q}=\left\{ \frac{1}{\mu(Q)}\int_{Q}\phi \vec{f}\,:\,\|\phi\|_{L^{s'}\left(\frac{d\mu}{\mu(Q)}\right)}\leq1\right\} .
\]
With the definition above, $\langle\langle \vec{f}\rangle\rangle_{s,Q}$
is a compact, convex, symmetric set.

Given an operator $T$, following  \cite{L, LO} we define
\[
\mathcal{M}_{T,\alpha}^{\#}f(x)=\sup_{B\ni x}\esssup_{y,z\in B}|T(f\chi_{X\setminus\alpha B})(y)-T(f\chi_{X\setminus\alpha B})(z)|
\]
where each $B$ is a ball and for $\alpha \in \mathbb{R^+}$ the set $\alpha B$ is the ball $B$ dilated with factor $\alpha$.

In what follows, we shall use a particular case of the mixed norms introduced in \cite{BP}. That is, for a scalar function $g=g(x_1,x_2)$ where $x_i \in (X_i,\mu_i)$ for $i=1,2$ with each $(X_i,\mu_i)$ a measure space and $1\leq p_1,p_2\leq\infty$, we will write 
\[
||g||_{L^{{(p_{1},p_{2})}}[(A_{1},\mu_{1}),(A_{2},\mu_{2})]}=\|\|g(\cdot_{1},\cdot_{2})\|_{L^{p_{1}}(A_{1},d\mu_{1}(\cdot_{1}))}\|_{L^{p_{2}}(A_{2},d\mu_{2}(\cdot_{2}))}.
\]

Having the definitions above at our disposal we are in the position to state our convex body domination result. This Theorem generalizes the result obtained in \cite{NPTV} and will be a cornerstone in this work.

\begin{thm}
\label{thm:ThmSparse}Let $(X,d,\mu)$ be a space of homogeneous type 
and $\mathcal{D}$ a dyadic system with parameters $c_{0}$, $C_{0}$
and $\delta$. Let us fix $\alpha\geq\frac{3c_{d}^{2}}{\delta}$ and
let $\vec{f}:X\rightarrow\mathbb{R}^{n}$ be a boundedly supported
function such that $|\vec{f}|\in L^{s}(X)$. Let $1\le q,r<\infty$
and $s=\max(q,r)$, and assume that there exist non-increasing functions
$\psi$ and $\phi$ such that for every ball $B$  and any
boundedly supported function  $g\in L^{s}(X,\mathbb{R})$

\[
\mu\left(\{x\in B:|T(g\chi_{B})(x)|>\psi(\lambda)\langle g\rangle_{q,B}\}\right)\le\lambda\mu(B)\quad(0<\lambda<1)
\]
and 
\[
\mu(\{x\in B:\mathcal{M}_{T,\alpha}^{\#}(g\chi_{B})(x)>\phi(\lambda)\langle g\rangle_{r,B}\})\le\lambda\mu(B)\quad(0<\lambda<1)
\]
Then there exists a $\frac{1}{2}$-sparse family $\mathcal{S}\subset\mathcal{D}$
such that 
\[
T\vec{f}(x)\in c_{n,d}c_{T}\sum_{Q\in\mathcal{S}}\langle\langle\vec{f}\rangle\rangle_{s,\alpha Q}\chi_{Q}(x)
\]
In other words we have that
\[
T\vec{f}(x)=c_{n,d}c_{T}\sum_{Q\in\mathcal{S}}\frac{1}{\mu(\alpha Q)}\int_{\alpha Q}\vec{f}(y)k_{Q}(x,y)d\mu(y)\chi_{Q}(x)
\]
where
\[
\left\Vert k_{Q}\right\Vert _{L^{(s',\infty)}\left[(Q,\frac{\mu}{\mu(Q)}),(\alpha Q,\mu)\right]}\leq1. 
\]
Furthermore, there exist $0<c_{0}\leq C_{0}<\infty$,
$0<\delta<1$, $\gamma\geq1$ and $m\in\mathbb{N}$ such that there
are dyadic systems $\mathcal{D}_{1},\dots,\mathcal{D}_{m}$ with parameters
$c_{0},C_{0}$ and $\delta$ and $m$ sparse families $\mathcal{S}_{i}\subset\mathcal{D}_{i}$ 
such that 
\[
T\vec{f}(x)=c_{n,d}c_{T}\sum_{j=1}^{m}\sum_{Q\in\mathcal{S}_{j}}\frac{1}{\mu(Q)}\int_{Q}\vec{f}(y)k_{Q}(x,y)d\mu(y)\chi_{Q}(x)
\]
where\[\left\Vert k_{Q}\right\Vert _{L^{(s',\infty)}\left[(Q,\frac{d\mu(y)}{\mu(Q)}),(Q,d\mu(x))\right]}\leq1  \,. \]
\end{thm}
We remit to Section \ref{sec:dyadicSt} for some further insight on the dyadic structures involved in the statement of the theorem above. As a consequence of this result we will derive a number of quantitative weighted estimates and applications and we will as well generalize the $T(1)$-sparse theorem in \cite[Theorem 4.1]{LO} to spaces of homogeneous type and to the vector valued setting. In the following subsections we gather some of the main results.

\subsection{Weighted estimates}
As we pointed out earlier, the matrix weighted theory was initiated by Treil and Volberg for $p=2$ and further extended by Volberg for every $p>1$. Here we shall use the presentation of the $A_p$ classes in \cite{R}, and also use the definition of the $A_1$ class in \cite{FR}. We say that a matrix weight $W$ is in $A_{p}$ if
\begin{align*}
[W]_{A_{p}} & =\sup_{B}\frac{1}{|B|}\int_{B}\left(\frac{1}{|B|}\int_{B}|W^{-\frac{1}{p}}(y)W^{\frac{1}{p}}(x)|_{op}^{p'}d\mu(y)\right)^{\frac{p}{p'}}d\mu(x)<\infty\qquad1<p<\infty\\{}
[W]_{A_{1}} & =\sup_{B}\esssup_{y\in B}\frac{1}{|B|}\int_{B}|W(x)W^{-1}(y)|_{op}d\mu(y)<\infty\qquad p=1
\end{align*}
where the suprema above are taken over balls of the space.

We recall that the Christ-Goldberg maximal function is defined as follows. Given a weight $W$ we define
\[M_{p,W}\vec{f}(x) =\sup_{B\ni x}\frac{1}{\mu(B)}\int_{B}\left|W^{\frac{1}{p}}(x)W^{-\frac{1}{p}}(y)\vec{f}(y)\right|d\mu(y).\]

\begin{thm}\label{thm:StrongMax} Let $(X,d,\mu)$ be a space of homogeneous type. Let $1\leq  q<p<\infty$. If $W\in A_{p}$ we have that
\[
\|M_{p,W}\vec{f}\|_{L^{p}}\lesssim[W]_{A_{p}}^{\frac{1}{p}}[W^{-\frac{1}{p-1}}]_{A_{\infty,p'}^{sc}}^{\frac{1}{p}}\|\vec{f}\|_{L^{p}}
\]
and if $W\in A_q$,  then
\[
\|M_{p,W}\vec{f}\|_{L^{p}}\lesssim[W]_{A_{q}}^{\frac{1}{p}}\|\vec{f}\|_{L^{p}}.
\]
\end{thm}
This result extends the sharp estimates obtained in \cite{IM,IKP} to spaces of homogeneous type. As in those works, the sharp dependence of  the scalar case is recovered as well.

Let $1\leq r\leq\infty$. We say that $T$ is an $L^{r}$-H\"ormander singular operator if $T$
is bounded on $L^{2}$ and it admits the following representation
\begin{equation}
Tf(x)=\int_{X}K(x,y)f(y)dy\label{eq:Rep}
\end{equation}
provided that {\color{blue}$f\in\mathcal{C}_{c}$} and $x\not\in\supp f$
where $K:X\times X \setminus\left\{ (x,x)\,:\,x\in X \right\} \rightarrow\mathbb{R}$
is a locally integrable kernel satisfying the size condition $|K(x,y)|\leq \dfrac{C}{\mu(B(x,d(x,y)))}$ and the $L^{r}$-H\"ormander condition,
namely 
\begin{align}
\label{eq:Hr1}H_{r,1} & =\sup_{B}\sup_{x,z\in\frac{1}{2c_d}B}\sum_{k=1}^{\infty}\mu(2^kB)\left\langle \left(K(x,\cdot)-K(z,\cdot)\right)\chi_{2^{k}B\setminus2^{k-1}B} \right\rangle_{r,2^kB} <\infty,\\
\label{eq:Hr2}H_{r,2} & =\sup_{B}\sup_{x,z\in\frac{1}{2c_d}B}\sum_{k=1}^{\infty}\mu(2^kB)\left\langle \left(K(\cdot,x)-K(\cdot,z)\right)\chi_{2^{k}B\setminus2^{k-1}B}\right\rangle _{r,2^kB}<\infty,
\end{align}
for $r<\infty$ and
\begin{align}
	\label{eq:Hinf1}H_{\infty,1} & =\sup_{B}\sup_{x,z\in\frac{1}{2c_d}B}\sum_{k=1}^{\infty}\mu(2^kB)\Vert \left(K(x,\cdot)-K(z,\cdot)\right)\chi_{2^{k}B\setminus2^{k-1}B} \Vert_{L^\infty(2^kB)} <\infty,\\
	\label{eq:Hinf2}H_{\infty,2} & =\sup_{B}\sup_{x,z\in\frac{1}{2c_d}B}\sum_{k=1}^{\infty}\mu(2^kB)\Vert \left(K(\cdot,x)-K(\cdot,z)\right)\chi_{2^{k}B\setminus2^{k-1}B}\Vert _{L^\infty(2^kB)}<\infty,
\end{align}	
for $r=\infty$.

Note that if a Calder\'on-Zygmund operator has associated kernel $K$ satisfying the size condition and  a Dini condition in both variables, then $K$  satisfies the $L^\infty$-H\"ormander condition.

Our results for the operators defined above are the following.
\begin{thm} \label{Thm:StrongTypeLrH} Let $(X,d,\mu)$ be a space of homogeneous type.
Let $1\leq r<p<\infty$ and let $T$ be a $L^{r'}$-H\"ormander operator.  \begin{enumerate}
\item If $W\in A_{p/r}$ we have that
\begin{equation}\label{eq:AprLrH}\|T\vec{f}\|_{L^p(W)}\lesssim [W]_{A_{\frac{p}{r}}}^{\frac{1}{p}} [W^{-\frac{r}{p}(\frac{p}{r})'}]_{A^{sc}_{\infty,(\frac{p}{r})'}}^{\frac{1}{p}} [W]_{A^{sc}_{\infty,\frac{p}{r}}}^{\frac{1}{p'}}\|\vec{f}\|_{L^p(W)}
\end{equation}
\item If $1\leq q<p$ with $\frac{p}{q}>r$ and $W\in A_q$ then
\begin{equation}\label{eq:AqLrH}\|T \vec{f}\|_{L^{p}(W)}\lesssim \left(\frac{p}{rq}\right)' [W]_{A_{q}}^{\frac{1}{p}}[W]_{A^{sc}_{\infty,q}}^{\frac{1}{p'}}\|\vec{f}\|_{L^{p}(W)}\end{equation}
\end{enumerate}
\end{thm}

Besides the aforementioned results we will provide some endpoint estimates and some counterparts of Coifman type inequalities. We remit the reader to Section \ref{sec:WEst} for further details and proofs of these results and the ones stated in this subsection.

\subsection{ Sparse T(1)-type theorem.}
$T(1)$ theorems were established in \cite{C,DJS}
Continuing the line of research in \cite{LMA} and \cite{LO}  we address the extension of sparse $T(1)$-type results in the setting of spaces of  homogeneous type and to the vector-valued case. More presicely our contribution is a vector-valued version of  \cite[Theorem 4.1]{LO} in the setting of spaces of homogeneous type. For this purpose, we shall say that a kernel $K:X\times X \setminus\left\{ (x,x)\,:\,x\in X \right\} \rightarrow\mathbb{R}$ is in $H_r$ if it satisfies \eqref{eq:Hr1} in the case $1<r<\infty$ or \eqref{eq:Hinf1} in the case $r=\infty$. Note that $H_1\supset H_r\supset  H_s\supset H_\infty$  for $1<r<s<\infty$. We shall denote $K^*(x,y)=K(y,x)$ and $T^{*}$ the operator associated to $K^{*}$.

\begin{thm} 
\label{thm:sparseT1}
Let $(X,d,\mu)$ be a space of homogeneous type 
and $\mathcal{D}$ a dyadic system with parameters $c_{0}$, $C_{0}$
and $\delta$. Let us fix $\alpha\geq\frac{3c_{d}^{2}}{\delta}$.
Assume that $K\in H_{r}$ for some $1<r\leq\infty$
and that $K^{*}\in H_{1}$. Suppose that there exists $c>0$ such
that for every ball $B$ and every measurable subset $E\subset B$
\begin{equation}
\int_{B}|T^{*}\chi_{E}|dx\leq c\mu(B).\label{eq:HipT1}
\end{equation}
Then, for every boundedly supported function $\vec{f}$ such that $|\vec{f}|\in L^{r}(X)$,
there exists a sparse family $\mathcal{S}\subset \mathcal{D}$ such that for a.e. $x\in X$
\[
T\vec{f}(x)\in c_{n,d}c_{T}\sum_{Q\in\mathcal{S}}\langle\langle\vec{f}\rangle\rangle_{r',\alpha Q}\chi_{Q}(x).
\]
In other words we have that
\[
T\vec{f}(x)=c_{n,d}c_{T}\sum_{Q\in\mathcal{S}}\frac{1}{\mu(\alpha Q)}\int_{\alpha Q}\vec{f}(y)k_{Q}(x,y)d\mu(y)\chi_{Q}(x)
\]
where
\[
\left\Vert k_{Q}\right\Vert _{L^{(r,\infty)}\left[(Q,\frac{\mu}{\mu(Q)}),(\alpha Q,\mu)\right]}\leq1.
\]
Furthermore, there exist $0<c_{0}\leq C_{0}<\infty$,
$0<\delta<1$, $\gamma\geq1$ and $m\in\mathbb{N}$ such that there
are dyadic systems $\mathcal{D}_{1},\dots,\mathcal{D}_{m}$ with parameters
$c_{0},C_{0}$ and $\delta$ and $m$ sparse families $\mathcal{S}_{i}\subset\mathcal{D}_{i}$
such that 
\[
T\vec{f}(x)=c_{n,d}c_{T}\sum_{j=1}^{m}\sum_{Q\in\mathcal{S}_{j}}\frac{1}{\mu(Q)}\int_{Q}\vec{f}(y)k_{Q}(x,y)d\mu(y)\chi_{Q}(x)
\]
where
\[
\left\Vert k_{Q}\right\Vert _{L^{(r,\infty)}\left[(Q,\frac{\mu}{\mu(Q)}),(\alpha Q,\mu)\right]}\leq1.
\]
\end{thm}
The remainder of the paper is organized as follows. In Section \ref{sec:prelim} we gather some preliminaries on dyadic structures and matrix $A_p$ weights. Section \ref{sec:WEst} is devoted to provide weighted estimates. The proof of Theorem \ref{thm:sparseT1} is addressed in Section \ref{sec:SparseT1proof}. In Section \ref{sec:SparseProof}, Theorem \ref{thm:ThmSparse} is settled. We end this paper with Section  \ref{sec:PetOp}, which is devoted to revisit Haar shifts and the Petermichl operator.

\section{Preliminaries}\label{sec:prelim}
\subsection{Dyadic structures}\label{sec:dyadicSt}

We will borrow the presentation and the notation from \cite{L}. Fix
$0<c_{0}\leq C_{0}<\infty$ and $\delta\in(0,1)$. Suppose that for
$k\in\mathbb{Z}$ we have an index set $J_{k}$ and a pairwise disjoint
collection $\mathcal{D}=\{Q_{j}^{k}\}_{j\in J_{k}}$ of measurable
sets and a collection of points $\{z_{j}^{k}\}_{j\in J_{k}}$. We
say that $\mathcal{D}=\bigcup_{k\in\mathbb{Z}}\mathcal{D}_{k}$ is
a dyadic system with parameters $c_{0},C_{0}$ and $\delta$ if it
satisfies the following properties.
\begin{enumerate}
\item For every $k\in\mathbb{Z}$ 
\[
X=\bigcup_{j\in J_{k}}Q_{j}^{k}.
\]
\item For $k\geq l$ if $P\in\mathcal{D}_{k}$ and $Q\in\mathcal{D}_{l}$
then either $Q\cap P=\emptyset$ or $P\subseteq Q$.
\item For each $k\in\mathbb{Z}$ and $j\in J_{k}$ 
\[
B(z_{j}^{k},c_{0}\delta^{k})\subseteq Q_{j}^{k}\subseteq B(z_{j}^{k},C_{0}\delta^{k}).
\]
\end{enumerate}
We will call the elements of $\mathcal{D}$ cubes and we shall denote
\[
\mathcal{D}(Q):=\left\{ P\in\mathcal{D}:P\subseteq Q\right\} 
\]
the family of cubes of $\mathcal{D}$ contained in $Q$. We will say
that an estimate depends on $\mathcal{D}$ if it depends on the parameters
$c_{0}$, $C_{0}$ and $\delta$. 

We may regard $z_{j}^{k}$ as the ``center'' and $\delta^{k}$ as
the ``side length'' of a cube $Q_{j}^{k}\in\mathcal{D}_{k}$. These
need to be with respect a certain $k\in\mathbb{Z}$ since $k$ may
not be unique. Then a cube $Q$ also encodes the information of its
center $z$ and generation $k$. Since the structure of dyadic cubes
can be very messy in spaces of homogeneous type, those cubes do not
have a canonical definition and hence we shall define the dilations
$\alpha Q$ for $\alpha\geq1$ of $Q\in\mathcal{D}$ as 
\[
\alpha Q:=B(z_{j}^{k},\alpha C_{0}\delta^{k}).
\]
We shall abuse of this dilation notation and denote 
\[
1Q:=B(z_{j}^{k},C_{0}\delta^{k}).
\]
 Note that these dilations are actually balls.

We summarize in the following statement some properties of subsets in homogeneous spaces, particularly those of dyadic balls and cubes (see \cite{HK}).

\begin{lem}
\label{lemaResumen} Let $(X,\rho,\mu)$ be a space of homogeneous type and $\mathcal{D}$ a dyadic system with parameters $c_{0},C_{0}$ and $\delta$. 

\begin{enumerate}
\item 
If $Q_{\beta}^{k+1}\subseteq Q_{\alpha}^{k}$ and $z_{\beta}^{k+1}$, $z_{\alpha}^{k}$ are their respective centers, then $\rho(z_{\beta}^{k+1}, z_{\alpha}^{k}) < C_{0} \delta^k $
\item 

Let $c_{\mu}$ be the doubling constant and $x$ in $X$. If $\nu=log_2(c_{\mu})$ and $0<r\leq R$, then
\[
\mu(B(x,R))\leq \left(\frac{R}{r}\right)^{\nu} \mu(B(x,r)).
\]
\item 
If $x\in B(x',R)$ then, for $0<r\leq2c_{\rho}R$, we get that 
\[
\mu(B(x',R))\leq\left(\dfrac{2c_{\rho}R}{r}\right)^{log_{2}(c_{\mu})}\mu(B(x,r)).
\]
\item 
If $Q_{\beta}^{k}\subseteq Q_{\alpha}^{l}$ and $z_{\beta}^{k}$, $z_{\alpha}^{l}$ are their respective centers, then $\rho(z_{\beta}^{k}, z_{\alpha}^{l}) < C_{1} \delta^k $, where $C_{1} = 2c_{\rho}C_{0}$.

\end{enumerate}
\end{lem}

The following propoosition which is settled in \cite{HK} ensures
the existence of dyadic systems that provide a suitable replacement
for the translations of the usual dyadic systems in euclidean spaces.
\begin{prop}
\label{proposition:dyadicsystem} Let $(X,d,\mu)$ be a space of homogeneous
type. There exist $0<c_{0}\leq C_{0}<\infty$, $\gamma\geq1$, $0<\delta<1$
and $m\in\mathbb{N}$ such that there are dyadic systems $\mathcal{D}_{1},\dots,\mathcal{D}_{m}$
with parameters $c_{0}$, $C_{0}$ and $\delta$, and with the property
that for each $s\in X$ and $\rho>0$ there is a $j\in\left\{ 1,\cdots,m\right\} $
and a $Q\in\mathcal{D}_{j}$ such that 
\[
B(s,\rho)\subseteq Q,\qquad\text{and}\qquad\text{diam}(Q)\leq\gamma\rho.
\]
\end{prop}
We borrow from \cite{L} the following covering Lemma.
\begin{lem}
\label{lemma:covering} Let $(X,d,\mu)$ be a space of homogeneous
type and $\mathcal{D}$ a dyadic system with parameters $c_{0}$,
$C_{0}$ and $\delta$. Suppose that $\text{diam}(X)=\infty$, take
$\alpha\geq3c_{d}^{2}/\delta$ and let $E\subseteq X$ satisfy $0<\text{diam}(E)<\infty$.
Then there exists a partition $\mathcal{P}\subseteq\mathcal{D}$ of
$X$ such that $E\subseteq\alpha Q$ for all $Q\in\mathcal{P}$. 
\end{lem}
We end up this section recalling the definition of sparse an Carleson families and the connection between those properties.
\begin{defn} We say that $\mathcal{S}$ is a $\eta$-sparse family for $0 < \eta < 1$ if for every $Q \in \mathcal{S}$ there exists a measurable set $E_Q \subset Q$ such that $\mu(E_Q) \geq \eta \mu(Q)$ and the sets $\{E_Q\}_{Q \in \mathcal{S}}$ are pairwise disjoint.
\end{defn}

\begin{defn} We say that $\mathcal{S}$ is $\Lambda$-Carleson for $\Lambda > 1$, if for every cube $Q \in \mathcal{S}$,
\[
\sum_{\substack{P \in \mathcal{S} \\ P \subset Q}} \mu(P) \leq \Lambda \mu(Q).
\]
\end{defn}

We borrow from \cite{DGKLWY} the following equivalence.

\begin{prop}
		Given \(0 < \eta < 1\) and a collection \(\mathcal{S} \subset \mathcal{D}\) of dyadic cubes, the following statements hold.
		\begin{itemize}
			\item If \(\mathcal{S}\) is \(\eta\)-sparse, then \(\mathcal{S}\) is \(\frac{1}{\eta}\)-Carleson.
			\item If \(\mathcal{S}\) is \(\frac{1}{\eta}\)-Carleson, then \(\mathcal{S}\) is \(\eta\)-sparse.
		\end{itemize}	
\end{prop}

\subsection{The $A_{p}$ condition}
\subsubsection{Reducing matrices}
 If we fix a ball $B \subseteq X$ with $(X,d,\mu)$ a space of homogeneous type 
and a matrix weight $W$, for any $1\leq p<\infty$, 
\[
\rho(\vec{e})=\left(\frac{1}{|B|}\int_{B}|W^{\frac{1}{p}}(x)\vec{e}|^{p}d\mu(x)\right)^{\frac{1}{p}}
\]
defines a norm over $\mathbb{C}^{n}$. Hence following \cite[Section 1]{V}, we have that there exists a positive selfadjoint operator $\mathcal{W}_{B,p}: \mathbb{C}^{n} \longrightarrow \mathbb{C}^{n}$  such that 
\[
\{\vec{e}\,:\,|\mathcal{W}_{B,p}\vec{e}|\leq1\}\subset\{\vec{e}\,:\,\rho(\vec{e})\leq1\}\subset C\{\vec{e}\,:\,|\mathcal{W}_{B,p}\vec{e}|\leq1\}
\]
where the set $\{\vec{e}\,:\,|\mathcal{W}_{B,p}\vec{e}|\leq1\}$ is the largest elipsoid contained in the unit ball $\{\vec{e}\,:\,\rho(\vec{e})\leq1\}$ and $C$ is a constant that depend on the dimension $n$.
Therefore
\[
\rho(\vec{e})\simeq|\mathcal{W}_{B,p}\vec{e}|.
\]
Analogously, in the case ${p}>1$, since 
\[\rho^{*}(\vec{e}) = 
\left(\frac{1}{|B|}\int_{B}|W^{-\frac{1}{p}}(x)\vec{e}|^{p'}d\mu(x)\right)^{\frac{1}{p'}},
\]
defines a norm over $\mathbb{C}^n$, there exists some matrix $\mathcal{W}'_{B,p}$ which is a positive
selfadjoint operator as well such that 
\begin{equation}\label{modulo rho estrella}
\rho^{*}(\vec{e})\simeq|\mathcal{W}'_{B,p}\vec{e}|.
\end{equation}
Note that these matrices are particularly useful when dealing with
$A_{p}$ weights since, via trace norm, they allow to rewrite the
conditions in a way that turns out to be quite useful for applications. In fact, since for positive definite self-adjoint  matrices, $A$ and $B$, 
\begin{equation}\label{conmutar}
\left|AB\right|_{op} = \left|BA\right|_{op},
\end{equation} then  
 if $\{\vec{e_{j}}: j=1,...,n\}$ is an ortonormal basis in $\mathbb{C}^{n}$,
\begin{align}\label{cte Ap via matrix}
[W]_{A_{p}} & =\sup_{B}\frac{1}{\mu(B)}\int_{B}\left(\frac{1}{\mu(B)}\int_{B}|W^{-\frac{1}{p}}(y)W^{\frac{1}{p}}(x)|_{op}^{p'}d\mu(y)\right)^{\frac{p}{p'}}d\mu(x)\nonumber \\
& \simeq\sup_{B}\frac{1}{\mu(B)}\int_{B}\left(\frac{1}{\mu(B)}\int_{B}\left(\sum_{j=1}^{n}|W^{-\frac{1}{p}}(y)W^{\frac{1}{p}}(x)\vec{e}_{j}|\right)^{p'}d\mu(y)\right)^{\frac{p}{p'}}d\mu(x)\nonumber \\
& \simeq\sup_{B}\frac{1}{\mu(B)}\int_{B}\left(\sum_{j=1}^{n}|\mathcal{W}'_{B,p}W^{\frac{1}{p}}(x)\vec{e}_{j}|\right)^{p}d\mu(x)\simeq\sup_{B}\frac{1}{\mu(B)}\int_{B}\left(|\mathcal{W}'_{B,p}W^{\frac{1}{p}}(x)|_{op}\right)^{p}d\mu(x)\nonumber \\
& =\sup_{B}\frac{1}{\mu(B)}\int_{B}\left(|W^{\frac{1}{p}}(x)\mathcal{W}'_{B,p}|_{op}\right)^{p}d\mu(x)\simeq\sup_{B}\frac{1}{\mu(B)}\int_{B}\left(\sum_{j=1}^{n}|W^{\frac{1}{p}}(x)\mathcal{W}'_{B,p}\vec{e}_{j}|\right)^{p}d\mu(x)\nonumber \\
& \simeq\sup_{B}\sum_{j=1}^{n}\left(\frac{1}{\mu(B)}\int_{B}\left(|W^{\frac{1}{p}}(x)\mathcal{W}'_{B,p}\vec{e}_{j}|\right)^{p}d\mu(x)\right)^{\frac{1}{p}p}\simeq\sup_{B}\sum_{j=1}^{n}\left(|\mathcal{W}{}_{B,p}\mathcal{W}'_{B,p}\vec{e}_{j}|\right)^{p}\nonumber \\
& \simeq \sup_{B}|\mathcal{W}{}_{B,p}\mathcal{W}'_{B,p}|^p_{op}.
\end{align}
Analogously
\[
[W]_{A_{1}}\simeq\sup_{B}\esssup_{y\in B}|\mathcal{W}{}_{B,1}W^{-1}(y)|_{op}<\infty.
\]
\subsubsection{Reverse H\"older inequality}
As it was shown in \cite{B}, matrix $A_p$ weights do not satisfy a reverse H\"older inequality. However, as we see in this section, we can still expoit in a suitable weight the scalar reverse H\"older inequality. For that purpose, we begin settling the following result
\begin{lem}
If $W\in A_{p}$ for $p\geq1$ then for every $\vec{e}\not=0$ we have that 
$
|W^{\frac{1}{p}}(\cdot)\vec{e}|^{p}
$
 is a scalar $A_{p}$ weight and 
\[
[|W^{\frac{1}{p}}(\cdot)\vec{e}|^{p}]_{A_{p}}\lesssim[W]_{A_{p}}.
\]
\end{lem}
\allowdisplaybreaks
\begin{proof}
First we deal with the case $p>1$. We fix $\vec{e}\not=$0 and argue
as follows. For any $\vec{v}\not=0$ we have that
\begin{align*}
&\left(\frac{1}{\mu(B)}\int_{B}|W^{\frac{1}{p}}(x)\vec{e}|^{p}d\mu(x)\right)^{\frac{1}{p}}\left(\frac{1}{\mu(B)}\int_{B}|W^{\frac{1}{p}}(y)\vec{e}|^{-p'}dy\right)^{\frac{1}{p'}}\\
& =\left(\frac{1}{\mu(B)}\int_{B}|W^{\frac{1}{p}}(x)\vec{e}|^{p}d\mu(x)\right)^{\frac{1}{p}}\left(\frac{1}{\mu(B)}\int_{B}\left(\sup_{\vec{w}\not=0}\frac{|\langle\vec{e},\vec{w}\rangle|}{|W^{-\frac{1}{p}}(y)\vec{w}|}\right)^{-p'}dy\right)^{\frac{1}{p'}}\\
& \leq\left(\frac{1}{\mu(B)}\int_{B}|W^{\frac{1}{p}}(x)\vec{e}|^{p}d\mu(x)\right)^{\frac{1}{p}}\left(\frac{1}{\mu(B)}\int_{B}\left(\frac{|\langle\vec{e},\vec{v}\rangle|}{|W^{-\frac{1}{p}}(y)\vec{v}|}\right)^{-p'}dy\right)^{\frac{1}{p'}}\\
& \leq\left(\frac{1}{\mu(B)}\int_{B}|W^{\frac{1}{p}}(x)\vec{e}|^{p}d\mu(x)\right)^{\frac{1}{p}}\left(\frac{1}{\mu(B)}\int_{B}\left(\frac{|W^{-\frac{1}{p}}(y)\vec{v}|}{|\langle\vec{e},\vec{v}\rangle|}\right)^{p'}dy\right)^{\frac{1}{p'}}\\
& =\frac{\left(\frac{1}{\mu(B)}\int_{B}|W^{\frac{1}{p}}(x)\vec{e}|^{p}d\mu(x)\right)^{\frac{1}{p}}}{|\langle\vec{e},\vec{v}\rangle|}\left(\frac{1}{\mu(B)}\int_{B}|W^{-\frac{1}{p}}(y)\vec{v}|^{p'}dy\right)^{\frac{1}{p'}}
\end{align*}
Let $\vec{v}=\mathcal{W}_{p,B}\vec{u}$ with $\vec{u}$ to be chosen.
Then
\begin{align*}
 	& \left(\frac{1}{\mu(B)}\int_{B}|W^{\frac{1}{p}}(x)\vec{e}|^{p}d\mu(x)\right)^{\frac{1}{p}}\left(\frac{1}{\mu(B)}\int_{B}\left(\frac{|W^{-\frac{1}{p}}(y)\vec{v}|}{|\langle\vec{e},\vec{v}\rangle|}\right)^{p'}dy\right)^{\frac{1}{p'}}\\
 & \simeq\frac{|\mathcal{W}_{p,B}\vec{e}|}{|\langle\vec{e},\mathcal{W}_{p,B}\vec{u}\rangle|}\left(\frac{1}{\mu(B)}\int_{B}|W^{-\frac{1}{p}}(y)\mathcal{W}_{p,B}\vec{u}|^{p'}dy\right)^{\frac{1}{p'}}\\
 & \simeq\frac{|\mathcal{W}_{p,B}\vec{e}||\mathcal{W}'_{p,B}\mathcal{W}{}_{p,B}\vec{u}|}{|\langle\vec{e},\mathcal{W}_{p,B}\vec{u}\rangle|} \leq\frac{|\mathcal{W}_{p,B}\vec{e}||\vec{u}|}{|\langle\vec{e},\mathcal{W}_{p,B}\vec{u}\rangle|}|\mathcal{W}'_{p,B}\mathcal{W}{}_{p,B}|_{op}\\
 & \leq\frac{|\mathcal{W}_{p,B}\vec{e}||\vec{u}|}{|\langle\mathcal{W}_{p,B}\vec{e},\vec{u}\rangle|}[W]_{A_{p}}^{\frac{1}{p}}.
\end{align*}
Then, choosing $\vec{u}=\mathcal{W}_{p,B}\vec{e}$  we are
done.

For $p=1$ observe that for almost every $y\in B$, 
\begin{align*}
\frac{1}{\mu(B)}\int|W(x)\vec{e}||W(y)\vec{e}|^{-1}d\mu(x) & =\frac{1}{\mu(B)}\int_{B}|W(x)W^{-1}(y)W(y)\vec{e}||W(y)\vec{e}|^{-1}d\mu(x)\\
& \leq\frac{1}{\mu(B)}\int_{B}|W(x)W^{-1}(y)|_{op}|W(y)\vec{e}||W(y)\vec{e}|^{-1}d\mu(x)\\
& =\frac{1}{\mu(B)}\int_{B}|W(x)W^{-1}(y)|_{op}d\mu(x)
\end{align*}
and hence
\[
[|W(\cdot)\vec{e}|]_{A_{1}}\leq[W]_{A_{1}}.
\]
\end{proof}
Recall that in the scalar setting, $w\in A_{\infty}$ if 
\[
[w]_{A_{\infty}}=\sup_{B}\frac{1}{w(B)}\int_{B}M(w\chi_{B})d\mu<\infty
\]
and that if $w\in A_p$ then $[w]_{A_\infty}\lesssim [w]_{A_p}$.
The following sharp weak reverse
H\"older inequality that was settled in \cite[Theorem 1.1]{HPR} holds for $w\in A_\infty$.
\[
\left(\frac{1}{\mu(B)}\int_{B}w^{r(w)}(x)d\mu(x)\right)^{\frac{1}{r(w)}}\leq2(4c_{d})^{D_{\mu}}\frac{1}{\mu(2c_{d}B)}\int_{2c_{d}B}wd\mu
\]
where $D_\mu=\log_2(c_\mu)$, $B$ is any ball and 
\[
1<r(w)\leq 1+\frac{1}{\tau_{c_{d}\mu}[w]_{A_{\infty}}}=1+\frac{1}{6(32c_{d}^{2}(4c_{d}^{2}+c_{d})^{2})^{D_{\mu}}[w]_{A_{\infty}}}.
\]
Following \cite{NPTV}, we define
\[
[W]_{A_{\infty,p}^{sc}}=\sup_{\vec{e}\not=0}[|W^{\frac{1}{p}}(x)\vec{e}|^{p}]_{A_{\infty}}.
\]
Our next lemma is a version in spaces of homogeneous type of a result in \cite{MRR}.
\begin{lem}\label{lem:RevHolderDouble}
Let $1\leq p<\infty$ and $W\in A_{\infty,p}^{sc}$. If $1\leq r\leq1+\frac{1}{\tau_{c_{d}\mu}[W]_{A_{\infty,p}^{sc}}}$
and $A$ is a self-adjoint positive definite matrix then 
\[
\left(\frac{1}{\mu(B)}\int_{B}|W^{\frac{1}{p}}(x)A|_{op}^{pr}d\mu(x)\right)^{\frac{1}{r}}\lesssim\frac{1}{\mu(2c_{d}B)}\int_{2c_{d}B}|W^{\frac{1}{p}}(x)A|_{op}^{p}d\mu(x)
\]
\end{lem}
\begin{proof}
Fix some orthonormal basis $\{\vec{e_{j}}\}$ on $\mathbb{R}^{n}$.
Since for $r$ in the hypothesis $|W^{\frac{1}{p}}\vec{e}|^{p}$ satisfies
the reverse H\"older inequality uniformly on $\vec{e}$ we have that
\begin{align*}
\left(\frac{1}{\mu(B)}\int_{B}|W^{\frac{1}{p}}(x)A|_{op}^{pr}d\mu(x)\right)^{\frac{1}{r}} & \lesssim\sum_{j=1}^{n}\left(\frac{1}{\mu(2c_{d}B)}\int_{2c_{d}B}|W^{\frac{1}{p}}(x)A\vec{e}_{j}|^{pr}d\mu(x)\right)^{\frac{1}{r}}\\
 & \lesssim\sum_{j=1}^{n}\frac{1}{\mu(2c_{d}B)}\int_{2c_{d}B}|W^{\frac{1}{p}}(x)A\vec{e}_{j}|^{p}d\mu(x)\\
 & \lesssim\frac{1}{\mu(2c_{d}B)}\int_{2c_{d}B}|W^{\frac{1}{p}}(x)A|_{op}^{p}d\mu(x)
\end{align*}
and we are done.
\end{proof}
We end this section with a result based on \cite[Lemma 2.1]{B}.
\begin{lem}\label{lem:alphaFuera}
If $A$ and $B$ are positive self-adjoint matrices and $0<\alpha<1$ then
\[
|A^{\alpha}B^{\alpha}|_{op}\lesssim|AB|_{op}^{\alpha}.
\]
\end{lem}

%

\section{Weighted estimates}\label{sec:WEst}
%
%
%

\subsection{Proof of Theorem \ref{thm:StrongMax}}
We will follow the approach in \cite{IPRR}. The novelty here, besides the extension to spaces of homogeneous type, is that, essentially the same argument provided there works as well for $W\in A_p$.

We define
\[
M_{p,W}\vec{f}(x)=\sup_{B\ni x}\dashint_B |W^\frac{1}{p}(x)W^{-\frac{1}{p}}(y)\vec{f}(y)|d\mu(y)
\]
where $\mathcal{D}$ is any dyadic system.
First we observe that by Proposition \ref{proposition:dyadicsystem}
\[
M_{p,W}\vec{f}\leq\sum_{t=1}^{m}M_{p,W}^{\mathcal{D}^{t}}\vec{f}.
\]
Note that this reduces the problem to deal with each $M_{p,W}^{\mathcal{D}^{t}}\vec{f}$,
namely, dropping the superindex $t$ and the subindex $p$ in the notation it suffices to show that
\[
\|M_{W}^{\mathcal{D}}\vec{f}\|_{L^{p}}\lesssim \kappa \|\vec{f}\|_{L^{p}}.
\]
where \[\kappa=\begin{cases}[W]_{A_{p}}^{\frac{1}{p}}[W^{-\frac{1}{p-1}}]_{A_{\infty,p'}^{sc}}^{\frac{1}{p}} &\text{if }W\in A_p \\ [W]_{A_{q}}^{\frac{1}{p}}, & \text{if }W\in A_q \quad 1< q<p. \end{cases}\]
We argue as follows. Let 
\[
M_{W,k}^{\mathcal{D}}\vec{f}(x)=\sup_{Q\ni x,Q\in\mathcal{D}_{j},j\leq k}\frac{1}{\mu(Q)}\int_{Q}\left|W^{\frac{1}{p}}(x)W^{-\frac{1}{p}}(y)\vec{f}(y)\right|d\mu(y)
\]
Note that then by dominated convergence theorem
\[
\|M_{W}^{\mathcal{D}}\vec{f}\|_{L^{p}}=\lim_{k\rightarrow\infty}\|M_{W,k}^{\mathcal{D}}\vec{f}\|_{L^{p}}
\]
and then it suffices to deal with the term in the right hand side.
We begin observing that 
\[
M_{W,k}^{\mathcal{D}}\vec{f}(x)=\sum_{J\in\mathcal{D}_{k}}M_{J,W}\vec{f}(x)\chi_{J}(x)
\]
where 
\[
M_{J,W}\vec{f}(x)=\sup_{Q\ni x,Q\in\mathcal{D}(J)}\frac{1}{\mu(Q)}\int_{Q}\left|W^{\frac{1}{p}}(x)W^{-\frac{1}{p}}(y)\vec{f}(y)\right|d\mu(y).
\]
Then we have that
\begin{align*}
\int_{X}M_{W,k}^{\mathcal{D}}\vec{f}(x)^{p}d\mu(x) & =\int_{X}\sum_{J\in\mathcal{D}_{k}}M_{J,W}\vec{f}(x)^{p}\chi_{J}(x)d\mu(x)\\
 & =\sum_{J\in\mathcal{D}_{k}}\int_{J}M_{J,W}\vec{f}(x)^{p}d\mu(x).
\end{align*}
And it suffices to deal with the latter. We shall begin providing
a pointwise domination for $M_{J,W}\vec{f}(x)$. First we introduce
a stopping time condition. Let $\mathcal{J}(J)$ the maximal cubes
$L\in\mathcal{D}(J)$ where 
\begin{equation}
\frac{1}{\mu(L)}\int_{L}\left|\mathcal{W}_{J,p}W^{-\frac{1}{p}}\vec{f}\right|d\mu>4\frac{1}{\mu(J)}\int_{J}\left|\mathcal{W}_{J,p}W^{-\frac{1}{p}}\vec{f}\right|d\mu\label{eq:StoppingTime}.
\end{equation}
By maximality
\begin{align*}
\sum_{L\in\mathcal{J}(J)}\mu(L) & \leq\frac{1}{4}\frac{1}{\frac{1}{\mu(L)}\int_{J}\left|\mathcal{W}_{J,p}W^{-\frac{1}{p}}\vec{f}\right|}\sum_{L\in\mathcal{J}(J)}\int_{L}\left|\mathcal{W}_{J,p}W^{-\frac{1}{p}}\vec{f}\right|d\mu\\
 & \leq\frac{1}{4}\frac{1}{\frac{1}{\mu(L)}\int_{J}\left|\mathcal{W}_{J,p}W^{-\frac{1}{p}}\vec{f}\right|}\int_{J}\left|\mathcal{W}_{J,p}W^{-\frac{1}{p}}\vec{f}\right|d\mu\\
 & \leq \frac{\mu(J)}{4}.
\end{align*}
Now let $\mathcal{F}(J)$ be the collection of cubes in $Q \in \mathcal{D}(J):$
that are not a subset of any cube $I\in\mathcal{J}(J)$. We shall
call $\cup\mathcal{J}(J)=\cup_{L\in\mathcal{J}(J)}L$. 

We provide the aforementioned pointwise domination for $M_{J,W}\vec{f}(x)$
studying two cases.
\begin{itemize}
\item Case 1: For $Q\not\in\mathcal{F}(J)$ we have that $Q\subseteq L$
for some $L\in\mathcal{J}(J)$ and hence, if $x\in Q$ then $x\in\cup\mathcal{J}(J)$.
Hence 
\[
\frac{1}{\mu(Q)}\int_{Q}\left|W^{\frac{1}{p}}(x)W^{-\frac{1}{p}}(y)\vec{f}(y)\right|d\mu(y)\leq\sup_{L\in\mathcal{J}(J)}M_{L,W}\vec{f}(x).
\]	
	
\item Case 2: Let $Q\in\mathcal{F}(J)$.
Then, by the maximality of $I\in \mathcal{J}(J)$ we have that 

\begin{align*}
 & \frac{1}{\mu(Q)}\int_{Q}\left|W^{\frac{1}{p}}(x)W^{-\frac{1}{p}}(y)\vec{f}(y)\right|d\mu(y)\\
\leq & \left|W^{\frac{1}{p}}(x)\mathcal{W}_{J,p}^{-1}\right|_{op} \frac{1}{\mu(Q)}\int_{Q}\left|\mathcal{W}_{J,p}W^{-\frac{1}{p}}(y)\vec{f}(y)\right|d\mu(y)\\
\leq & 4\left|W^{\frac{1}{p}}(x)\mathcal{W}_{J,p}^{-1}\right|_{op}\frac{1}{\mu(J)}\int_{J}\left|\mathcal{W}_{J,p}W^{-\frac{1}{p}}(y)\vec{f}(y)\right|d\mu(y).
\end{align*}
From this point we consider different cases.  
\begin{itemize}
\item If $W\in A_p$, choosing 
\[
r=1+\frac{1}{\tau_{c_{d}\mu}     \left[W^{-\frac{1}{p-1}}\right]_{A_{\infty,p}^{sc}}}
\]
 using H\"older inequality, and taking into a account the doubling property  of the measure $\mu$, Lemma 3, \eqref{conmutar}, \eqref{modulo rho estrella}, and \eqref{cte Ap via matrix} we have that
	\begin{align*}
	& \left|W^{\frac{1}{p}}(x)\mathcal{W}_{J,p}^{-1}\right|_{op}\frac{1}{\mu(J)}\int_{J}\left|\mathcal{W}_{J,p}W^{-\frac{1}{p}}(y)\vec{f}(y)\right|d\mu(y)\\
	\leq & \left|W^{\frac{1}{p}}(x)\mathcal{W}_{J,p}^{-1}\right|_{op}\left(\frac{1}{\mu(J)}\int_{J}\left|\mathcal{W}_{J,p}W^{-\frac{1}{p}}(y)\right|^{rp'}_{ op}d\mu(y)\right)^{\frac{1}{rp'}}\left(\frac{1}{\mu(J)}\int_{J}\left|\vec{f}(y)\right|^{(rp')'}d\mu(y)\right)^{\frac{1}{(rp')'}}\\			
	\leq &\left|W^{\frac{1}{p}}(x)\mathcal{W}_{J,p}^{-1}\right|_{op} \left(\frac{1}{\mu(J)}\int_{1J}\left|\mathcal{W}_{J,p}W^{-\frac{1}{p}}(y)\right|^{rp'}_{ op}d\mu(y)\right)^{\frac{1}{rp'}}\left(\frac{1}{\mu(J)}\int_{J}\left|\vec{f}(y)\right|^{(rp')'}d\mu(y)\right)^{\frac{1}{(rp')'}}\\
	=&\left|W^{\frac{1}{p}}(x)\mathcal{W}_{J,p}^{-1}\right|_{op}\left(\frac{\mu(1J)}{\mu(J)}\frac{1}{\mu(1J)}\int_{1J}\left|\mathcal{W}_{J,p}W^{-\frac{1}{p}}(y)\right|^{rp'}_{ op}d\mu(y)\right)^{\frac{1}{rp'}}\left(\frac{1}{\mu(J)}\int_{J}\left|\vec{f}(y)\right|^{(rp')'}d\mu(y)\right)^{\frac{1}{(rp')'}}\\			
	\leq&\left|W^{\frac{1}{p}}(x)\mathcal{W}_{J,p}^{-1}\right|_{op}     c_{\mu}^{\log_2\frac{C_0}{c_0}} \left(\frac{1}{\mu(1J)}\int_{1J} \left|\mathcal{W}_{J,p}W^{-\frac{1}{p}}(y)\right|^{rp'}_{op} d\mu(y)\right)^{\frac{1}{rp'}} \left(\frac{1}{\mu(J)}\int_{J}\left|\vec{f}(y)\right|^{(rp')'}d\mu(y)\right)^{\frac{1}{(rp')'}}\\			
	\lesssim & \left|W^{\frac{1}{p}}(x)\mathcal{W}_{J,p}^{-1}\right|_{op}\left(\frac{1}{\mu(1J)}\int_{1J}\left|\mathcal{W}_{J,p}W^{-\frac{1}{p}}(y)\right|^{rp'}_{ op}d\mu(y)\right)^{\frac{1}{rp'}}\left(\frac{1}{\mu(J)}\int_{J}\left|\vec{f}(y)\right|^{(rp')'}d\mu(y)\right)^{\frac{1}{(rp')'}}\\
	\lesssim & \left|W^{\frac{1}{p}}(x)\mathcal{W}_{J,p}^{-1}\right|_{op}\left(\frac{1}{\mu(2c_{d}J)}\int_{2c_{d}J}\left|\mathcal{W}_{J,p}W^{-\frac{1}{p}}(y)\right|^{p'}_{ op}d\mu(y)\right)^{\frac{1}{p'}}\left(\frac{1}{\mu(J)}\int_{J}\left|\vec{f}(y)\right|^{(rp')'}d\mu(y)\right)^{\frac{1}{(rp')'}}\\
	\simeq & \left|W^{\frac{1}{p}}(x)\mathcal{W}_{J,p}^{-1}\right|_{op}\left|\mathcal{W}_{J,p}\mathcal{W}'_{2c_{d}J,p}\right|_{op}\left(\frac{1}{\mu(J)}\int_{J}\left|\vec{f}(y)\right|^{(rp')'}d\mu(y)\right)^{\frac{1}{(rp')'}}\\
	\simeq & \left|W^{\frac{1}{p}}(x)\mathcal{W}_{J,p}^{-1}\right|_{op}\left(\frac{1}{\mu(J)}\int_{J}\left|W^{\frac{1}{p}}(y)\mathcal{W}'_{2c_{d}J,p}\right|^{p}_{ op}d\mu(y)\right)^{\frac{1}{p}}\left(\frac{1}{\mu(J)}\int_{J}\left|\vec{f}(y)\right|^{(rp')'}d\mu(y)\right)^{\frac{1}{(rp')'}}\\
	\leq & \left|W^{\frac{1}{p}}(x)\mathcal{W}_{J,p}^{-1}\right|_{op}\left(\frac{\mu(2c_{d}J)}{\mu(J)}\frac{1}{\mu(2c_{d}J)}\int_{2c_{d}J}\left|W^{\frac{1}{p}}(y)\mathcal{W}'_{2c_{d}J,p}\right|^{p}_{ op}d\mu(y)\right)^{\frac{1}{p}}\left(\frac{1}{\mu(J)}\int_{J}\left|\vec{f}(y)\right|^{(rp')'}d\mu(y)\right)^{\frac{1}{(rp')'}}\\
	\lesssim & \left|W^{\frac{1}{p}}(x)\mathcal{W}_{J,p}^{-1}\right|_{op}\left(\frac{1}{\mu(2c_{d}J)}\int_{2c_{d}J}\left|W^{\frac{1}{p}}(y)\mathcal{W}'_{2c_{d}J,p}\right|^{p}_{ op}d\mu(y)\right)^{\frac{1}{p}}\left(\frac{1}{\mu(J)}\int_{J}\left|\vec{f}(y)\right|^{(rp')'}d\mu(y)\right)^{\frac{1}{(rp')'}}\\
	\simeq & \left|W^{\frac{1}{p}}(x)\mathcal{W}_{J,p}^{-1}\right|_{op}\left|\mathcal{W}_{2c_{d}J,p}\mathcal{W}'_{2c_{d}J,p}\right|_{ op}\left(\frac{1}{\mu(J)}\int_{J}\left|\vec{f}(y)\right|^{(rp')'}d\mu(y)\right)^{\frac{1}{(rp')'}}\\
	\lesssim & [W]_{A_{p}}^{\frac{1}{p}}\left|W^{\frac{1}{p}}(x)\mathcal{W}_{J,p}^{-1}\right|_{op}\left(\frac{1}{\mu(J)}\int_{J}\left|\vec{f}(y)\right|^{(rp')'}d\mu(y)\right)^{\frac{1}{(rp')'}}. 
\end{align*}

\item If $W\in A_q$ for $1<q<p$, taking into account Lemma \ref{lem:alphaFuera}, using H\"older inequality, Jensen inequality and  \eqref{cte Ap via matrix} we have that
\begin{align*}
	&4\left|W^{\frac{1}{p}}(x)\mathcal{W}_{J,p}^{-1}\right|_{op}\frac{1}{\mu(J)}\int_{J}\left|\mathcal{W}_{J,p}W^{-\frac{1}{p}}(y)\vec{f}(y)\right|d\mu(y)\\
	& \leq 4\left|W^{\frac{1}{p}}(x)\mathcal{W}_{J,p}^{-1}\right|_{op}\frac{1}{\mu(J)}\int_{J}\left(\frac{1}{\mu(J)}\int_{J}\left|W^{\frac{1}{q}\frac{q}{p}}(x)W^{-\frac{1}{q}\frac{q}{p}}(y)\right|_{op}^{p}dx\right)^{\frac{1}{p}}|\vec{f}(y)|d\mu(y)\\
	& \lesssim\left|W^{\frac{1}{p}}(x)\mathcal{W}_{J,p}^{-1}\right|_{op}\frac{1}{\mu(J)}\int_{J}\left(\frac{1}{\mu(J)}\int_{J}\left|W^{\frac{1}{q}}(x)W^{-\frac{1}{q}}(y)\right|_{op}^{q}dx\right)^{\frac{1}{p}}|\vec{f}(y)|d\mu(y)\\
	& \lesssim\left|W^{\frac{1}{p}}(x)\mathcal{W}_{J,p}^{-1}\right|_{op}\frac{1}{\mu(J)}\int_{J}\left|\mathcal{W}_{J,q}W^{-\frac{1}{q}}(y)\right|_{op}^{\frac{q}{p}}|\vec{f}(y)|d\mu(y)\\
	& \lesssim\left|W^{\frac{1}{p}}(x)\mathcal{W}_{J,p}^{-1}\right|_{op}\left(\frac{1}{\mu(J)}\int_{J}\left|\mathcal{W}_{J,q}W^{-\frac{1}{q}}(y)\right|_{op}^{\frac{q'q}{p}}d\mu(y)\right)^{\frac{1}{q'}}\left(\frac{1}{\mu(J)}\int_{J}|\vec{f}(y)|^{q}d\mu(y)\right)^{\frac{1}{q}}\\
	& \lesssim\left|W^{\frac{1}{p}}(x)\mathcal{W}_{J,p}^{-1}\right|_{op}\left(\frac{1}{\mu(J)}\int_{J}\left|\mathcal{W}_{J,q}W^{-\frac{1}{q}}(y)\right|_{op}^{\frac{q'q}{p}}d\mu(y)\right)^{\frac{pq}{q'pq}}\left(\frac{1}{\mu(J)}\int_{J}|\vec{f}(y)|^{q}d\mu(y)\right)^{\frac{1}{q}}\\
	&\leq\left|W^{\frac{1}{p}}(x)\mathcal{W}_{J,p}^{-1}\right|_{op}\left(\frac{1}{\mu(J)}\int_{J}\left|\mathcal{W}_{J,q}W^{-\frac{1}{q}}(y)\right|_{op}^{\frac{q'q}{p}\frac{p}{q}}d\mu(y)\right)^{\frac{q}{q'p}}\left(\frac{1}{\mu(J)}\int_{J}|\vec{f}(y)|^{q}d\mu(y)\right)^{\frac{1}{q}}\\	
	& =\left|W^{\frac{1}{p}}(x)\mathcal{W}_{J,p}^{-1}\right|_{op}\left(\frac{1}{\mu(J)}\int_{J}\left|\mathcal{W}_{J,q}W^{-\frac{1}{q}}(y)\right|_{op}^{q'}d\mu(y)\right)^{\frac{1}{q'}\frac{q}{p}}\left(\frac{1}{\mu(J)}\int_{J}|\vec{f}(y)|^{q}d\mu(y)\right)^{\frac{1}{q}}\\
	& \lesssim\left|W^{\frac{1}{p}}(x)\mathcal{W}_{J,p}^{-1}\right|_{op}\left|\mathcal{W}'_{J,q}\mathcal{W}_{J,q}\right|_{op}^{\frac{q}{p}}\left(\frac{1}{\mu(J)}\int_{J}|\vec{f}(y)|^{q}d\mu(y)\right)^{\frac{1}{q}}\\
	& \lesssim\left|W^{\frac{1}{p}}(x)\mathcal{W}_{J,p}^{-1}\right|_{op}[W]_{A_{q}}^{\frac{1}{p}}\left(\frac{1}{\mu(J)}\int_{J}|\vec{f}(y)|^{q}d\mu(y)\right)^{\frac{1}{q}}.
\end{align*}

\item If $W\in A_1$, from  Lemma \ref{lem:alphaFuera} and the definition of $[W]_{A_1}$ constant we get

\begin{align*}
	& 4\left|W^{\frac{1}{p}}(x)\mathcal{W}_{J,p}^{-1}\right|_{op}\frac{1}{\mu(J)}\int_{J}\left|\mathcal{W}_{J,p}W^{-\frac{1}{p}}(y)\vec{f}(y)\right|d\mu(y)\\
	\lesssim & \left|W^{\frac{1}{p}}(x)\mathcal{W}_{J,p}^{-1}\right|_{op}\frac{1}{\mu(J)}\int_{J}\left(\frac{1}{\mu(J)}\int_{J}\left|W^{\frac{1}{p}}(x)W^{-\frac{1}{p}}(y)\right|_{op}^{p}dx\right)^{\frac{1}{p}}|\vec{f}(y)|d\mu(y)\\
	\lesssim & \left|W^{\frac{1}{p}}(x)\mathcal{W}_{J,p}^{-1}\right|_{op}\frac{1}{\mu(J)}\int_{J}\left(\frac{1}{\mu(J)}\int_{J}\left|W(x)W^{-1}(y)\right|_{op}dx\right)^{\frac{1}{p}}|\vec{f}(y)|d\mu(y)\\
	\lesssim & \left|W^{\frac{1}{p}}(x)\mathcal{W}_{J,p}^{-1}\right|_{op}[W]_{A_{1}}^{\frac{1}{p}}\frac{1}{\mu(J)}\int_{J}|\vec{f}(y)|d\mu(y).
\end{align*}
\end{itemize}
Summarizing

\begin{align*}
	&\frac{1}{\mu(Q)}\int_{Q}\left|W^{\frac{1}{p}}(x)W^{-\frac{1}{p}}(y)\vec{f}(y)\right|d\mu(y)\\
	\leq& 4\left|W^{\frac{1}{p}}(x)\mathcal{W}_{J,p}^{-1}\right|_{op}\frac{1}{\mu(J)}\int_{J}\left|\mathcal{W}_{J,p}W^{-\frac{1}{p}}(y)\vec{f}(y)\right|d\mu(y)\\
	\lesssim &\left|W^{\frac{1}{p}}(x)\mathcal{W}_{J,p}^{-1}\right|_{op}\begin{cases} 
		[W]_{A_{p}}^{\frac{1}{p}}\left(\frac{1}{\mu(J)}\int_{J}\left|\vec{f}(y)\right|^{(rp')'}d\mu(y)\right)^{\frac{1}{(rp')'}}
		&\text{if }  W\in A_p,\\
		[W]_{A_{q}}^{\frac{1}{p}}\left(\frac{1}{\mu(J)}\int_{J}|\vec{f}(y)|^{q}d\mu(y)\right)^{\frac{1}{q}} & \text{if } W\in A_q \quad 1\leq q<p.\end{cases}
\end{align*}

\end{itemize}
Gathering the estimates above we have that 

\[
M_{J,W}\vec{f}(x)\leq\max\left\{ c_{n,p}[W]_{A_{s}}^{\frac{1}{p}}\left|W^{\frac{1}{p}}(x)\mathcal{W}_{J,p}^{-1}\right|_{op}\left(\dashint_{J}\left|\vec{f}(y)\right|^{t}d\mu(y)\right)^{\frac{1}{t}},\chi_{\cup\mathcal{J}(J)}(x)\sup_{L\in\mathcal{J}(J)}M_{L,W}\vec{f}(x)\right\} 
\]
where $s=p$ and $t=(rp')'$ if $W\in A_p$, and $s=q=t$ if $W\in A_q$, with $1\leq q<p$. Hence

\begin{align*}
	& \int_{J}\left(M_{J,W}\vec{f}\right)^{p}d\mu\\	
	\leq & c_{n,p}[W]_{A_{s}}\frac{1}{\mu(J)}\int_{J}\left|W^{\frac{1}{p}}(x)\mathcal{W}_{J,p}^{-1}\right|_{op}^{p}\left(\frac{1}{\mu(J)}\int_{J}\left|\vec{f}(y)\right|^{t}d\mu(y)\right)^{\frac{p}{t}}\mu(J)+\sum_{L\in\mathcal{J}(J)}\int_{L}\left(M_{L,W}\vec{f}\right)^{p}d\mu\\
	\simeq & c_{n,p}[W]_{A_{s}}|\mathcal{W}_{J,p}\mathcal{W}_{J,p}^{-1}|_{op}^p\left(\frac{1}{\mu(J)}\int_{J}\left|\vec{f}(y)\right|^{t}d\mu(y)\right)^{\frac{p}{t}}\mu(J)+\sum_{L\in\mathcal{J}(J)}\int_{L}\left(M_{L,W}\vec{f}\right)^{p}d\mu\\
	=& c_{n,p}[W]_{A_{s}}\left(\frac{1}{\mu(J)}\int_{J}\left|\vec{f}(y)\right|^{t}d\mu(y)\right)^{\frac{p}{t}}\mu(J)+\sum_{L\in\mathcal{J}(J)}\int_{L}\left(M_{L,W}\vec{f}\right)^{p}d\mu.
\end{align*}

Note that if as usual, $\mathcal{J}_{t}(J)=\left\{ L\in\mathcal{J}(Q)\,:\,Q\in\mathcal{J}_{t-1}(J)\right\} $
with $\mathcal{J}_{0}(J)=\{J\}$ and $\mathcal{S}=\cup_{t}\mathcal{J}_{t}(J)$
then $\mathcal{S\subset D}(J)$ is a sparse and by iteration
\begin{align*}
	\int_{J}\left(M_{J,W}\vec{f}(x)\right)^{p}dx & \lesssim c_{n,p}[W]_{A_{s}}\sum_{L\in\mathcal{S}}\left(\frac{1}{\mu(L)}\int_{L}\left|\vec{f}(y)\right|^{t}d\mu(y)\right)^{\frac{p}{t}}\mu(L)\\
	&\lesssim c_{n,p}[W]_{A_{s}}\sum_{L\in\mathcal{S}}\left(\inf_{z\in L}M_t|\vec{f}(z)|\right)^p\mu(E_L)\\
	& \leq c_{n,p}[W]_{A_{s}}\int_{J}(M_{t}|\vec{f}(x)|)^{p}d\mu(x).
\end{align*}
Hence
\begin{align*}
	\int_{X}\left( M_{W,k}^{\mathcal{D}}\vec{f}(x)\right) ^{p}d\mu(x) & = \int_{X}\sum_{J\in\mathcal{D}_{k}}\left( M_{J,W}\vec{f}(x)\chi_{J}(x)d\right)^p\mu(x)\\
	& =\sum_{J\in\mathcal{D}_{k}}\int_{J}\left( M_{J,W}\vec{f}(x)\right) ^{p}d\mu(x).\\
	& \lesssim[W]_{A_{s}}\sum_{J\in\mathcal{D}_{k}}\int_{J}\left(M_{t}|\vec{f}(x)|\right)^{p}d\mu(x)\\
	& =[W]_{A_{s}}\int_{X}\left(M_{t}|\vec{f}(x)|\right)^{p}d\mu(x).
\end{align*}
Finally observe that
\begin{align*}
\left(\int_{X}M_{t}(|\vec{f}|)^{p}d\mu\right)^{\frac{1}{p}} & =\left(\int_{X}M(|\vec{f}|^{t})^{\frac{p}{t}}d\mu\right)^{\frac{t}{p}\frac{1}{t}}\lesssim\left[\left(\frac{p}{t}\right)^{'}\right]^{\frac{1}{t}}\left(\int_{X}|\vec{f}|^{p}d\mu\right)^{\frac{1}{p}}\\
 & \lesssim\left[\left(\frac{p}{t}\right)^{'}\right]^{\frac{1}{t}}\left(\int_{X}|\vec{f}|^{p}d\mu\right)^{\frac{1}{p}}.
\end{align*}
This yields
\begin{align*}
\|M_{W,k}^{\mathcal{D}}\vec{f}\|_{L^{p}} & \lesssim[W]_{A_{s}}^{\frac{1}{p}}\|M_{t}|\vec{f}|\|_{L^{p}}\lesssim[W]_{A_{s}}^{\frac{1}{p}}\left[\left(\frac{p}{t}\right)^{'}\right]^{t}\|\vec{f}\|_{L^{p}}\\
 & \leq[W]_{A_{s}}^{\frac{1}{p}}\left[\left(\frac{p}{t}\right)^{'}\right]^{\frac{1}{t}}\|\vec{f}\|_{L^{p}}.
\end{align*}
This allows us to end the proof of the theorem since, in the case that $W\in A_q$ with $1\leq q<p$, we have that
\[[W]_{A_{s}}^{\frac{1}{p}}\left[\left(\frac{p}{t}\right)^{'}\right]^{\frac{1}{t}}\|\vec{f}\|_{L^{p}}=[W]_{A_{q}}^{\frac{1}{p}}\left[\left(\frac{p}{q}\right)^{'}\right]^{\frac{1}{q}}\|\vec{f}\|_{L^{p}}\simeq [W]_{A_{q}}^{\frac{1}{p}}\|\vec{f}\|_{L^{p}}\]
and in the case $W\in A_p$, since
\[
\left(\frac{p}{(rp')'}\right)^{'}\leq p'r'\qquad\text{and}\qquad\frac{1}{(rp')'}\leq\frac{1}{r'}+\frac{1}{p}
\]
by our choice for $r$ we have that 
\[
\left[\left(\frac{p}{(rp')'}\right)^{'}\right]^{\frac{1}{(rp')'}}\leq\left[p'r'\right]^{\frac{1}{r'}+\frac{1}{p}}\lesssim[W]_{A_{\infty,p}^{sc}}^{\frac{1}{p}}
\]
and hence
\[
[W]_{A_{s}}^{\frac{1}{p}}\left[\left(\frac{p}{t}\right)^{'}\right]^{\frac{1}{t}}\|\vec{f}\|_{L^{p}}
=[W]_{A_{p}}^{\frac{1}{p}}\left[\left(\frac{p}{(rp')'}\right)^{'}\right]^{\frac{1}{(rp')'}}\|\vec{f}\|_{L^{p}}
\lesssim[W]_{A_{p}}^{\frac{1}{p}}[W]_{A_{\infty,p}^{sc}}^{\frac{1}{p}}\|\vec{f}\|_{L^{p}}.
\]
This ends the proof.

\subsection{Proof of Theorem \ref{Thm:StrongTypeLrH}}
To settle Theorem \ref{Thm:StrongTypeLrH} it suffices to adapt arguments in \cite{MRR}. In this subsection we provide some hints on how to do that.
The following lemma is going to be used for the result for such operators

\begin{lem}\label{lem:ApFromRH}
Let $q,r,s>1$. Assume that 
\begin{align*}|\mathcal{V}_{B}\vec{e}|&\simeq\left(\frac{1}{\mu(B)}\int_B|W^{-\frac{1}{q}}(x)\vec{e}|^{q'r}\right)^\frac{1}{rq'}\\
	|\mathcal{U}_{B}\vec{e}|&\simeq\left(\frac{1}{\mu(B)}\int_B|W^{\frac{1}{q}}(x)\vec{e}|^{qs}\right)^\frac{1}{qs}
\end{align*}
for every $\vec{e}\in\mathbb{R}^n$ and that 
\begin{align*}\left(\frac{1}{\mu(B)}\int_B|W^{-\frac{1}{q}}(x)\vec{e}|^{q'r}\right)^\frac{1}{rq'}&\lesssim\left(\frac{1}{\mu(2c_{d}B)}\int_{2c_{d}B}|W^{-\frac{1}{q}}(x)\vec{e}|^{q'}\right)^\frac{1}{q'}\\
	\left(\frac{1}{\mu(B)}\int_B|W^{\frac{1}{q}}(x)\vec{e}|^{qs}\right)^\frac{1}{qs}& \lesssim\left(\frac{1}{\mu(2c_{d}B)}\int_{2c_{d}B}|W^{\frac{1}{q}}(x)\vec{e}|^{q}\right)^\frac{1}{q}
\end{align*}
for every $\vec{e}\in\mathbb{R}^n$. Then
\[|\mathcal{V}_{B}\mathcal{U}_{B}\vec{e}|\lesssim|\mathcal{W}_{2c_{d}B,q}\mathcal{W}'_{2c_{d}B,q}\vec{e}|.\]
\end{lem}
\begin{proof}
Note that, taking into account the reverse H\"older inequality in the hypothesis,
\begin{align*}
|\mathcal{V}_{B}\mathcal{U}_{B}\vec{e}|&\simeq \left(\dashint_{\mu(B)}|W^{-\frac{1}{q}}(x)\mathcal{U}_{B}\vec{e}|^{rq'}\right)^{\frac{1}{rq'}}\lesssim \left(\dashint_{\mu(2c_{d}B)}|W^{\frac{-1}{q}}(x)\mathcal{U}_{B}\vec{e}|^{q'}\right)^{\frac{1}{q'}}\\
&\simeq |\mathcal{W}'_{2c_{d}B,q}\mathcal{U}_{B}\vec{e}|.
\end{align*}
Hence \[|\mathcal{V}_{B}\mathcal{U}_{B}|_{op}\lesssim  |\mathcal{W}'_{2c_{d}B,q}\mathcal{U}_{B}|_{op}.\]
Now observe that  $|\mathcal{W}'_{2c_{d}B,q}\mathcal{U}_{B}|_{op}=|\mathcal{U}_{B}\mathcal{W}'_{2c_{d}B,q}|_{op}$. Then, again by the reverse H\"older inequality in the hypothesis,
\begin{align*}
|\mathcal{U}_{B}\mathcal{W}'_{2c_{d}B,q}\vec{e}|&\simeq \left(\dashint_{\mu(B)}|W^{\frac{1}{q}}(x)\mathcal{W}'_{2c_{d}B,q}\vec{e}|^{sq}\right)^{\frac{1}{s q}}\\
&\lesssim \left(\dashint_{\mu(2c_{d}B)}|W^{\frac{1}{q}}(x)\mathcal{W}'_{2c_{d}B,q}\vec{e}|^{q}\right)^{\frac{1}{q}}\simeq|\mathcal{W}_{2c_{d}B,q}\mathcal{W}'_{2c_{d}B,q}\vec{e}|
\end{align*}
and we are done.
\end{proof}


By standard arguments it is not hard to settle the following Lemma.
\begin{lem}
\label{lema3MRR} Let $p,r,s\geq1$ and $W$ be a matrix weight. For
every $\eta$-sparse family $\mathcal{S}$, we have that
\begin{align*}
 & \sum_{Q\in\mathcal{S}}\frac{1}{\mu(Q)}\int_{Q}\frac{1}{\mu(Q)}\int_{Q}\left|\left\langle W^{-\frac{1}{p}}(x)\vec{f}(x)k_{Q}(x,y),W^{\frac{1}{p}}(y)\vec{g}(y)\right\rangle \right|dxdy\mu(Q)\\
 & \leq\frac{1}{\eta}\sup_{Q}|\mathcal{U}_{Q}\mathcal{V}_{Q}|_{op}\,\|M_{\mathcal{V},W^{-\frac{1}{p}},r}\|_{L^{p}\to L^{p}}\|M_{\mathcal{U},W^{\frac{1}{p}},s}\|_{L^{p'}\to L^{p'}}\|\vec{f}\|_{L^{p}}\|\vec{g}\|_{L^{p'}}
\end{align*}
where $\ensuremath{\left\Vert k_{Q}(x,y)\right\Vert _{L^{(s',r')}\left[(Q,\frac{d\mu(y)}{\mu(Q)}),(Q,\frac{d\mu(x)}{\mu(Q)})\right]}\leq1}$,
\[
M_{\mathcal{V},W^{-\frac{1}{p}},r}(\vec{f})(z)=\sup_{x\in Q}\left(\frac{1}{\mu(Q)}\int_{Q}\left|(\mathcal{V}_{Q})^{-1}W^{-\frac{1}{p}}(x)\vec{f}(x)\right|^{r}d\mu(x)\right)^{1/r},
\]
\[
M_{\mathcal{U},W^{\frac{1}{p}},s}(\vec{g})(z)=\sup_{x\in Q}\left(\frac{1}{\mu(Q)}\int_{Q}\left|(\mathcal{U}_{Q})^{-1}W^{\frac{1}{p}}(x)\vec{g}(x)\right|^{s}d\mu(x)\right)^{1/s},
\]
and $\{\mathcal{U}_{Q}\}_{Q}$, $\{\mathcal{V}_{Q}\}_{Q}$ are families
of positive definite self-adjoint matrices.
\end{lem}

Even though the proof of Theorem \ref{Thm:StrongTypeLrH} could be provided relying upon arguments in \cite{NPTV,MRR,CUIM} we provide a full proof here for reader's convenience.
\begin{proof}[Proof of Theorem \ref{Thm:StrongTypeLrH}]
It is clear that it suffices to show that
\[
\|W^{\frac{1}{p}}T(W^{-\frac{1}{p}}\vec{f})\|_{L^{p}}\leq\kappa\||\vec{f}|\|_{L^{p}}
\]
where $\kappa$ is the constant depending on the  constant of the weight in each case.

First we observe that by our sparse domination result and
Lemma \ref{lema3MRR}
we have that
\begin{align}
 & \left\Vert W^{\frac{1}{p}}T(W^{-\frac{1}{p}}\vec{f})\right\Vert _{L^{p}}\nonumber \\
 & ={\displaystyle \sup_{\|g\|_{L^{p'}}\leq1}\int_{X}\left\vert \left\langle W^{\frac{1}{p}}T(W^{-\frac{1}{p}}\vec{f}),g\right\rangle \right\vert d\mu}\nonumber \\
 & ={\displaystyle \sup_{\|g\|_{L^{p'}}\leq1}\int_{X}\left\vert \left\langle T(W^{-\frac{1}{p}}\vec{f}),W^{\frac{1}{p}}g\right\rangle \right\vert d\mu}\nonumber \\
 & \leq{\displaystyle \sup_{\|g\|_{L^{p'}}\leq1}\sum_{j=1}^{m}\sum_{Q\in\mathcal{S}}\frac{1}{\mu(Q)}\int_{Q}\frac{1}{\mu(Q)}\int_{Q}\left|\left\langle W^{-\frac{1}{p}}(x)\vec{f}(x)k_{Q}(x,y),W^{\frac{1}{p}}(y)\vec{g}(y)\right\rangle \right|dxdy\mu(Q)}\nonumber \\
 & \leq\sum_{j=1}^{m}{\displaystyle \sup_{\|g\|_{L^{p'}}\leq1}2c_{n,T}{\displaystyle \sup_{Q}|\mathcal{U}_{Q}\mathcal{V}_{Q}|_{op}\left\Vert M_{\mathcal{V},W^{-\frac{1}{p}},r}\right\Vert _{L^{p}\to L^{p}}\left\Vert M_{\mathcal{U},W^{\frac{1}{p}},1}\right\Vert _{L^{p'}\to L^{p'}}\left\Vert \vec{f}\right\Vert _{L^{p}}\left\Vert \vec{g}\right\Vert _{L^{p'}}}}\nonumber \\
 & \leq2mc_{n,T}{\displaystyle \sup_{Q}|\mathcal{U}_{Q}\mathcal{V}_{Q}|_{op}\left\Vert M_{\mathcal{V},W^{-\frac{1}{p}},r}\right\Vert _{L^{p}\to L^{p}}\left\Vert M_{\mathcal{U},W^{\frac{1}{p}},1}\right\Vert _{L^{p'}\to L^{p'}}\left\Vert \vec{f}\right\Vert _{L^{p}}}\label{estimacion}
\end{align}
where each $\mathcal{S}_{j}$ is a sparse family and $\mathcal{U}=\left\lbrace \mathcal{U}_{Q}\right\rbrace $
and $\mathcal{V}=\left\lbrace \mathcal{V}_{Q}\right\rbrace $ are families
of positive self-adjoint matrices.

We begin establishing \eqref{eq:AprLrH}. We have to show that 
\[
\sup_{Q}|\mathcal{U}_{Q}\mathcal{V}_{Q}|_{op}\left\Vert M_{\mathcal{V},W^{-\frac{1}{p}},r}\right\Vert _{L^{p}\to L^{p}}\left\Vert M_{\mathcal{U},W^{\frac{1}{p}},1}\right\Vert _{L^{p'}\to L^{p'}}\lesssim[W]_{A_{\frac{p}{r}}}^{\frac{1}{p}}[W^{-\frac{r}{p}(\frac{p}{r})'}]_{A_{\infty,(\frac{p}{r})'}^{sc}}^{\frac{1}{p}}[W]_{A_{\infty,\frac{p}{r}}^{sc}}^{\frac{1}{p'}}
\]
Let us choose $\mathcal{V}_{Q}=\mathcal{A}_{Q}^{\frac{1}{r}}$, where
\[
\left|\mathcal{A}_{Q}\vec{e}\right|\simeq\left(\frac{1}{\mu(1Q)}\int_{1Q}\left|W^{-\frac{r}{p}}\vec{e}\right|^{\left(\frac{p}{r}\right)'\alpha}d\mu\right)^{\frac{1}{\left(\frac{p}{r}\right)'\alpha}}
\]
with
\[
\alpha=1+\dfrac{1}{\tau_{c_{\rho}\mu}\left[W^{-\frac{r}{p}\left(\frac{p}{r}\right)'}\right]_{A_{\infty,\left(\frac{p}{r}\right)'}^{sc}}}
\]
and $\mathcal{U}_{Q}=\mathcal{B}_{Q}^{\frac{1}{r}}$, where 
\[
\left|\mathcal{B}_{Q}\vec{e}\right|\simeq\left(\frac{1}{\mu(1Q)}\int_{1Q}\left|W^{\frac{r}{p}}\vec{e}\right|^{\left(\frac{p}{r}\right)\beta}d\mu\right)^{\frac{1}{\left(\frac{p}{r}\right)\beta}}
\]
with
\[
\beta=1+\dfrac{1}{\tau_{c_{\rho}\mu}\left[W\right]_{A_{\infty,\left(\frac{p}{r}\right)}^{sc}}}.
\]
By Lemma \ref{lem:alphaFuera} we have that 
\[
|\mathcal{U}_{Q}\mathcal{V}_{Q}|_{op}=|\mathcal{A}_{Q}^{\frac{1}{r}}\mathcal{B}_{Q}^{\frac{1}{r}}|_{op}\lesssim|\mathcal{A}_{Q}\mathcal{B}_{Q}|_{op}^{\frac{1}{r}}
\]
and by Lemma \ref{lem:ApFromRH} combined with the reverse Hölder
inequality (Lemma \ref{lem:RevHolderDouble}), 
\[
|\mathcal{A}_{Q}\mathcal{B}_{Q}|_{op}^{\frac{1}{r}}\lesssim\left|\mathcal{W}_{2c_{\rho}B,\frac{p}{r}}\mathcal{W}'_{2c_{\rho}B,\frac{p}{r}}\right|_{op}^{\frac{1}{r}}.
\]
Consequently
\begin{align}
{\displaystyle \sup_{Q}|\mathcal{U}_{Q}\mathcal{V}_{Q}|_{op}} & \lesssim{\displaystyle \sup_{B}\left|\mathcal{W}_{2c_{\rho}B,\frac{p}{r}}\mathcal{W}'_{2c_{\rho}B,\frac{p}{r}}\right|_{op}^{\frac{1}{r}}\simeq[W]_{A_{\frac{p}{r}}}^{\frac{1}{p/r}\cdot\frac{1}{r}}=[W]_{A_{\frac{p}{r}}}^{\frac{1}{p}}}.\label{T3.1}
\end{align}
We continue showing that
\begin{equation}
\left\Vert M_{\mathcal{V},W^{-\frac{1}{p}},r}\right\Vert _{L^{p}\to L^{p}}\lesssim[W^{-\frac{r}{p}\left(\frac{p}{r}\right)'}]_{A_{\infty,\left(\frac{p}{r}\right)'}^{sc}}^{\frac{1}{p}}.\label{eq:McalV}
\end{equation}
Let us call $B=1Q$. Using the matrix reverse Hölder and Lemma \ref{lem:alphaFuera}
we have that 
\begin{align*}
 & \left(\frac{1}{\mu(Q)}\int_{Q}\left|(\mathcal{V}_{Q})^{-1}W^{-\frac{1}{p}}(x)\vec{h}(x)\right|^{r}d\mu(x)\right)^{\frac{1}{r}}\\
 & \lesssim\left(\frac{1}{\mu(B)}\int_{B}\left|\mathcal{A}_{Q}^{-\frac{1}{r}}W^{-\frac{1}{p}}(x)\right|_{op}^{r\left(\frac{p}{r}\right)'\alpha}d\mu(x)\right)^{\frac{1}{r\left(\frac{p}{r}\right)'\alpha}}\left(\frac{1}{\mu(B)}\int_{B}|\vec{h}|^{r\left(\left(\frac{p}{r}\right)'\alpha\right)'}d\mu(x)\right)^{\frac{1}{r\left(\left(\frac{p}{r}\right)'\alpha\right)'}}\\
 & \lesssim\left(\frac{1}{\mu(B)}\int_{B}\left|\mathcal{A}_{Q}^{-1}W^{-\frac{r}{p}}(x)\right|_{op}^{\left(\frac{p}{r}\right)'\alpha}d\mu(x)\right)^{\frac{1}{r\left(\frac{p}{r}\right)'\alpha}}\left(\frac{1}{\mu(B)}\int_{B}|\vec{h}|^{r\left(\left(\frac{p}{r}\right)'\alpha\right)'}d\mu(x)\right)^{\frac{1}{r\left(\left(\frac{p}{r}\right)'\alpha\right)'}}\\
 & =\left(\frac{1}{\mu(B)}\int_{B}\left|W^{-\frac{r}{p}}(x)\mathcal{A}_{Q}^{-1}\right|_{op}^{\left(\frac{p}{r}\right)'\alpha}d\mu(x)\right)^{\frac{1}{r\left(\frac{p}{r}\right)'\alpha}}\left(\frac{1}{\mu(B)}\int_{B}|\vec{h}|^{r\left(\left(\frac{p}{r}\right)'\alpha\right)'}d\mu(x)\right)^{\frac{1}{r\left(\left(\frac{p}{r}\right)'\alpha\right)'}}.
\end{align*}
Bearing in mind the definition of $\mathcal{A}_{Q}$, 
\begin{align*}
 & \left(\frac{1}{\mu(B)}\int_{B}\left|W^{-\frac{r}{p}}(x)\mathcal{A}_{Q}^{-1}\right|_{op}^{r\left(\frac{p}{r}\right)'\alpha}d\mu(x)\right)^{\frac{1}{r\left(\frac{p}{r}\right)'\alpha}}\left(\frac{1}{\mu(B)}\int_{B}|\vec{h}|^{r\left(\left(\frac{p}{r}\right)'\alpha\right)'}d\mu(x)\right)^{\frac{1}{r\left(\left(\frac{p}{r}\right)'\alpha\right)'}}\\
\simeq & \left|\mathcal{A}_{Q}\mathcal{A}_{Q}^{-1}\right|_{op}^{\frac{1}{r}}\left(\frac{1}{\mu(B)}\int_{B}|\vec{h}|^{r\left(\left(\frac{p}{r}\right)'\alpha\right)'}d\mu(x)\right)^{\frac{1}{r\left(\left(\frac{p}{r}\right)'\alpha\right)'}}\\
= & \left(\frac{1}{\mu(B)}\int_{B}|\vec{h}|^{r\left(\left(\frac{p}{r}\right)'\alpha\right)'}d\mu(x)\right)^{\frac{1}{r\left(\left(\frac{p}{r}\right)'\alpha\right)'}}.
\end{align*}

By our choice of $\alpha$, and the estimates above, 
\begin{align*}
\left\Vert M_{\mathcal{V},W^{-\frac{1}{p}},r}(\vec{h})\right\Vert _{L^{p}} & \lesssim\left\Vert M_{r\left(\left(\frac{p}{r}\right)'\alpha\right)'}(|\vec{h}|)\right\Vert {}_{L^{p}}\lesssim\left[\left(\dfrac{p}{r\left(\left(\frac{p}{r}\right)'\alpha\right)'}\right)'\right]^{\frac{1}{r\left(\left(\frac{p}{r}\right)'\alpha\right)'}}\left\Vert \vec{h}\right\Vert {}_{L^{p}}\\
 & \leq[W^{-\frac{r}{p}\left(\frac{p}{r}\right)'}]_{A_{\infty,\left(\frac{p}{r}\right)'}^{sc}}^{\frac{1}{p}}\left\Vert \vec{h}\right\Vert {}_{L^{p}}.
\end{align*}
Consequently, \eqref{eq:McalV} holds.

Now we show that
\begin{equation}
\left\Vert M_{\mathcal{U},W^{\frac{1}{p}},1}\right\Vert _{L^{p'}\to L^{p'}}\lesssim[W]_{A_{\infty,\frac{p}{r}}^{sc}}^{\frac{1}{p'}}\label{eq:McalU}.
\end{equation}
Arguing as above, calling $B=1Q$, we have that
\begin{align*}
 & \frac{1}{\mu(Q)}\int_{Q}|\mathcal{U}_{Q}W^{\frac{1}{p}}(x)\vec{g}(x)|d\mu(x)\\
\lesssim & \left(\frac{1}{\mu(B)}\int_{B}|\mathcal{B}_{Q}^{-\frac{1}{r}}W^{\frac{1}{p}}(x)|_{op}^{p\beta}d\mu(x)\right)^{\frac{1}{p\beta}}\left(\frac{1}{\mu(B)}\int_{B}|\vec{g}(x)|^{(p\beta)'}d\mu(x)\right)^{\frac{1}{(p\beta)'}}\\
\lesssim & \left(\frac{1}{\mu(B)}\int_{B}|\mathcal{B}_{Q}^{-1}W^{\frac{r}{p}}(x)|_{op}^{\frac{p}{r}\beta}d\mu(x)\right)^{\frac{1}{p\beta}}\left(\frac{1}{\mu(B)}\int_{B}|\vec{g}(x)|^{(p\beta)'}d\mu(x)\right)^{\frac{1}{(p\beta)'}}\\
\simeq & \left(\frac{1}{\mu(B)}\int_{B}|W^{\frac{r}{p}}(x)\mathcal{B}_{Q}^{-1}|_{op}^{\frac{p}{r}\beta}d\mu(x)\right)^{\frac{1}{p\beta}}\left(\frac{1}{\mu(B)}\int_{B}|\vec{g}(x)|^{(p\beta)'}d\mu(x)\right)^{\frac{1}{(p\beta)'}}.
\end{align*}
Now by the definition of $\mathcal{B}_{Q},$ 
\begin{align*}
 & \left(\frac{1}{\mu(B)}\int_{B}|W^{\frac{r}{p}}(x)\mathcal{B}_{Q}^{-1}|_{op}^{\frac{p}{r}\beta}d\mu(x)\right)^{\frac{1}{p\beta}}\left(\frac{1}{\mu(B)}\int_{B}|\vec{g}(x)|^{(p\beta)'}d\mu(x)\right)^{\frac{1}{(p\beta)'}}\\
\simeq & |\mathcal{B}_{Q}\mathcal{B}_{Q}^{-1}|_{op}^{\frac{1}{r}}\left(\frac{1}{\mu(B)}\int_{B}|\vec{g}(x)|^{(p\beta)'}d\mu(x)\right)^{\frac{1}{(p\beta)'}}\\
\simeq & \left(\frac{1}{\mu(B)}\int_{B}|\vec{g}(x)|^{(p\beta)'}d\mu(x)\right)^{\frac{1}{(p\beta)'}}.
\end{align*}
Consecuently, bearing in mind our choice for $\beta,$ 
\[
\left\Vert M_{\mathcal{U},W^{\frac{1}{p}},1}\right\Vert _{L^{p'}\to L^{p'}}\leq\left\Vert M_{(p\beta)'}\right\Vert _{L^{p'}\to L^{p'}}\lesssim\left[\left(\dfrac{p'}{(p\beta)'}\right)'\right]^{\frac{1}{(p\beta)'}}\lesssim[W]_{A_{\infty,\frac{p}{r}}^{sc}}^{\frac{1}{p'}}.
\]
Gathering the estimates above,
\[
\left\Vert W^{\frac{1}{p}}T(W^{-\frac{1}{p}}\vec{f})\right\Vert _{L^{p}}\lesssim[W]_{A_{\frac{p}{r}}}^{\frac{1}{p}}[W^{-\frac{r}{p}\left(\frac{p}{r}\right)'}]_{A_{\infty,\left(\frac{p}{r}\right)'}^{sc}}^{\frac{1}{p}}[W]_{A_{\infty,\frac{p}{r}}^{sc}}^{\frac{1}{p'}}\left\Vert \vec{f}\right\Vert _{L^{p}}
\]
and we are done.

Now we settle \eqref{eq:AqLrH}. Again, it suffices to make a suitable choice of $\mathcal{U}_{Q}$ and 
$\mathcal{V}_{Q}$ in order to provide bounds for
\[
\sup_{Q}|\mathcal{U}_{Q}\mathcal{V}_{Q}|_{op}\,\qquad\left\Vert M_{\mathcal{V},W^{-\frac{1}{p}},r}\right\Vert _{L^{p}\to L^{p}},\qquad\left\Vert M_{\mathcal{U},W^{\frac{1}{p}},1}\right\Vert _{L^{p'}\to L^{p'}}.
\]
For the case $q=1$ we choose $\mathcal{V}_{Q}=\mathcal{A}_{2c_{\rho}B}^{-\frac{1}{p}}$,
where $B=1Q$, and
\[
\left|\mathcal{A}_{2c_{\rho}B}\vec{e}\right|\simeq\int_{2c_{\rho}B}|W(x)\vec{e}|d\mu(x).
\]
We also choose, $\mathcal{U}_{Q}=\mathcal{V}_{Q}^{-1}$. Hence $\sup_{Q}|\mathcal{U}_{Q}\mathcal{V}_{Q}|_{op}=1$.

Let us see that
\[
\left\Vert M_{\mathcal{V},W^{-\frac{1}{p}},r}\right\Vert _{L^{p}\rightarrow L^{p}}\lesssim[W]_{A_{1}}^{1/p}.
\]
Indeed, using the definition of $\mathcal{A}_{2c_{\rho}B}$ and taking
into account that $W\in A_{1}$, 
\begin{align*}
 & \left(\frac{1}{\mu(Q)}\int_{Q}\left|\mathcal{V}_{B}^{-1}W^{-\frac{1}{p}}(x)\vec{f}(x)\right|^{r}d\mu(x)\right)^{\frac{1}{r}}=\left(\frac{1}{\mu(Q)}\int_{Q}\left|\mathcal{A}_{2c_{\rho}B}^{\frac{1}{p}}W^{-\frac{1}{p}}(x)\vec{f}(x)\right|^{r}d\mu(x)\right)^{\frac{1}{r}}\\
 & \leq\left(\frac{1}{\mu(Q)}\int_{Q}\left|\mathcal{A}_{2c_{\rho}B}W^{-1}(x)\right|_{op}^{\frac{r}{p}}\left|\vec{f}(x)\right|^{r}d\mu(x)\right)^{\frac{1}{r}}\leq\\
 & \lesssim\left(\frac{1}{\mu(2c_{\rho}B)}\int_{2c_{\rho}B}\left({\displaystyle \sum_{j=1}^{n}\frac{1}{\mu(2c_{\rho}B)}\int_{2c_{\rho}B}\left|W(y)W^{-1}(x)\vec{e}_{j}\right|}d\mu(y)\right)^{\frac{r}{p}}\left|\vec{f}(x)\right|^{r}d\mu(x)\right)^{\frac{1}{r}}\\
 & \leq n^{\frac{1}{p}}\left(\frac{1}{\mu(2c_{\rho}B)}\int_{2c_{\rho}B}\left({\displaystyle \frac{1}{\mu(2c_{\rho}B)}\int_{2c_{\rho}B}\left|W(y)W^{-1}(x)\right|_{op}}d\mu(y)\right)^{\frac{r}{p}}\left|\vec{f}(x)\right|^{r}d\mu(x)\right)^{\frac{1}{r}}\\
 & \lesssim_{n,p}[W]_{A_{1}}^{\frac{1}{p}}\left(\frac{1}{\mu(2c_{\rho}B)}\int_{2c_{\rho}B}\left|\vec{f}(x)\right|^{r}d\mu(x)\right)^{\frac{1}{r}}.
\end{align*}
This yields that
\[
\left\Vert M_{\mathcal{V},W^{-\frac{1}{p}},r}\right\Vert _{L^{p}\rightarrow L^{p}}\leq\left\Vert M_{r}\right\Vert _{L^{p}\rightarrow L^{p}}[W]_{A_{1}}^{\frac{1}{p}}\lesssim[W]_{A_{1}}^{\frac{1}{p}}\left[\left(\frac{p}{r}\right)'\right]{}^{\frac{1}{r}}
\]
as we wanted to show.

Let us show now that 
\[
\left\Vert M_{\mathcal{U},W^{-\frac{1}{p}},r}\right\Vert _{L^{p'}\rightarrow L^{p'}}\leq[W]_{A_{\infty,1}^{sc}}^{1/p'}.
\]
For that purpose we choose $\beta=1+\dfrac{1}{\tau_{c_{\rho}\mu}[W]_{A_{\infty,1}^{sc}}}$.
Then we can argue as follows

\begin{align*}
 & \frac{1}{\mu(Q)}\int_{Q}\left|(\mathcal{U}_{B})^{-1}W^{\frac{1}{p}}(z)\vec{g}(z)\right|d\mu(z)=\frac{1}{\mu(Q)}\int_{Q}\left|\mathcal{A}_{2c_{\rho}B}^{-\frac{1}{p}}W^{\frac{1}{p}}(z)\vec{g}(z)\right|d\mu(z)\\
 & \lesssim\left(\frac{1}{\mu(1Q)}\int_{1Q}\left|\mathcal{A}_{2c_{\rho}B}^{-\frac{1}{p}}W^{\frac{1}{p}}(z)\right|_{op}^{p\beta}d\mu(z)\right)^{\frac{1}{p\beta}}\left(\frac{1}{\mu(Q)}\int_{Q}\left|\vec{g}(z)\right|^{(p\beta')}\right)^{\frac{1}{(p\beta)'}}\\
 & \lesssim\left(\frac{1}{\mu(1Q)}\int_{1Q}\left|W(z)\mathcal{A}_{2c_{\rho}B}^{-1}\right|_{op}^{\beta}d\mu(z)\right)^{\frac{1}{p\beta}}\left(\frac{1}{\mu(Q)}\int_{Q}\left|\vec{g}(z)\right|^{(p\beta')}\right)^{\frac{1}{(p\beta)'}}\\
 & \lesssim\left(\frac{1}{\mu(2c_{\rho}B)}\int_{2c_{\rho}B}\left|W(z)\mathcal{A}_{2c_{\rho}B}^{-1}\right|_{op}d\mu(z)\right)^{\frac{1}{p}}\left(\frac{1}{\mu(Q)}\int_{Q}\left|\vec{g}(z)\right|^{(p\beta')}\right)^{\frac{1}{(p\beta)'}}\\
 & \simeq\left|\mathcal{A}_{2c_{\rho}B}\mathcal{A}_{2c_{\rho}B}^{-1}\right|_{op}^{\frac{1}{p}}\left(\frac{1}{\mu(Q)}\int_{Q}\left|\vec{g}(z)\right|^{(p\beta')}\right)^{\frac{1}{(p\beta)'}}\\
 & \simeq\left(\frac{1}{\mu(Q)}\int_{Q}\left|\vec{g}(z)\right|^{(p\beta')}\right)^{\frac{1}{(p\beta)'}}.
\end{align*}
At this point, this allows us to conclude that
\[
\left\Vert M_{\mathcal{U},W^{-\frac{1}{p}},r}\right\Vert _{L^{p'}\rightarrow L^{p'}}\lesssim\left\Vert M_{(p\beta)'}\right\Vert _{L^{p'}\rightarrow L^{p'}}\lesssim\left[\left(\dfrac{p'}{(p\beta)'}\right)'\right]^{\frac{1}{(p\beta)'}}\lesssim[W]_{A_{\infty,1}^{sc}}^{\frac{1}{p'}}
\]
and combining the estimates above we are done.

For the case $1<q<p$, arguing as above, it is enough to choose matrices
$\mathcal{U}_{Q}$ and $\mathcal{V}_{Q}$ such that
\[
\sup_{Q}\left|\mathcal{U}_{Q}\mathcal{V}_{Q}\right|_{op}\leq1,\qquad\left\Vert M_{\mathcal{V},W^{-\frac{1}{p}},r}\right\Vert _{L^{p}\rightarrow L^{p}}\lesssim\left[\left(\dfrac{p}{rq}\right)'\right]^{\frac{1}{rq}},\qquad\left\Vert M_{\mathcal{U},W^{-\frac{1}{p}},r}\right\Vert _{L^{p'}\rightarrow L^{p'}}\lesssim[W]_{A_{\infty,q}^{sc}}^{1/p'}[W]_{A_{q}}^{1/p}.
\]
Choosing $\mathcal{V}_{Q}=\mathcal{A}_{2c_{\rho}B}^{\frac{q}{p}}$,
where $B=1Q$ and
\[
\left|\mathcal{A}_{2c_{\rho}B}\vec{e}\right|\simeq\left(\frac{1}{\mu(2c_{\rho}B)}\int_{2c_{\rho}B}|W^{-\frac{1}{q}}(z)\vec{e}|^{q'}d\mu(z)\right)^{\frac{1}{q'}}\,,
\]
and also $\mathcal{U}_{Q}=\mathcal{V}_{Q}^{-1}$, its clear that the
first estimate above holds. For the second, since $0<r\frac{q}{p}<1$,
we have that
\begin{align*}
 & \left(\frac{1}{\mu(Q)}\int_{Q}\left|\mathcal{V}_{Q}^{-1}W^{-\frac{1}{p}}(x)\vec{f}(x)\right|^{r}d\mu(x)\right)^{\frac{1}{r}}=\left(\frac{1}{\mu(Q)}\int_{Q}\left|\mathcal{A}_{2c_{\rho}B}^{-\frac{q}{p}}W^{-\frac{1}{p}}(x)\vec{f}(x)\right|^{r}d\mu(x)\right)^{\frac{1}{r}}\\
\leq & \left(\frac{1}{\mu(Q)}\int_{Q}\left|\mathcal{A}_{2c_{\rho}B}^{-\frac{q}{p}}W^{-\frac{1}{q}\frac{q}{p}}(x)\right|_{op}^{r}\left|\vec{f}(x)\right|^{r}d\mu(x)\right)^{\frac{1}{r}}\leq\left(\frac{1}{\mu(Q)}\int_{Q}\left|\mathcal{A}_{2c_{\rho}B}^{-1}W^{-\frac{1}{q}}(x)\right|^{r\frac{q}{p}}\left|\vec{f}(x)\right|^{r}d\mu(x)\right)^{\frac{1}{r}}\\
\leq & \left(\frac{1}{\mu(Q)}\int_{Q}\left|\mathcal{A}_{2c_{\rho}B}^{-1}W^{-\frac{1}{q}}(x)\right|^{r\frac{q}{p}q'}d\mu(x)\right)^{\frac{1}{rq'}}\left(\frac{1}{\mu(Q)}\int_{Q}\left|\vec{f}(x)\right|^{rq}d\mu(x)\right)^{\frac{1}{r}}\\
\lesssim & \left(\frac{1}{\mu(2c_{\rho}B)}\int_{2c_{\rho}B}\left|\mathcal{A}_{2c_{\rho}B}^{-1}W^{-\frac{1}{q}}(x)\right|^{q'}d\mu(x)\right)^{\frac{1}{q'}\frac{q}{p}}\left(\frac{1}{\mu(Q)}\int_{Q}\left|\vec{f}(x)\right|^{rq}d\mu(x)\right)^{\frac{1}{r}}\\
 & \simeq\left|\mathcal{A}_{2c_{\rho}B}^{-1}\mathcal{A}_{2c_{\rho}B}\right|_{op}^{\frac{q}{p}}\left(\frac{1}{\mu(Q)}\int_{Q}\left|\vec{f}(x)\right|^{rq}d\mu(x)\right)^{\frac{1}{r}}.
\end{align*}
Hence
\[
\left\Vert M_{\mathcal{V},W^{-\frac{1}{p}},r}\right\Vert _{L^{p}\rightarrow L^{p}}\lesssim\left\Vert M_{rq}\right\Vert _{L^{p}\rightarrow L^{p}}\lesssim\left[\left(\dfrac{p}{rq}\right)'\right]^{\frac{1}{rq}}
\]
Now we estimate the remaining term. Let $\beta=1+\dfrac{1}{\tau_{c_{\rho}\mu}[W]_{A_{\infty,q}^{sc}}}$.
Reverse Hölder inequality allows us to argue as follows

\begin{align*}
\quad & \frac{1}{\mu(Q)}\int_{Q}|\mathcal{U}_{Q}^{-1}W^{\frac{1}{p}}(y)\vec{g}(y)|d\mu(y)\leq\frac{1}{\mu(Q)}\int_{Q}|\mathcal{A}_{2c_{\rho}B}^{\frac{q}{p}}W^{\frac{1}{q}\frac{q}{p}}(y)\|\vec{g}(y)|d\mu(y)\\
 & \lesssim\left(\frac{1}{\mu(2c_{\rho}B)}\int_{2c_{\rho}B}|\mathcal{A}_{Q}W^{\frac{1}{q}}(x)|_{op}^{\frac{q}{p}p\beta}d\mu(x)\right)^{\frac{1}{p\beta}}\left(\frac{1}{\mu(2c_{\rho}B)}\int_{2c_{\rho}B}|\vec{g}|^{(p\beta)'}d\mu(x)\right)^{\frac{1}{(p\beta)'}}\\
 & \leq\left(\frac{1}{\mu(2c_{\rho}B)}\int_{2c_{\rho}B}|\mathcal{A}_{2c_{\rho}B}W^{\frac{1}{q}}(x)|_{op}^{q\beta}d\mu(x)\right)^{\frac{1}{q\beta}\frac{q}{p}}\left(\frac{1}{\mu(2c_{\rho}B)}\int_{2c_{\rho}B}|\vec{g}|^{(p\beta)'}d\mu(x)\right)^{\frac{1}{(p\beta)'}}\\
 & \lesssim\left(\frac{1}{\mu(2c_{\rho}B)}\int_{Q}|\mathcal{A}_{2c_{\rho}B}W^{\frac{1}{q}}(x)|_{op}^{q}d\mu(x)\right)^{\frac{1}{p}}\left(\frac{1}{\mu(2c_{\rho}B)}\int_{2c_{\rho}B}|\vec{g}|^{(p\beta)'}d\mu(x)\right)^{\frac{1}{(p\beta)'}}\\
 & \lesssim[W]_{A_{q}}^{\frac{1}{p}}\left(\frac{1}{|2c_{\rho}B|}\int_{2c_{\rho}B}|\vec{g}|^{(p\beta)'}d\mu(x)\right)^{\frac{1}{(p\beta)'}}.
\end{align*}
Finally, for this choice of $\beta$ we have that
\[
\|M_{\mathcal{U},W^{\frac{1}{p}},1}\|_{L^{p'}\rightarrow L^{p'}}  \lesssim[W]_{A_{q}}^{\frac{1}{p}}\|M_{(p\beta)'}\|_{L^{p'}\rightarrow L^{p'}}.
\lesssim[W]_{A_{q}}^{\frac{1}{p}}\left[\left(\frac{p'}{(p\beta)'}\right)'\right]^{\frac{1}{(p\beta)'}}\lesssim[W]_{A_{q}}^{\frac{1}{p}}[W]_{A_{\infty,q}^{sc}}^{\frac{1}{p'}},
\]
and this ends the proof.
\end{proof}

\subsection{Endpoint estimates}
Up until now we have dealt with strong type estimates. In order to study strong type estimates, we have actually studied unweighted estimates such as: 
\[
\left\| W^{1/p} T(W^{-1/p} \vec{f}) \right\|_{L^p} \lesssim \|\vec{f}\|_{L^p}
\]

Hence a reasonable choice for weak type estimates is to deal with weighted operators as well.
\[
\left| \left\{ x\in X : |W(x) T(W^{-1} \vec{f})(x)| > t \right\} \right| \lesssim \frac{1}{t} \int_{X} |\vec{f}(x)| \, dx
\]
That is particularly useful when dealing with endpoint estimates as it was shown in \cite{CUIMPRR}. Here we extend to spaces of homogeneous setting the result there for Calder\'on-Zygmund and $L^\infty$-Hörmander operators since they share the same convex body domination.
\begin{thm}	
	\label{endpoint}
	Let $(X,d,\mu)$ a space of homogeneous type and $W \in A_1$. Then 
	\begin{equation}
		\label{weakendpointT}
		\left| \left\{ x \in X: |W(x) T(W^{-1} \vec{f})(x)| > t \right\} \right| \lesssim [W]_{A_1} [W]_{A_{\infty}^{\mathrm{sc},1}} \frac{1}{t} \int_{X} |\vec{f}(x)| d\mu(x)
	\end{equation}
	where $T$ is a Calder\'on-Zygmund operator or a $L^\infty$-Hörmander operator and 
\begin{equation}
		\label{weakendpointM_W}
		\left| \left\{ x \in X: |M_{W,1} \vec{f}(x)| > t \right\} \right| \lesssim [W]_{A_1} [W]_{A_{\infty}^{\mathrm{sc},1}} \frac{1}{t} \int_{X} |\vec{f}(x)| d\mu(x)
	\end{equation}

\begin{proof}
	Given a dyadic cube we consider the reducing matrix
	\[|\mathcal{W}_{1,Q}\vec{v}|\simeq\dfrac{1}{\mu(2c_d 1Q)}\int_{2c_d 1Q}|W(x)\vec{v}|d\mu(x).\]
	
	For Calder\'on-Zygmund operators, applying the convex body domination theorem we have that there exists  $m\in\mathbb{N}$ such that there
	are dyadic systems $\mathcal{D}_{1},\dots,\mathcal{D}_{m}$ and $m$ $\frac{1}{2}$-sparse families $\mathcal{S}_{i}\subset\mathcal{D}_{i}$
	such that 
	\[
	\begin{split}
	|W(x) T(W^{-1} \vec{f})(x)|&=\left\vert c_{n,d,T}\sum_{j=1}^{m}\sum_{Q\in\mathcal{S}_{j}}\frac{1}{\mu(Q)}\int_{Q}k_{Q}(x,y)W(x)W^{-1}(y)\vec{f}(y)d\mu(y)\chi_{Q}(x)\right\vert\\
	&\lesssim \sum_{j=1}^{m}\sum_{Q\in\mathcal{S}_{j}}\frac{1}{\mu(Q)}\int_{Q}|W(x)W^{-1}(y)\vec{f}(y)|d\mu(y)\chi_{Q}(x)\\
	&\lesssim [W]_{A_1}\sum_{j=1}^{m}\sum_{Q\in\mathcal{S}_{j}}|W(x)\mathcal{W}_{1,Q}^{-1}|_{op}\dfrac{1}{\mu(Q)}\int_{Q}|\vec{f}(y)|d\mu(y)\chi_{Q}(x)\\
	&:=[W]_{A_1}\sum_{j=1}^m \mathcal{T}_{S_j,W}(\vec{f})(x)\,.
	\end{split}
	\] 
Observe that in this case we can split he family in order to make it $\frac{4}{5}$-sparse, and therefore $\frac{5}{4}$-Carleson \cite[Theorem 3.3]{DGKLWY}. 
		
Analogously, for the maximal function, arguing as we did in the proof of Theorem \ref{thm:StrongMax}, we have a  $\frac{5}{4}$-Carleson dyadic family $\mathcal{S}(J)\subset \mathcal{D}(J)$ such that	
	\[
	\begin{split}
		M_{J,W}\vec{f}(x)\leq & \sum_{Q\in\mathcal{S}(J)}\left|W(x)\mathcal{W}_{1,Q}^{-1}\right|_{op} \frac{1}{\mu(Q)}\int_{Q}\left|\mathcal{W}_{1,Q}W^{-1}(y)\vec{f}(y)\right|d\mu(y)\\
		&\lesssim [W]_{A_1}\sum_{Q\in\mathcal{S}(J)}\left|W(x)\mathcal{W}_{1,Q}^{-1}\right|_{op}\dfrac{1}{\mu(Q)}\int_{Q}\left|\vec{f}(y)\right|d\mu(y)\\
		&\lesssim [W]_{A_1}\sum_{j=1}^{m}\sum_{Q\in\mathcal{S}_{j}}|W(x)\mathcal{W}_{1,Q}^{-1}|_{op}\dfrac{1}{\mu(Q)}\int_{Q}|\vec{f}(y)|d\mu(y)\chi_{Q}(x).
	\end{split}
	\]	
Hence to settle Theorem \ref{endpoint} it suffices to show that
	\[\mu\left(\left\lbrace \mathcal{T}_{S_j,W}(\vec{f})(x)> t\right\rbrace\right)\lesssim [W]_{A_{1,\infty}^{sc}}\dfrac{1}{t}\int_X|\vec{f}(x)|\]
By homogeneity, it is enough to settle the case $t=1$. To that end, let $P$ be a cube, then we have
\[
\begin{split}
	\mu\left( \left\lbrace \left| \mathcal{T}_{S,W}(\vec{f})(x) \right| > 1 \right\rbrace\cap kP \right) 
	 &\lesssim  \mu\left(\left( \left\lbrace \left| \mathcal{T}_{S,W}(\vec{f})(x) \right| > 1 \right\rbrace 
	\setminus \left\lbrace M\vec{f}(x) > 1 \right\rbrace \right)\cap kP \right)\\
	& \quad + \| \vec{f} \|_{L^1}.
\end{split}
\]
	
Assume by now (we will prove prove this claim in Lemma \ref{medidaG}) that for $$G=\left(\{|\mathcal{T}_{S,W}(\vec{f})(x)|>1\}\setminus \{M\vec{f}(x)>1\}\right)\cap kP$$	
we have that 
\[
\mu(G)\lesssim 2[W]_{A_{1,\infty}^{sc}}\|\vec{f}\|_{L^1}. 
\]
Then, letting $k\to \infty$ this yields 
\[\mu\left( \left\lbrace \left| \mathcal{T}_{S,W}(\vec{f})(x) \right| > 1 \right\rbrace 
\setminus \left\lbrace M\vec{f}(x) > 1 \right\rbrace \right) \lesssim 2[W]_{A_{1,\infty}^{sc}}\|\vec{f}\|_{L^1}\, .  \]
Gathering the preceding estimates, 	
\[
\begin{split}
	\mu\left( \left\lbrace \left| \mathcal{T}_{S,W}(\vec{f})(x) \right| > 1 \right\rbrace \right)&\lesssim 2[W]_{A_{1,\infty}^{sc}}\|\vec{f}\|_{L^1 }+ \| \vec{f} \|_{L^1}\\
	&\leq 3[W]_{A_{1,\infty}^{sc}}\|\vec{f}\|_{L^1 }.
\end{split} 
\]
\end{proof}
\end{thm}	

To complete the proof of Theorem \ref{endpoint}, it remains to justify the claim made at the end of its proof. This is the content of the following lemma. 
	\begin{lem}
		\label{medidaG}
		Given $W\in A_{1,\infty}^{sc}$, $$G=\left(\{|\mathcal{T}_{S,W}(\vec{f})(x)|>1\}\setminus \{M\vec{f}(x)>1\}\right)\cap kP$$ where $P$ is a fixed cube, we have that 
		\[\mu(G)\lesssim 2[W]_{A_{1,\infty}^{sc}}\|\vec{f}\|_{L^1(X)}. \]
		
\begin{proof}
Relying upon the definition of $G$, we have that
			\[
			\mu(G) \leq \sum_{Q \in \mathcal{S}} \frac{1}{\mu(Q)} \int_{Q \cap G} \left| W(x) W^{-1}_{1,Q} \right|_{op}  \, \frac{1}{\mu(Q)} \int_Q |\vec{f}| \, \mu(Q). \]
Choosing \( g = \chi_G \) and \( s = 1 + \frac{1}{\tau_{c_{d}\mu} [W]_{A^{\text{sc}}_{1,\infty}}} \), and taking into account Lemma \ref{lem:RevHolderDouble},			
\[
			\begin{split}
			\frac{1}{\mu(Q)} \int_{Q \cap G} \left| W(x) \mathcal{W}^{-1}_{1,Q} \right|_{op} \,  
			&\lesssim \left( \frac{1}{\mu(Q)} \int_Q \left| W(x) \mathcal{W}^{-1}_{1,Q	} \right|_{op}^s \,  \right)^{\frac{1}{s}} 
			\left( \frac{1}{\mu(Q)} \int_Q g \right)^{\frac{1}{s'}}\\
			&\lesssim \frac{1}{\mu(2c_d1Q)} \int_{2c_d1Q} \left| W(x) \mathcal{W}^{-1}_{1,Q} \right|_{op} \left( \frac{1}{\mu(Q)} \int_Q g \right)^{\frac{1}{s'}}\\
			&=\left( \frac{1}{\mu(Q)} \int_Q g \right)^{\frac{1}{s'}}.
			\end{split}
			\]
			Hence, 
			\[\begin{split}
				\mu(G)&\leq \sum_{Q \in \mathcal{S}} \frac{1}{\mu(Q)} \int_{Q \cap G} \left| W(x) W^{-1}_{1,Q} \right|_{op}  \, \frac{1}{\mu(Q)} \int_Q |\vec{f}| \, \mu(Q)\\
				&\leq \kappa \sum_{Q \in \mathcal{S}} \left( \frac{1}{\mu(Q)} \int_Q g \right)^{\frac{1}{s'}} \frac{1}{\mu(Q)} \int_Q |\vec{f}| \, \mu(Q)
			\end{split}\]
		for some positive constant $\kappa$ depending on the constants of the dyadic structure.
			
		 We split the sparse family as follows \( Q \in \mathcal{S}_{k,j} \), \( k, j \geq 0 \) if
		 \[
		 \begin{split}
		 	2^{-j-1} &< \frac{1}{\mu(Q)} \int_Q |\vec{f}| \, \leq 2^{-j}, \\
		 	2^{-k-1} &< \langle g \rangle_{Q,s'} \leq 2^{-k}.
		 \end{split}
		 \]
		 Then
		 \[
		 \begin{split}
		 	&\sum_{Q \in \mathcal{S}} \left( \frac{1}{\mu(Q)} \int_Q g \right)^{\frac{1}{s'}} \left( \frac{1}{\mu(Q)} \int_Q |\vec{f}| \right) \mu(Q)\\
		 	&= \sum_{j=0}^\infty \sum_{k=0}^\infty \sum_{Q \in \mathcal{S}_{k,j}} \left( \frac{1}{\mu(Q)} \int_Q g \right)^{\frac{1}{s'}} \left( \frac{1}{\mu(Q)} \int_Q |\vec{f}| \right) \mu(Q) \\
		 	&:= \sum_{k=0}^\infty \sum_{j=0}^\infty s_{k,j}.
		 \end{split}
		 \]
		 
		 To estimate  $s_{k,j}$, let us define $\mathcal{S}_{j,k}(Q)=\{P\in \mathcal{S}_{j,k}\subsetneq Q\}$ and \(\widetilde{E_Q}=Q\setminus \bigcup_{Q'\in\mathcal{S}_{j,k}(Q)}Q' \,.\)
		 Since $\mathcal{S}$ is a $\frac{5}{4}$-Carleson family,
$$\int_Q |\vec{f}|\leq 2\int_{\widetilde{E_Q}} |\vec{f}|\, ,$$
		and then 
		\[
		s_{k,j} \leq \sum_{Q \in \mathcal{S}_{j,k}} 2 \int_{\widetilde{E_Q}} |\vec{f}| \left( \frac{1}{\mu(Q)} \int_Q g \right)^{\frac{1}{s'}} 
		\leq 2 \cdot 2^{-k} \sum_{Q \in \mathcal{S}_{j,k}} \int_{\widetilde{E_Q}} |\vec{f}| 
		\leq 2 \cdot 2^{-k} \| \vec{f} \|_{L^1}.
		\]
		On the other hand, 
		 \[
		 \begin{split}
		 	s_{k,j} &\leq 2^{-j} 2^{-k} \sum_{Q \in \mathcal{S}_{j,k}} \mu(Q) 
		 	\leq 2 \cdot 2^{-j} 2^{-k} \sum_{Q \in \mathcal{S}_{j,k}} \mu(E_Q) \\
		 	&\leq 2 \cdot 2^{-j} 2^{-k} \mu\left( \bigcup_{Q \in \mathcal{S}_{j,k}} Q \right)
		 	\leq 2 \cdot 2^{-j} 2^{-k} \mu\left( \left\{ x \in X : M_{s'} g > 2^{-k-1} \right\} \right) \\
		 	&= 2 \cdot 2^{-j} 2^{-k} \mu\left( \left\{ x \in X : M g > 2^{-s'k - s'} \right\} \right) 
		 	\leq  2^{-j-k+1+s'k+s'} \mu(G).
		 \end{split}
		 \]
		 and therefore,
		 \[
		 \begin{split}
		 	s_{k,j} \leq \alpha_{k,j}:= \min \left\{  2^{-k+1} \|\vec{f}\|_{L^1}, \, 2^{-j-k+1+s'k+s'} \mu(G) \right\}.
		 \end{split}
		 \]
		So we have the estimate
		\[
		\begin{split}
			\mu(G)&\leq \kappa \sum_{k=0}^\infty \sum_{j=0}^\infty s_{k,j}
			\leq \kappa\sum_{k=0}^\infty \sum_{j=0}^\infty \alpha_{k,j}\\	
			&=\kappa\left(\sum_{k=0}^\infty\sum_{\substack{j \ge \gamma + \lceil k(s' - 1) + s'\rceil+ k}} \alpha_{k,j}
			+ \sum_{\substack{j < \gamma + \lceil k(s' - 1) + s'\rceil+ k}} \alpha_{k,j} \right)			
		\end{split}
		\]
		 where $\lceil\alpha \rceil$ stands for the smallest integer $k_\alpha$ such that $\alpha\leq k_\alpha$ and  $\gamma$ is some for some positive integer to be chosen. For the first term, 
		\[
\begin{split}\sum_{k=0}^{\infty}\sum_{j\geq\gamma+\lceil k(s'-1)+s'\rceil+k}\alpha_{k,j} & \leq\mu(G)\sum_{k=0}^{\infty}2^{k(s'-1)+1+s'}\sum_{j\ge\gamma+\lceil k(s'-1)+s'\rceil+k}2^{-j}\\
 & \leq2\mu(G)\sum_{k=0}^{\infty}2^{k(s'-1)+s'}\cdot2^{-(\gamma+\lceil k(s'-1)+s'\rceil+k)}\\
 & \leq2\frac{\mu(G)}{2^{\gamma}}\sum_{k=0}^{\infty}2^{-k}
 =4\frac{\mu(G)}{2^{\gamma}}.
\end{split}
		\]
		For the second term, since $ s = 1 + \frac{1}{\tau_{c_{d}\mu} [W]_{A^{\text{sc}}_{1,\infty}}}$
		\[
\begin{split}\sum_{k=0}^{\infty}\sum_{0\leq j<\gamma+\lceil k(s'-1)+s'\rceil+k}\alpha_{k,j} & \leq2\|\vec{f}\|_{L^{1}}\sum_{k=0}^{\infty}\sum_{0\leq j<\gamma+\lceil k(s'-1)+s'\rceil+k}2^{-k}\\
 & \leq2\|\vec{f}\|_{L^{1}}\sum_{k=0}^{\infty}(\gamma+\lceil k(s'-1)+s'\rceil+k)2^{-k}\\
 & \leq c_\gamma\|\vec{f}\|_{L^{1}}[W]_{A_{1,\infty}^{\text{sc}}}
\end{split}
\]
for some $c_{\gamma}>0$. 

Combining the preceding estimates, we have that
\[
\mu(G)\leq\kappa c_{\gamma}\|\vec{f}\|_{L^{1}}[W]_{A_{1,\infty}^{\text{sc}}}+4\kappa\frac{\mu(G)}{2^{\gamma}}\,
\]
and therefore, choosing $\gamma$ such that $\frac{4\kappa}{2^{\gamma}}<\frac{1}{2}$
\[
\mu(G)\lesssim\|\vec{f}\|_{L^{1}}[W]_{A_{1,\infty}^{\text{sc}}}
\]
as we wanted to show.
		\end{proof}
	\end{lem}

\section{Proof of the $T(1)$-like sparse result}\label{sec:SparseT1proof}
\subsection{Lemmatta} 
To settle Theorem \ref{thm:sparseT1} we will need some the following Lemmatta.
\begin{lem}
\label{lem:Let:Aux}Let $(X,d,\mu)$ be a space of homogeneous type.
Assume that $K^{*}\in H_{1}$. Then: 
\begin{enumerate}
\item For every ball $B$ and $k\geq2c_{d}$ 
\begin{equation}
\int_{X\setminus kB}|Tf(x)|d\mu\leq c\int_{B}|f|d\mu\label{eq:Part1}
\end{equation}
provided $f$ is supported in $B$ and $\int_{B}f=0$ 
\item If additionally there exist $A,\gamma>0$ such that for every dyadic
system given by Proposition \ref{proposition:dyadicsystem}, any $Q\in\mathcal{D}$
and any $f\in L^{\infty}(Q)$, 
\begin{equation}
\mu\left(\left\{ x\in Q\,:\,|T(f\chi_{Q})(x)|>\alpha\right\} \right)\leq A\left(\frac{\|f\|_{L^{\infty}(Q)}}{\alpha}\right)^{\gamma}\mu(Q),\qquad(\alpha>0)\label{eq:HipLemma}
\end{equation}
then there exists $c>1$ such that for any $f\in L_{c}^{\infty}$
and for every ball $B$ 
\[
\mu\left(\left\{ x\in B\,:\,|T(f\chi_{B})(x)|>\alpha\langle|f|\rangle_{B}\right\} \right)\leq Ac\frac{1}{\alpha^{\frac{\gamma}{1+\gamma}}}\mu(B)\qquad(\alpha>0).
\]
\end{enumerate}
\end{lem}

\begin{proof}
We begin settling the first item. Assume that for a ball $B$ we have
that $\int_{B}f=0$ and that $f$ is supported on $B$. Note that
then 
\begin{align*}
\int_{X\setminus kB}|Tf(x)|d\mu(x) & =\int_{X\setminus kB}\left|\int_{B}f(y)K(x,y)dy\right|d\mu\\
 & =\int_{X\setminus kB}\left|\int_{B}f(y)(K(x,y)-K(x,c_{B}))d\mu(y)\right|d\mu(x)\\
 & \leq\int_{B}\left|f(y)\right|\sum_{j=1}^{\infty}\mu(2^{j}kB)\langle(K^{*}(y,\cdot)-K^{*}(c_{B},\cdot))\chi_{2^{j}kB\setminus2^{j-1}kB}\rangle_{1,2^{j}kB}d\mu(y)\\
 & \leq c\int_{B}|f(y)|d\mu(y).
\end{align*}

Now we deal with the second item. By Proposition \ref{proposition:dyadicsystem}
there exists $\delta\geq1$, a dyadic system $\mathcal{D}_{i}$ and
a cube $Q_{i}\in\mathcal{D}_{i}$ such that $B\subseteq Q_{i}$ and
$diam(Q_{i})\leq\delta r_{B}$.

Let $\tilde{f}=f\chi_{B}$. By the local Calderón-Zygmund decomposition
we shall take the maximal cubes $\lbrace P_{j}\rbrace\subset\mathcal{D}(Q_{i})$
such that 
\[
\left\{ M^{\mathcal{D}}(\tilde{f})>\alpha^{\frac{\gamma}{1+\gamma}}\langle|f|\rangle_{B}\right\} =\bigcup_{j}P_{j}.
\]
Note that then we have that
\[
\alpha^{\frac{\gamma}{1+\gamma}}\langle|f|\rangle_{B}<\langle|\tilde{f}|\rangle_{P_{j}}\leq c\alpha^{\frac{\gamma}{1+\gamma}}\langle|f|\rangle_{B}.
\]
Now we let $b=\sum_{j}b_{j}$ with $b_{j}=(\tilde{f}-\langle\tilde{f}\rangle_{P_{j}})\chi_{P_{j}}$,
and $g=\tilde{f}-b$. Note that by Lebesgue differentiation theorem
\[
|g(x)|=|\tilde{f}(x)|=\lim_{Q\rightarrow\{x\}}\dashint_{Q}|\tilde{f}(y)|d\mu(y)\leq M^{\mathcal{D}}\tilde{f}(x)\leq\alpha^{\frac{\gamma}{1+\gamma}}\langle|f|\rangle_{B}\,
\]
a.e. $x\in Q\setminus\bigcup_{j}P_{j}$. Also, if $x\in P_{j}$ we
have that 
\[
|g(x)|=\langle|\tilde{f}|\rangle_{P_{j}}\leq c\alpha^{\frac{\gamma}{1+\gamma}}\langle|f|\rangle_{B}
\]
and then for every $P_{j},$
\[
\|g\|_{L^{\infty}(Q_{i})}\leq\|g\|_{L^{\infty}}\lesssim\alpha^{\frac{\gamma}{1+\gamma}}\langle|f|\rangle_{B}.
\]
Relying upon the functions defined above, for some $k\geq2c_{d}$,
we have that
\begin{align*}
 & \mu\left(\left\{ x\in B:T(f\chi_{B})>\alpha\langle|f|\rangle_{B}\right\} \right)\\
= & \mu\left(\left\{ x\in B:T(\tilde{f}\chi_{Q_{i}})>\alpha\langle|f|\rangle_{B}\right\} \right)\\
\leq & \mu\left\{ x\in Q_{i}\,:\,|T(g\chi_{Q_{i}})(x)|>\frac{1}{2}\alpha\langle|f|\rangle_{B}\right\} \\
+ & \mu\left\{ x\in Q_{i}\setminus\bigcup_{j}kP_{j}\,:\,|T(b\chi_{B})(x)|>\frac{1}{2}\alpha\langle|f|\rangle_{B}\right\} +\mu\left(\bigcup kP_{j}\right)\\
= & I+II+III\,.
\end{align*}
For $I$, by (\ref{eq:HipLemma}) we have that
\begin{equation}
\begin{split} & \mu\left\{ x\in Q_{i}\,:\,|T(g\chi_{Q_{i}})(x)|>\frac{1}{2}\alpha\langle|f|\rangle_{B}\right\} \\
= & \mu\left\{ x\in Q_{i}\,:\,\left|T\left(\dfrac{g\chi_{Q_{i}}}{\langle|f|\rangle_{B}}\right)(x)\right|>\frac{\alpha}{2}\right\} \leq A\left(\frac{2\|g\|_{L^{\infty}(Q_{i})}}{\alpha\langle|f|\rangle_{B}}\right)^{\gamma}\mu(Q_{i})\\
\lesssim & A\left(\frac{2\alpha^{\frac{\gamma}{1+\gamma}}\langle|f|\rangle_{B}}{\alpha\langle|f|\rangle_{B}}\right)^{\gamma}\mu(B)=A\left(\frac{2}{\alpha^{\frac{1}{1+\gamma}}}\right)^{\gamma}\mu(B)=A\frac{2^{\gamma}}{\alpha^{\frac{\gamma}{1+\gamma}}}\mu(B)\,.
\end{split}
\label{6.2}
\end{equation}
For $II$,
\begin{equation}
\begin{split} & \mu\left\{ x\in Q_{i}\setminus\bigcup_{j}kP_{j}\,:\,|T(b\chi_{Q_{i}})(x)|>\frac{1}{2}\alpha\langle|f|\rangle_{B}\right\} \\
\leq & \dfrac{2}{\alpha\langle|f|\rangle_{B}}\int_{Q_{i}\setminus\bigcup_{j}kP_{j}}|T(b\chi_{Q_{i}})(x)|d\mu(x)=\dfrac{2}{\alpha\langle|f|\rangle_{B}}\int_{Q_{i}\setminus\bigcup P_{j}}|T(\sum_{j}b_{j})(x)|d\mu(x)\\
\leq & \dfrac{2}{\alpha\langle|f|\rangle_{B}}\sum_{j}\int_{Q_{i}\setminus\bigcup_{j}kP_{j}}|T(b_{j})(x)|d\mu(x)\\
= & \dfrac{2}{\alpha\langle|f|\rangle_{B}}\sum_{j}\int_{Q_{i}\setminus\bigcup_{j}kP_{j}}|T((\tilde{f}-\langle\tilde{f}\rangle_{P_{j}})\chi_{P_{j}})(x)|d\mu(x)\\
\leq & \dfrac{2}{\alpha\langle|f|\rangle_{B}}\sum_{j}\int_{Q_{i}\setminus kP_{j}}|T((\tilde{f}-\langle\tilde{f}\rangle_{P_{j}})\chi_{P_{j}})(x)|d\mu(x)=\circledast
\end{split}
\label{6.3}
\end{equation}
Observe that 
\[
\int_{1P_{j}}(\tilde{f}-\langle\tilde{f}\rangle_{P_{j}})\chi_{P_{j}}=\int_{P_{j}}(\tilde{f}-\langle\tilde{f}\rangle_{P_{j}})\chi_{P_{j}}=0.
\]
Then, by \eqref{eq:Part1}, we have that
\begin{equation}
\begin{split}\circledast & \leq\dfrac{2}{\alpha\langle|f|\rangle_{B}}\sum_{j}c\int_{P_{j}}|\tilde{f}-\langle\tilde{f}\rangle_{P_{j}}|d\mu\leq\dfrac{2}{\alpha\langle|f|\rangle_{B}}\sum_{j}c\left(\int_{P_{j}}|\tilde{f}|+\int_{P_{j}}|\langle\tilde{f}\rangle_{P_{j}}|\right)\\
 & =\dfrac{2}{\alpha\langle|f|\rangle_{B}}\sum_{j}c\left(\mu(P_{j})\langle|\tilde{f}|\rangle_{P_{j}}+\mu(P_{j})|\langle\tilde{f}\rangle_{P_{j}}|\right)\leq\dfrac{2}{\alpha\langle|f|\rangle_{B}}\sum_{j}c\left(2\mu(P_{j})\langle|\tilde{f}|\rangle_{P_{j}}\right)\\
 & \leq\dfrac{4c}{\alpha\langle|f|\rangle_{B}}\sum_{j}\mu(P_{j})\alpha^{\frac{\gamma}{1+\gamma}}\langle|f|\rangle_{B}=\dfrac{4c}{\alpha^{\frac{1}{1+\gamma}}}\sum_{j}\mu(P_{j})\leq\dfrac{4c}{\alpha^{\frac{1}{1+\gamma}}}\mu(Q_{i})\leq\dfrac{4c}{\alpha^{\frac{1}{1+\gamma}}}c_{\mu}^{\log_{2}\delta}\mu(B)
\end{split}
\label{6.4}
\end{equation}
For $III$ we have
\begin{equation}
\begin{split}\mu\left(\bigcup_{j}kP_{j}\right) & \lesssim\sum_{j}\mu(P_{j})\leq\sum_{j}\frac{c}{\alpha^{\frac{\gamma}{1+\gamma}}\langle|f|\rangle_{B}}\int_{P_{j}}|\tilde{f}|\\
 & =\sum_{j}\frac{c}{\alpha^{\frac{\gamma}{1+\gamma}}\langle|f|\rangle_{B}}\int_{P_{j}}|f\chi_{B}|\leq\frac{c}{\alpha^{\frac{\gamma}{1+\gamma}}\langle|f|\rangle_{B}}\int_{B}|f|=\frac{c}{\alpha^{\frac{\gamma}{1+\gamma}}}\mu(B)\,.
\end{split}
\label{6.5}
\end{equation}
Gathering  \eqref{6.2}, \eqref{6.3}, \eqref{6.4} and \eqref{6.5}
we have: 
\[
\begin{split}\mu\left\{ x\in B\,:\,|T(f\chi_{B})(x)|>\alpha\right\}  & \leq\left(A\frac{2^{\gamma}}{\alpha^{\frac{\gamma}{1+\gamma}}}\mu(B)+\dfrac{4c}{\alpha^{\frac{1}{1+\gamma}}}c_{\mu}^{\log_{2}\delta}\mu(B)+\frac{c}{\alpha^{\frac{\gamma}{1+\gamma}}}\mu(B)\right)\\
 & =\frac{C'}{\alpha^{\frac{\gamma}{1+\gamma}}}\mu(B)\,.
\end{split}
\]
\end{proof}
\begin{lem}
\label{lem:Mtr'r'}Let $r>1$ and $K\in H_{r}$ and $\alpha\geq\dfrac{3c_{d}^{2}}{\delta}\geq3$
. Then $\mathcal{M}_{T,\alpha}^{\#}$ is weak -type. 
\end{lem}

\begin{proof}
Let us fix $x\in X$ and let $B$ a ball such that $x\in B$. Let
$y,z\in B$, then we also have that $y,z\in\frac{1}{2c_{d}}\alpha B$ and
\begin{align*}
 & \left|T(f\chi_{X\setminus\alpha B})(y)-T(f\chi_{X\setminus\alpha B})(z)\right|\\
= & \left|\int_{X\setminus\alpha B}f(t)(K(y,t)-K(z,t))d\mu(t)\right|\leq\int_{X\setminus\alpha B}|f(t)|\left|K(y,t)-K(z,t)\right|d\mu(t)\\
\leq & \int_{X\setminus3B}|f(t)|\left|K(y,t)-K(z,t)\right|d\mu(t)\leq\sum_{k=1}^{\infty}\int_{2^{k}3B\setminus2^{k-1}3B}|f(t)|\left|K(y,t)-K(z,t)\right|d\mu(t)\\
\leq & \sum_{k=1}^{\infty}\left(\int_{2^{k}3B\setminus2^{k-1}3B}|f(t)|^{r'}d\mu(t)\right)^{\frac{1}{r'}}\left(\int_{2^{k}3B\setminus2^{k-1}3B}|K(y,\cdot)-K(z,\cdot)|^{r}\right)^{\frac{1}{r}}\\
= & \sum_{k=1}^{\infty}\left(\frac{1}{\mu(2^{k}3B)}\int_{2^{k}3B\setminus2^{k-1}3B}|f(t)|^{r'}d\mu(t)\right)^{\frac{1}{r'}}\mu(2^{k}3B)^{\frac{1}{r'}}\left(\int_{2^{k}3B\setminus2^{k-1}3B}|K(y,\cdot)-K(z,\cdot)|^{r}\right)^{\frac{1}{r}}\\
\leq & \sum_{k=1}^{\infty}\left(\frac{1}{\mu(2^{k}3B)}\int_{2^{k}3B}|f(t)|^{r'}d\mu(t)\right)^{\frac{1}{r'}}\mu(2^{k}3B)^{\frac{1}{r'}}\left(\int_{2^{k}3B\setminus2^{k-1}3B}|K(y,\cdot)-K(z,\cdot)|^{r}\right)^{\frac{1}{r}}\\
\leq & M_{r'}f(x)\sum_{k=1}^{\infty}\mu(2^{k}3B)\left(\frac{1}{\mu(2^{k}3B)}\int_{2^{k}3B\setminus2^{k-1}3B}|K(y,\cdot)-K(z,\cdot)|^{r}\right)^{\frac{1}{r}}\\
\lesssim & M_{r'}f(x).
\end{align*}
Consequently
\[
\mathcal{M}_{T,\alpha}^{\#}f(x)\lesssim M_{r'}(f)(x)
\]
and from this the desired conclusion readily follows.
\end{proof}
\subsection{Proof of Theorem \ref{thm:sparseT1}}

Let $\mathcal{D}$ a dyadic system and $Q\in\mathcal{D} $. Let $s=sgn(Tf\chi_Q)$. Observe that by  \eqref{eq:HipT1} there exists $c\textgreater 0$ such that:
\[
\begin{split}
	\int_Q|T(f\chi_Q)|&\leq \int_Q|fT^*(\chi_{\{s=1\}})|+\int_Q|fT^*(\chi_{\{s=-1\}})|\\
	&\leq \|f\|_{L^\infty(Q)}\int_{1Q}|T^*(\chi_{\{s=1\}})|+\|f\|_{L^\infty(Q)}\int_{1Q}|T^*(\chi_{\{s=-1\}})|\\
	&\leq 2c\|f\|_{L^\infty(Q)}\mu(1Q)\leq 2cc' \|f\|_{L^\infty(Q)}\mu(Q)
\end{split}
\]
and therefore
\[\begin{split}
	\mu\left(\left\{ x\in Q\,:\,|T(f\chi_{Q})(x)|>\alpha\right\} \right)&\leq \dfrac{1}{\alpha}\int_{Q}|T(f\chi_{Q})(x)|\\
	&\leq \dfrac{2cc'\|f\|_{L^\infty(Q)}\mu(Q)}{\alpha}.
\end{split}	
\]
Now we apply Lemma \ref{lem:Let:Aux}  with $A=2cc'$ and $\gamma=1$. There exists $c'>1$ such that
\[
\mu\left(\left\{ x\in B\,:\,|T(f\chi_{B})(x)|>\alpha\langle |f|\rangle_{B}\right\} \right)\leq \frac{2cc'}{\alpha^{\frac{1}{2}}}\mu(B) \quad \alpha > 0
\]
By Lemma \ref{lem:Mtr'r'} we have that 
\[
\mu\left(\left\{ x\in B\,:\,|\mathcal{M}_{T,\alpha}^{\#}(f\chi_{B})(x)|>\alpha\langle|f|\rangle_{r',B}\right\} \right)\leq \left(\dfrac{C}{\alpha\langle |f|\rangle_{r',B}}\right)^{r'}\|f\|_{L^{r'}(B)}^{r'}=\dfrac{C'}{\alpha}\mu(B).
\]
Define $\psi(\lambda)=\left(\dfrac{2cc'}{\lambda}\right)^2$ and $\phi(\lambda)=\dfrac{C'}{\lambda}$. We have that $\psi\left(\dfrac{2cc'}{\alpha^{1/2}}\right)=\phi\left(\dfrac{C'}{\alpha}\right)=\alpha$.  A direct application of Theorem \ref{thm:ThmSparse}
ends the proof.

\section{Proof of the sparse domination result}\label{sec:SparseProof}

%

\subsection{Lemmatta}

In this section we gather somme results that will be useful to settle our
convex body domination result.
\begin{lem}
\label{lem:BelongTest}Let $\{\vec{e_{i}}\}_{i=1}^{n}$ be the principal axes
of the John ellipsoid of $\langle\langle \vec{f}\rangle\rangle_{r,Q}$, where $0<\mu(Q)<\infty$.
Assume that for every $i=1,\dots,n$
\end{lem}
\[
|\langle \vec{e_{i}},\vec{x}\rangle |\leq \dfrac{1}{n}\left(\frac{1}{\mu(Q)}\int_{Q}|\langle \vec{e_{i}},\vec{f}\rangle|^{r}\right)^{\frac{1}{r}}
\]
then $\vec{x}\in\langle\langle \vec{f}\rangle\rangle_{r,Q}$.
\begin{proof}
For some $0\leq\sigma_{n}\leq..\leq\sigma_{1}$ the John Elipsoid
of $\langle\langle \vec{f}\rangle\rangle_{r,Q}$ consists of $\vec{y}\in\mathbb{R}^{n}$
such that
\[
\sum_{i=1}^{n}\left(\frac{|\langle \vec{e_{i}},\vec{y}\rangle|}{\sigma_{i}}\right)^{2}\leq1.
\]
Note that there exists $g_{j}$ with $\langle g_{j}\rangle_{r',Q}\leq 1$
such that 
\[
\frac{1}{\mu(Q)}\int_{Q}g_{j}\langle \vec{e_{j}},\vec{f}\rangle=\left(\frac{1}{\mu(Q)}\int_{Q}|\langle \vec{e_{i}},\vec{f}\rangle|^{r}\right)^{\frac{1}{r}}
\]
furthermore
\[
\vec{y}=\frac{1}{\mu(Q)}\int_{Q}g_{j}\vec{f}\in\langle\langle \vec{f}\rangle\rangle_{r,Q}\subset\sqrt{n}E
\]
and then $\frac{1}{\sqrt{n}}\vec{y}\in E$. Now we observe that
\begin{align*}
1 & \geq\sum_{i=1}^{n}\left(\frac{|\langle \vec{e_{i}},\vec{y}\rangle|}{\sqrt{n}\sigma_{i}}\right)^{2}\geq\left(\frac{|\langle \vec{e_{j}},\vec{y}\rangle|}{\sqrt{n}\sigma_{j}}\right)^{2}\\
 & =\left(\frac{\left|\left\langle \vec{e_{j}},\frac{1}{\mu(Q)}\int_{Q}g_{j}\vec{f}\right\rangle \right|}{\sqrt{n}\sigma_{j}}\right)^{2}=\left(\frac{\left|\frac{1}{\mu(Q)}\int_{Q}g_{j}\langle \vec{e_{j}},\vec{f}\rangle\right|}{\sqrt{n}\sigma_{j}}\right)^{2}\\
 & =\left(\frac{\left(\frac{1}{\mu(Q)}\int_{Q}|\langle \vec{e_{j}},\vec{f}\rangle|^{r}\right)^{\frac{1}{r}}}{\sqrt{n}\sigma_{j}}\right)^{2}.
\end{align*}
Consequently
\[
\left(\frac{\left(\frac{1}{\mu(Q)}\int_{Q}|\langle \vec{e_{j}},\vec{f}\rangle|^{r}\right)^{\frac{1}{r}}}{\sqrt{n}\sigma_{j}}\right)^{2}\leq 1
\]
for each $j=1,\dots,n$. 

Having that inequality at our disposal we can establish the criteria in the statement of the Lemma. Indeed, if for some $\vec{x}\in\mathbb{R}^{n}$ we have that for each for each $i$,
\[
|\langle \vec{e_{i}},\vec{x}\rangle|\leq\frac{1}{n}\left(\frac{1}{\mu(Q)}\int_{Q}|\langle \vec{e_{i}},\vec{f}\rangle|^{r}\right)^{\frac{1}{r}},
\]
then
\begin{align*}
\sum_{i=1}^{n}\left(\frac{|\langle \vec{e_{i}},\vec{x}\rangle|}{\sigma_{i}}\right)^{2} & \leq\sum_{i=1}^{n}\left(\frac{\left(\frac{1}{\mu(Q)}\int_{Q}|\langle \vec{e_{i}},\vec{f}\rangle|^{r}\right)^{\frac{1}{r}}}{n\sigma_{i}}\right)^{2}=\sum_{i=1}^{n}\left(\frac{1}{\sqrt{n}}\right)^{2}\left(\frac{\left(\frac{1}{\mu(Q)}\int_{Q}|\langle \vec{e_{i}},\vec{f}\rangle|^{r}\right)^{\frac{1}{r}}}{\sqrt{n}\sigma_{i}}\right)^{2}\\
 & \leq\sum_{i=1}^{n}\left(\frac{1}{\sqrt{n}}\right)^{2}=\sum_{i=1}^{n}\frac{1}{n}=1
\end{align*}
and consequently $\vec{x}\in\langle\langle\vec{f}\rangle\rangle_{r,Q}$ as we wanted to show.
\end{proof}
Generalizing  \cite[Lemma 2.3.9]{H} we can obtain the following Lemma, that will allow us to deal in a convenient way with the convex bodies involved in the convex body domination.
\begin{lem}
\label{lem:RepkQ}Let $\vec{f}\in L^{s}(Q;\mathbb{R}^{n})$, where
$Q\subset X$ and $0<\mu(Q)<\infty$,  and let
$\vec{g}:X\rightarrow\mathbb{R}^{n}$ be a measurable function such that
$\vec{g}(x)\in\langle\langle\vec{f}\rangle\rangle_{s,Q}$ for almost every
$x\in X$. Then there exists some measurable function $k(x,y)$ such
that for almost every $x\in X$
\[
\vec{g}(x)=\frac{1}{\mu(Q)}\int_{Q}k(x,y)\vec{f}(y)d\mu(y)
\]
where $\left\Vert k(x,y)\right\Vert _{L^{(s',\infty)}\left[(Q,\frac{d\mu(y)}{\mu(Q)}),(X,d\mu(x))\right]}=\left\Vert \left\Vert k(x,y)\right\Vert _{L^{s'}\left(Q,\frac{d\mu(y)}{\mu(Q)}\right)}\right\Vert _{L^{\infty}(X,d\mu(x))}\leq1$. 
\end{lem}

\begin{proof}
We split the proof in cases 
\end{proof}
\begin{enumerate}
\item Suppose $\vec{g}$ is simple, namely $\vec{g}=\sum_{j=1}^{J}\vec{a_{j}}\chi_{A_{j}}$
where the sets $A_{j}$ are measurable and disjoint and have finite
measure. Note that since $\vec{a_{j}}\in\langle\langle\vec{f}\rangle\rangle_{s,Q}$
we have that 
\[
\vec{a_{j}}=\frac{1}{\mu(Q)}\int_{Q}\phi_{j}(y)\vec{f}(y)d\mu(y)
\]
for some $\phi_{j}$ with $\|\phi_{j}\|_{L^{s'}\left(\frac{d\mu}{\mu(Q)}\right)}\leq1$.
Note that then 
\[
\vec{g}(x)=\frac{1}{\mu(Q)}\int_{Q}\sum_{j=1}^{J}\chi_{A_{j}}(x)\phi_{j}(y)\vec{f}(y)d\mu(y)
\]
and also
\[
\left\Vert \left\Vert \sum_{j=1}^{J}\chi_{A_{j}}(x)\phi_{j}(y)\right\Vert _{L^{s'}\left(Q,\frac{d\mu(y)}{\mu(Q)}\right)}\right\Vert _{L^{\infty}(X,d\mu(x))}\leq1.
\]
Indeed 
\begin{align*}
&\left\Vert \left\Vert \sum_{j=1}^{J}\chi_{A_{j}}(x)\phi_{j}(y)\right\Vert _{L^{s'}\left(Q,\frac{d\mu(y)}{\mu(Q)}\right)}\right\Vert _{L^{\infty}(X,d\mu(x))}\\ 
&\leq\left\Vert \left(\frac{1}{\mu(Q)}\int_{Q}\left(\left|\sum_{j=1}^{J}\chi_{A_{j}}(x)\phi_{j}(y)\right|\right)^{s'}d\mu(y)\right)^{\frac{1}{s'}}\right\Vert _{L^{\infty}(X,d\mu(x))}\\
 & \leq\left\Vert \left(\frac{1}{\mu(Q)}\int_{Q}\left(\sum_{j=1}^{J}\chi_{A_{j}}(x)\left|\phi_{j}(y)\right|\right)^{s'}d\mu(y)\right)^{\frac{1}{s'}}\right\Vert _{L^{\infty}(X,d\mu(x))}\\
 & \leq\left\Vert \sum_{j=1}^{J}\chi_{A_{j}}(x)\left(\frac{1}{\mu(Q)}\int_{Q}\left(\left|\phi_{j}(y)\right|\right)^{s'}d\mu(y)\right)^{\frac{1}{s'}}\right\Vert _{L^{\infty}(X,d\mu(x))}\\
 & =\left\Vert \sum_{j=1}^{J}\chi_{A_{j}}(x)\right\Vert _{L^{\infty}(X,d\mu(x))}=1
\end{align*}
and hence we are done in this case choosing $k(x,y)=\sum_{j=1}^{J}\chi_{A_{j}}(x)\phi_{j}(y)$.
\item Approximation by simple functions. Let $g$ be a general function
as in the assumptions and let $\varepsilon>0$. Since $\langle\langle\vec{f}\rangle\rangle_{s,Q}$
is a compact set, there exists a $\varepsilon$-net $(\vec{y_{j}})_{j=1}^{J}$
in $\langle\langle\vec{f}\rangle\rangle_{s,Q}$ namely, for every
$\vec{y}\in\langle\langle\vec{f}\rangle\rangle_{s,Q}$ there exists $\vec{y_{j}}$
such that 
\[
|\vec{y}-\vec{y_{j}}|\leq\varepsilon.
\]
We define the measurable sets 
\[
A_{j}:=\{x\in X\,:\,|\vec{g}(x)-\vec{y_{j}}|\leq\varepsilon\quad\text{and}\quad|\vec{g}(x)-\vec{y_{i}}|>\varepsilon\quad\forall i=1,\dots,j-1\}
\]
and let 
\[
\vec{g_{\varepsilon}}(x):=\sum_{j=1}^{J}\vec{y_{j}}\chi_{A_{j}}.
\]
Note that 
\begin{align*}
\|\vec{g}(x)-\vec{g}_{\varepsilon}(x)\|_{L^{\infty}(X)}\leq\varepsilon
\end{align*}
Furthermore, $\vec{g_{\varepsilon}}\in\langle\langle \vec{f}\rangle\rangle_{s,Q}$,
since each $\vec{y_{j}}\in\langle\langle \vec{f}\rangle\rangle_{s,Q}$ and the
sets $A_{j}$ are pairwise disjoint. 
\item By the preceding parts we can find simple functions $\vec{g_{j}}\rightarrow \vec{g}$
in $L^{\infty}(X)$ and kernels $k_{j}(x,y)$ with 
\[
\left\Vert \left\Vert k_{j}(x,y)\right\Vert _{L^{s'}\left(Q,\frac{d\mu(y)}{\mu(Q)}\right)}\right\Vert _{L^{\infty}(X,d\mu(x))}\leq1
\]
and
\[
\vec{g_{j}}(x)=\frac{1}{d\mu(Q)}\int_{Q}k_{j}(x,y)\vec{f}(y)d\mu(y).
\]
Note that as it was established in \cite[Theorem 1]{BP} the dual space of
\[E=L^{(s,1)}\left[(Q,\frac{d\mu(y)}{\mu(Q)}),(X,d\mu(x))\right]\]
is precisely \[L^{(s',\infty)}\left[(Q,\frac{d\mu(y)}{\mu(Q)}),(X,d\mu(x))\right]\]
and any functional $\Lambda$ in the dual of $E$ can be expressed
in the as 
\[
\Lambda(\rho)=\int\rho(x,y)h(x,y)d\mu(x,y)
\]
for some $h\in L^{(s',\infty)}\left[(Q,\frac{d\mu(y)}{\mu(Q)}),(X,d\mu(x))\right]$.
By the Banach-Alaoglu theorem applied to $E$ we have then that the
unit ball $L^{(s',\infty)}\left[(Q,\frac{d\mu(y)}{\mu(Q)}),(X,d\mu(x))\right]$
is weak-{*} sequential compact. In particular this yields that given
the sequence $\{k_{j}(x,y)\}$ that is contained in the closed unit
ball $L^{(s',\infty)}\left[(Q,\frac{d\mu(y)}{\mu(Q)}),(X,d\mu(x))\right]$
there exists a subsequence $\{k_{N_{j}}\}$ and $k\in L^{(s',\infty)}\left[(Q,\frac{d\mu(y)}{\mu(Q)}),(X,d\mu(x))\right]$
such that 
\[
\lim_{j\rightarrow\infty}\int k_{N_{j}}(x,y)\eta(x,y)d\mu(x,y)=\int k(x,y)\eta(x,y)d\mu(x,y)
\]
for every $\eta(x,y)\in L^{(s,1)}\left[(Q,\frac{d\mu(y)}{\mu(Q)}),(X,d\mu(x))\right]$.
Now we observe that for every $\psi\in L^{1}(X)$ we have that $\psi(x)\vec{f}(y)\in L^{(s,1)}\left[(Q,\frac{d\mu(y)}{\mu(Q)}),(X,d\mu(x))\right]$.
Taking that into account, 
\begin{align*}
 & \int_{X}\vec{g}(x)\psi(x)dx\\
= & \int_{X}\lim_{j}\vec{g_{N_{j}}}(x)\psi(x)d\mu(x)=\int_{X}\lim_{j}\frac{1}{\mu(Q)}\int_{Q}k_{N_{j}}(x,y)\vec{f}(y)dy\psi(x)d\mu(x)\\
= & \int_{X}\lim_{j}\frac{1}{\mu(Q)}\int_{Q}k_{N_{j}}(x,y)\vec{f}(y)\psi(x)dyd\mu(x)\\
= & \int_{X}\frac{1}{\mu(Q)}\int_{Q}k(x,y)\vec{f}(y)dy\psi(x)d\mu(x)=\int_{X}\vec{G}(x)\psi(x)d\mu(x).
\end{align*}
Note that since the preceding identity holds for every $\psi\in L^{1}(X)$
we have that, since Lebesgue differentiation property holds, $\vec{g}(x)=\vec{G}(x)$
$\mu$-a.e. In other words, 
\[
\vec{g}(x)=\frac{1}{\mu(Q)}\int_{Q}k(x,y)\vec{f}(y)dy
\]
for some $k(x,y)\in L^{(s',\infty)}\left[(Q,\frac{d\mu(y)}{\mu(Q)}),(X,d\mu(x))\right]$
as we wanted to show. 
\end{enumerate}

\subsection{Proof of Theorem \ref{thm:ThmSparse}}

We will rely upon the following Lemma. 
\begin{lem}
\label{lem:It}Let $(X,d,\mu)$ be a space of homogeneous type and
$\mathcal{D}$ a dyadic system with parameters $c_{0}$, $C_{0}$
and $\delta$. Let us fix $\alpha\geq\frac{3c_{d}^{2}}{\delta}$ and
let $\vec{f}:X\rightarrow\mathbb{R}^{n}$ be a boundedly supported
function such that $|\vec{f}|\in L^{s}(X)$. Let $1\le q,r<\infty$
and $s=\max(q,r)$, and assume that there exist non-increasing functions
$\psi$ and $\phi$ such that for every ball $B$ and every boundedly supported function $g\in L^{s}(X,\mathbb{R})$ 
\[
\mu\left(\{x\in B:|T(g\chi_{B})(x)|>\psi(\rho)\langle g\rangle_{q,B}\}\right)\le\rho\mu(B)\quad(0<\rho<1)
\]
and 
\[
\mu\left(\{x\in B:\mathcal{M}_{T,\alpha}^{\#}(g\chi_{B})(x)>\phi(\rho)\langle g\rangle_{r,B}\}\right)\le\rho\mu(B)\quad(0<\rho<1)
\]
Then for every cube $Q$ there exist disjoint subcubes $P_{j}\in\mathcal{D}(Q)$ such
that 
\[
\sum_{j}\mu(P_{j})\leq\frac{1}{2}\mu(Q)
\]
and 
\[T(\vec{f}\chi_{\alpha Q})(x)\chi_{Q}-\sum_{j}T(\vec{f}\chi_{\alpha P_{j}})(x)\chi_{P_{j}}(x)\in\text{\ensuremath{\kappa_{n,\rho,s}}}\langle\langle\vec{f}\rangle\rangle_{s,\alpha Q}\chi_{Q}\]
where $\kappa_{n,\rho,s}=n\left(3B(\rho)+B(\rho)\left(\rho^{-1}\|M\|_{L^{s}\rightarrow L^{s,\infty}}\right)^{\frac{1}{s}}\right)$
and $B(\rho)=\psi(\rho)+\phi(\rho)$ with $\rho$ a constant depending
on $n$ and on the parameters defining $\mathcal{D}$. Furthermore,
there exists a function $k_{Q}$ with 
\[
\left\Vert k_{Q}(x,y)\right\Vert _{L^{(s',\infty)}\left[(\alpha Q,\frac{d\mu(y)}{\mu(\alpha Q)}),( Q,d\mu(x))\right]}=\left\Vert \left\Vert k_{Q}(x,y)\right\Vert _{L^{s'}\left(\alpha Q,\frac{d\mu(y)}{\mu(\alpha Q)}\right)}\right\Vert _{L^{\infty}( Q,d\mu(x))}\leq1
\]
such that 
\[ T(\vec{f}\chi_{\alpha Q})(x)\chi_{Q}(x)-\sum_{j}T(\vec{f}\chi_{\alpha P_{j}})(x)\chi_{P_{j}}(x)=\kappa_{n,\rho,s}\frac{1}{\mu(\alpha Q)}\int_{\alpha Q}k_{Q}(x,y)\vec{f}(y)d\mu(y).\]
\end{lem}

\begin{proof}
By the doubling condition of the measure there exists $c_{1}$ such
that $\mu(\alpha P)\leq c_{1}\mu(P)$ for any cube $P$. Now we observe
that for any disjoint family $\{P_{j}\}\subset\mathcal{D}(Q)$ 
\[
T(\vec{f}\chi_{\alpha Q})(x)\chi_{Q}=T(\vec{f}\chi_{\alpha Q})(x)\chi_{Q\setminus\cup P_{j}}+\sum_{j}T(\vec{f}\chi_{\alpha Q\setminus\alpha P_{j}})(x)\chi_{P_{j}}+\sum_{j}T(\vec{f}\chi_{\alpha P_{j}})(x)\chi_{P_{j}}(x)
\]
and we have that 
\[
T(\vec{f}\chi_{\alpha Q})(x)\chi_{Q}-\sum_{j}T(\vec{f}\chi_{\alpha P_{j}})(x)\chi_{P_{j}}(x)=T(\vec{f}\chi_{\alpha Q})(x)\chi_{Q\setminus\cup P_{j}}+\sum_{j}T(\vec{f}\chi_{\alpha Q\setminus\alpha P_{j}})(x)\chi_{P_{j}}.
\]
We claim that there exists a disjoint family $\{P_{j}\}\subset\mathcal{D}(Q)$
with 
\[
\sum_{j}\mu(P_{j})\leq\frac{1}{2}\mu(Q)
\]
and 
\begin{equation}
T(\vec{f}\chi_{\alpha Q})(x)\chi_{Q\setminus\cup P_{j}}+\sum_{j}T(\vec{f}\chi_{\alpha Q\setminus\alpha P_{j}})(x)\chi_{P_{j}}\in\kappa_{n,\rho,s}\langle\langle\vec{f}\rangle\rangle_{s,\alpha Q}\chi_{Q}\label{eq:Claim}
\end{equation}
Actually by Lemma \ref{lem:BelongTest} it will be enough to show
that for each $i=1,\dots,n$, 
\begin{equation}
\left|\left\langle T(\vec{f}\chi_{\alpha Q})(x)\chi_{Q\setminus\cup P_{j}}+\sum_{j}T(\vec{f}\chi_{\alpha Q\setminus\alpha P_{j}})(x)\chi_{P_{j}},\vec{e_{i}}\right\rangle \right|\leq\kappa_{n,\rho,s}\frac{1}{n}\left(\frac{1}{\mu(\alpha Q)}\int_{\alpha Q}|\langle\vec{f},\vec{e_{i}}\rangle|^{s}\right)^{\frac{1}{s}}\label{eq:Claim2}
\end{equation}
where $\{\vec{e_{i}}\}_{i=1}^{n}$ are the principal axis of the John Ellipsoid
of $\kappa_{n,\rho,s}\langle\langle\vec{f}\rangle\rangle_{s,\alpha Q}$.
For $\rho\in(0,1)$ to be chosen and each $i\in\{1,\dots,n\}$ let
\[
\mathcal{\tilde{M}}_{T,i}(f)=\max\left\{ \frac{|T(\langle\vec{f}\chi_{\alpha Q},\vec{e_{i}}\rangle)|}{\left(B(\rho)\frac{1}{\mu(\alpha Q)}\int_{\alpha Q}|\langle\vec{f},\vec{e_{i}}\rangle|^{s}\right)^{\frac{1}{s}}},\frac{\mathcal{M}_{T,\alpha}^{\#}(\langle\vec{f}\chi_{\alpha Q},\vec{e_{i}}\rangle)}{\left(B(\rho)\frac{1}{\mu(\alpha Q)}\int_{\alpha Q}|\langle\vec{f},\vec{e_{i}}\rangle|^{s}\right)^{\frac{1}{s}}}\right\} 
\]
and let us define the sets 
\[
\Omega_{i}=\left\{ x\in Q\,:\,\max\left\{ \frac{M_{s}(\langle\vec{f}\chi_{\alpha Q},\vec{e_{i}}\rangle)(x)}{\left(\rho^{-1}\|M\|_{L^{s}\rightarrow L^{s,\infty}}\frac{1}{\mu(\alpha Q)}\int_{\alpha Q}|\langle\vec{f},\vec{e_{i}}\rangle|^{s}\right)^{\frac{1}{s}}},\mathcal{\tilde{M}}_{T,i}(\vec{f})\right\} >1\right\} 
\]
and $\Omega=\bigcup_{i}\Omega_{i}$. Observe that 
\begin{align*}
\begin{split}\mu(\Omega) & \leq\sum_{i=1}^{n}\mu\left(\left\{ x\in Q\,:\,\frac{M_{s}(\langle\vec{f}\chi_{\alpha Q},\vec{e_{i}}\rangle)(x)}{\left(\rho^{-1}\|M\|_{L^{s}\rightarrow L^{s,\infty}}\frac{1}{\mu(\alpha Q)}\int_{\alpha Q}|\langle\vec{f},\vec{e_{i}}\rangle|^{s}\right)^{\frac{1}{s}}}>1\right\} \right)\\
 & +\sum_{i=1}^{n}\mu\left(\left\{ x\in Q\,:\,\frac{|T(\langle\vec{f}\chi_{\alpha Q},\vec{e_{i}}\rangle)(x)|}{B(\rho)\left(\frac{1}{\mu(\alpha Q)}\int_{\alpha Q}|\langle\vec{f},\vec{e_{i}}\rangle|^{s}\right)^{\frac{1}{s}}}>1\right\} \right)\\
 & +\sum_{i=1}^{n}\mu\left(\left\{ x\in Q\,:\,\frac{\mathcal{M}_{T,\alpha}^{\#}(\langle\vec{f}\chi_{\alpha Q},\vec{e_{i}}\rangle)(x)}{B(\rho)\left(\frac{1}{\mu(\alpha Q)}\int_{\alpha Q}|\langle\vec{f},\vec{e_{i}}\rangle|^{s}\right)^{\frac{1}{s}}}>1\right\} \right)\\
 & \leq\left(\sum_{i=1}^{n}3\rho\right)\mu(\alpha Q)\leq3n\rho c_{1}\mu(Q).
\end{split}
\end{align*}
Now we take the local Calder\'on-Zygmund decomposition (see \cite[Lemma 4.5]{FN})
of 
\[
\Omega_{c_{2}}=\left\{ s\in Q\,:\,M^{\mathcal{D}(Q)}(\chi_{\Omega})>\frac{1}{c_{2}}\right\}. 
\]
By the weak type $(1,1)$  of the operator $M^{\mathcal{D}(Q)}$ we have that for a suitable choice of $c_{2}$, $\Omega_{c_{2}}$ is
a proper subset of $Q$ and that there exists a family $\{P_{j}\}\subset\mathcal{D}(Q)$
such that $\Omega_{c_{2}}=\bigcup P_{j}$ and 
\[
\frac{1}{c_{2}}\leq\frac{\mu(P_{j}\cap\Omega)}{\mu(P_{j})}\leq\frac{1}{2}.
\]
From the leftmost estimate it follows that 
\[
\sum_{j}\mu(P_{j})\leq c_{2}\sum_{j}\mu(P_{j}\cap\Omega)\leq c_{2}\mu(\Omega)\leq3n\rho c_{1}c_{2}\mu(Q),
\]
and choosing $\rho=\frac{1}{6nc_{1}c_{2}}$, 
\[
\sum_{j}\mu(P_{j})\leq\frac{1}{2}\mu(Q).
\]
Note that by the Lebesgue differentiation theorem there exists some
set $N$ of measure zero such that 
\begin{equation}
\Omega\setminus N\subset\Omega_{c_{2}}=\bigcup_{j}P_{j}\label{eq:Null}.
\end{equation}
Now we show that for this family $\{P_{j}\}$ \eqref{eq:Claim} holds.
Taking \eqref{eq:Null} into account if $x\in Q\setminus\cup P_{j}$
then the inequalities in $\Omega$ hold reversed a.e. in particular,
\[
\left|\langle T(\vec{f}\chi_{\alpha Q})(x),\vec{e_{i}}\rangle\right|=\left|T(\langle\vec{f}\chi_{\alpha Q},\vec{e_{i}}\rangle)(x)\right|\leq B(\rho)\left(\frac{1}{\mu(\alpha Q)}\int_{\alpha Q}|\langle\vec{f},\vec{e_{i}}\rangle|^{s}\right)^{\frac{1}{s}}.
\]
Now we deal with each term $T(\vec{f}\chi_{\alpha Q\setminus\alpha P_{j}})(x)\chi_{P_{j}}(x)$.
First we note that it is not hard to check that $\mu(P_{j}\setminus\Omega)\not=0.$ 
Then for each $x\in P_j$ and $x'\in P_j\setminus\Omega$,  
\[
\begin{split} & \left|\langle T(\vec{f}\chi_{\alpha Q\setminus\alpha P_{j}})(x)\chi_{P_j},\vec{e_{i}}\rangle\right|=|T(\langle\vec{f}\chi_{\alpha Q\setminus\alpha P_{j}},\vec{e_{i}}\rangle)(x)\chi_{P_j}|\\
	& \leq|T(\langle\vec{f}\chi_{\alpha Q\setminus\alpha P_{j}},\vec{e_{i}}\rangle)(x)\chi_{P_j}-T(\langle\vec{f}\chi_{\alpha Q\setminus\alpha P_{j}},\vec{e_{i}}\rangle)(x')|+\left|T(\langle\vec{f}\chi_{\alpha Q\setminus\alpha P_{j}},\vec{e_{i}}\rangle)(x')\right|.
\end{split}
\]
For the first term, given that $x'\in P_{j}\setminus\Omega$, we have that 
\[
\begin{split}
	|T(\langle\vec{f}\chi_{\alpha Q\setminus\alpha P_{j}},\vec{e_{i}}\rangle)(x)-T(\langle\vec{f}\chi_{\alpha Q\setminus\alpha P_{j}},\vec{e_{i}}\rangle)(x')|&\leq\mathcal{M}_{T,\alpha}^{\#}(\langle\vec{f}\chi_{\alpha Q},\vec{e_{i}}\rangle)(x')\\
	& \leq B(\rho)\left(\frac{1}{\mu(\alpha Q)}\int_{\alpha Q}|\langle\vec{f},\vec{e_{i}}\rangle|^{s}\right)^{\frac{1}{s}}.
\end{split}
\]
Then 
\[
|\langle T(\vec{f}\chi_{\alpha Q\setminus\alpha P_{j}})(x),\vec{e_{i}}\rangle|\leq 2B(\rho)\left(\frac{1}{\mu(\alpha Q)}\int_{\alpha Q}|\langle\vec{f},\vec{e_{i}}\rangle|^{s}\right)^{\frac{1}{s}}+\inf_{x'\in P_{j}\setminus\Omega}\left|T(\langle\vec{f}\chi_{\alpha Q\setminus\alpha P_{j}},\vec{e_{i}}\rangle)(x')\right|
\]
For the remaining term we note that 
\[\dfrac{1}{2}\mu(P_j)\leq\mu(P_j\setminus\Omega),\]
which implies that 
\[\inf_{x'\in P_{j}\setminus\Omega}\left|T\langle \vec{f}\chi_{\alpha P_{j}},\vec{e_{i}}\rangle(x')\right| \leq B(\rho)\left(\frac{1}{\mu(\alpha P_{j})}\int_{\alpha P_{j}}|\langle \vec{f},\vec{e_{i}}\rangle|^{s}\right)^{\frac{1}{s}}\]
since otherwise we would have that
\[P_{j}\setminus\Omega\subset\left\{ x\in P_{j}\,:\,\left|T(\langle\vec{f}\chi_{\alpha P_{j}},\vec{e_{i}}\rangle)(x)\right|>B(\rho)\left(\frac{1}{\mu(\alpha P_{j})}\int_{\alpha P_{j}}|\langle\vec{f},\vec{e_{i}}\rangle|^{s}\right)^{\frac{1}{s}}\right\}\, ,\]
which would in turn imply that 
\[
\begin{split} 
	\frac{1}{2}\mu(P_{j})& \leq\mu(P_{j}\setminus\Omega)\\
	&\leq\mu\left(\left\{ x\in P_{j}\,:\,\left|T(\langle\vec{f}\chi_{\alpha P_{j}},\vec{e_{i}}\rangle)(x)\right|>B(\rho)\left(\frac{1}{\mu(\alpha P_{j})}\int_{\alpha P_{j}}|\langle\vec{f},\vec{e_{i}}\rangle|^{s}\right)^{\frac{1}{s}}\right\} \right)\\
	&\leq \rho \mu(\alpha P_j) \leq \rho c_1 \mu(P_j)=\frac{1}{6nc_{2}}\mu(P_{j}),
\end{split}
\]
which is a contradiction. Then we have that
\begin{equation}\label{10.16}
	\begin{split}\inf_{x'\in P_{j}\setminus\Omega}\left|T\langle \vec{f}\chi_{\alpha P_{j}},\vec{e_{i}}\rangle(x')\right| & \leq B(\rho)\left(\frac{1}{\mu(\alpha P_{j})}\int_{\alpha P_{j}}|\langle \vec{f},\vec{e_{i}}\rangle|^{s}\right)^{\frac{1}{s}}\\
		& \leq B(\rho)\inf_{x'\in P_{j}\setminus\Omega}M_{s}(\langle \vec{f},\vec{e_{i}}\rangle\chi_{\alpha Q})(x')\\
		& \leq B(\rho)\left(\rho^{-1}\|M\|_{L^{s}\rightarrow L^{s,\infty}}\right)^{\frac{1}{s}}\left(\frac{1}{\mu(\alpha Q)}\int_{\alpha Q}|\langle \vec{f},\vec{e_{i}}\rangle|^{s}\right)^{\frac{1}{s}}.
	\end{split}
\end{equation}
Gathering the estimates above we have that
\[
\begin{split}
	&| \langle T(\vec{f}\chi_{\alpha Q})(x)\chi_{Q\setminus \cup P_j} + \sum_j T(\vec{f}\chi_{\alpha Q\setminus \alpha P_j})(x)\chi_{P_j}(x),\vec{e_i}\rangle |\\
	&\leq \left(3B(\rho) +B(\rho)\left(\rho^{-1}\|M\|_{L^{s}\rightarrow L^{s,\infty}}\right)^{\frac{1}{s}}  \right)\left(\frac{1}{\mu(\alpha Q)}\int_{\alpha Q}|\langle \vec{f},\vec{e_{i}}\rangle|^{s}\right)^{\frac{1}{s}}\\
	&=n\left(3B(\rho) +B(\rho)\left(\rho^{-1}\|M\|_{L^{s}\rightarrow L^{s,\infty}}\right)^{\frac{1}{s}}  \right)\dfrac{1}{n}\left(\frac{1}{\mu(\alpha Q)}\int_{\alpha Q}|\langle \vec{f},\vec{e_{i}}\rangle|^{s}\right)^{\frac{1}{s}}.
\end{split}
\]
and hence (\ref{eq:Claim2}) holds.

The representation in terms of functions $k_{Q}(x,y)$ follows from
Lemma \ref{lem:RepkQ}.
\end{proof}
Finally armed with the Lemma above we are in the position to settle
Theorem \ref{thm:ThmSparse}. 
\begin{proof}[Proof of Theorem \ref{thm:ThmSparse}]

Let $Q=Q_1^0$ be a cube. Iterating Lemma \ref{lem:It} we get a $\frac{1}{2}$-sparse family $\lbrace Q_j^k\rbrace\subset \mathcal{D}(Q)$ 
with
\[
\sum_{Q_{j}^{k+1}\subset Q_{i}^{k}}\mu(Q_{j}^{k+1})\leq\frac{1}{2}\mu(Q_{i}^{k})\, .
\] 
and functions $k_{Q_j^k}$ with  \[\left\Vert k_{Q_{j}^{k}}(x,y)\right\Vert _{L^{(s',\infty)}\left[(\alpha Q_{j}^{k},\frac{d\mu(y)}{\mu(\alpha Q_{j}^{k})}),(Q_{j}^{k},d\mu(x))\right]}\leq1\, ,\]	
such that

\begin{equation*}
	\label{T1.1}
	\begin{split}T(\vec{f}\chi_{\alpha Q})\chi_{Q}(x) & =\kappa\sum_{k=0}^{K-1}\sum_{j}\frac{1}{\mu(\alpha Q_{j}^k)}\int_{\alpha Q_{j}^k}k_{Q_{j}^{k}}(x,y)\vec{f}(y)dyd\mu(y)\chi_{Q_{j}^{k}}(x)\\
		& +\sum_{j}T(\vec{f}\chi_{\alpha Q_{j}^{K}})(x)\chi_{Q_{j}^{K}}(x)\, .
	\end{split}
\end{equation*}

Note that $\sum_j\mu(Q_j^K)\leq \frac{1}{2^K}\mu(Q)$, then $$\displaystyle\lim_{K\rightarrow\infty}\sum_j\mu(Q_j^K)=0$$ and hence, letting $K\rightarrow\infty$, for almost every $x\in Q$
\[
\begin{split}
	T(\vec{f}\chi_{\alpha Q})(x)\chi_{Q}(x)&=\kappa\sum_{k=0}^{\infty}\sum_{j}\dashint_{\alpha Q^k_{j}}k_{Q_{j}^{k}}(x,y)\vec{f}(y)d\mu(y)\chi_{Q_{j}^{k}}(x)\\
	&=\sum_{Q_j\in\mathcal{S}_Q}\dashint_{\alpha Q_{j}}k_{Q_{j}}(x,y)\vec{f}(y)d\mu(y)\chi_{Q_{j}}(x)\\
\end{split}	
\]
An application of Lemma \ref{lemma:covering} with $E=\text{sop}(f)$, allows us to say that there exists a partition of $X$, $\mathcal{P}\subset \mathcal{D}$ such that $E\subseteq \alpha Q $ for every $Q\in\mathcal{P}$. Then, 
\begin{equation}
	\label{diadicogenerico}
	\begin{split}
		T\vec{f}(x)&=\sum_{Q\in\mathcal{P}}T(\vec{f}\chi_{\alpha Q})(x)\chi_{Q}(x)\\
	\end{split}
\end{equation}
and it suffices to apply the estimate above to each term and we are done. Indeed, since $\mathcal{P}$ is a partition $\mathcal{S}=\cup_{Q\in\mathcal{P}}\mathcal{S}_{Q}$ is a $\frac{1}{2}$-sparse family.

To prove the furthermore part, we fix the parameters in Proposition \ref{proposition:dyadicsystem}. Then there exist $\mathcal{D}_{1},\dots,\mathcal{D}_{m}$ dyadic systems associated to those parameters  and a constant $\gamma\geq 1$ such that for every ball $B=B(s,\rho)$ there exists $Q\in \mathcal{D}_t$, $1\leq t\leq m$, such that
$$B(s,\rho)\subseteq Q, \hspace{1cm}  \hspace{1cm} diam(Q)\leq \gamma\rho.$$
Thus, repeating the argument above for  $\mathcal{D}_1$,  we have that for every $P\in\mathcal{S}\subset \mathcal{D}_1$, there exists $P'\in \mathcal{D}_j$ ($1\leq j\leq m$) such that 
$$\alpha P=B(z,\alpha C_0\delta^k)\subseteq P' \, , \quad diam(P')\leq \gamma\alpha C_0\delta^k\, .$$
By the doubling property there exists $c=c(X,\alpha)>0$ such that $$\mu(P')\leq \mu(B(z,\gamma\alpha C_0\delta^k))\leq c\mu(B(z,c_0\delta^k))\leq c\mu(P).$$
Since $\mathcal{S}$ is a $\frac{1}{2}$-sparse family, for every $P\in \mathcal{S}$ there exists $E_P\subset P$ such that $\frac{\mu(E_P)}{\mu(P)}\geq \frac{1}{2}$, and  $\left\lbrace E_P\right\rbrace_{P\in \mathcal{S}} $ is a disjoint family. Take $E_{P'}=E_P$ and then
$$\dfrac{\mu(E_{P'})}{\mu(P')}=\dfrac{\mu(E_P)}{\mu(P')}\geq\dfrac{\mu(E_P)}{c\mu(P)}\geq \dfrac{1}{2c},$$
therefore $\mathcal{S}_j=\left\lbrace P'\in\mathcal{D}_j:P\in\mathcal{S}\right\rbrace $ is a $\frac{1}{2c}$-sparse family.
We continue defining 
\[
k_{P'}(x,y)= \frac{\mu(P')}{\mu(\alpha P)}c^{-\frac{1}{s}} k_{P}(x,y)\chi_{P\times \alpha P}(x,y). 
\]
Then
\[
\begin{split}
	T\vec{f}(x)&=\sum_{P\in\mathcal{P}}T(\vec{f}\chi_{\alpha P})(x)\chi_{P}(x)\\
	&=\sum_{P\in\mathcal{S}} \dashint_{\alpha P}k_{P}(x,y)\vec{f}(y)d\mu(y)\chi_{P}(x)\\
	&=\sum_{P\in\mathcal{S}}   \frac{1}{\mu(P')}\int_{P'}\frac{\mu(P')}{\mu(\alpha P)}k_{P}(x,y)\vec{f}(y)\chi_{P\times \alpha P}(x,y)d\mu(y)\\
	&=c^{\frac{1}{s}} \sum_{j=1}^m\sum_{P'\in\mathcal{S}_j\subset \mathcal{D}_j}   \frac{1}{\mu(P')}\int_{P'}k_{P'}(x,y)\vec{f}(y)d\mu(y)\\
\end{split}
\]
where
\[
\left\Vert \left\Vert k_{P'}(x,y)\right\Vert _{L^{s'}\left(P',\frac{d\mu(y)}{\mu(P')}\right)}\right\Vert _{L^{\infty}(P',d\mu(x))}\leq1.
\]
Indeed,
\begin{align*}
 & \|k_{P'}(x,y)\|_{L^{s'}\left(P',\frac{d\mu(y)}{\mu(P')}\right)}={\displaystyle \sup_{\Vert h\Vert_{L^{s}}=1}\int_{P'}|k_{P'}(x,y)||h(y)|\frac{d\mu(y)}{\mu(P')}}\\
= & {\displaystyle \sup_{\Vert h\Vert_{L^{s}\left(P',\frac{d\mu(y)}{\mu(P')}\right)}=1}\int_{P'}\left|\frac{1}{\mu(\alpha P)}c^{-\frac{1}{s}}k_{P}(x,y)\chi_{P\times\alpha P}(x,y)\right||h(y)|d\mu(y)}\\
\leq & {\displaystyle \sup_{\Vert h\Vert_{L^{s}\left(P',\frac{d\mu(y)}{\mu(P')}\right)}=1}\left(\frac{1}{\mu(\alpha P)}\int_{P'}|c^{-\frac{1}{s}}k_{P}(x,y)|^{s'}\chi_{P\times\alpha P}(x)d\mu(y)\right)^{\frac{1}{s'}}\left(\frac{1}{\mu(\alpha P)}\int_{P'}|h(y)|^{s}d\mu(y)\right)^{\frac{1}{s}}}\\
= & {\displaystyle \sup_{\Vert h\Vert_{L^{s}\left(P',\frac{d\mu(y)}{\mu(P')}\right)}=1}c^{-\frac{1}{s}}\left(\dfrac{1}{\mu(\alpha P)}\int_{\alpha P}|k_{P}(x,y)|^{s'}\chi_{P}(x)d\mu(y)\right)^{\frac{1}{s'}}\left(\frac{\mu(P')}{\mu(\alpha P)}\cdot\frac{1}{\mu(P')}\int_{P'}|h(y)|^{s}d\mu(y)\right)^{\frac{1}{s}}}\\
\leq & {\displaystyle \sup_{\Vert h\Vert_{L^{s}\left(P',\frac{d\mu(y)}{\mu(P')}\right)}=1}c^{-\frac{1}{s}}\left(\dfrac{1}{\mu(\alpha P)}\int_{\alpha P}|k_{P}(x,y)|^{s'}\chi_{P}(x)d\mu(y)\right)^{\frac{1}{s'}}\cdot c^{\frac{1}{s}}\Vert h\Vert_{L^{s}\left(P',\frac{d\mu(y)}{\mu(P')}\right)}}\\
= & \left(\dfrac{1}{\mu(\alpha P)}\int_{\alpha P}|k_{P}(x,y)|^{s'}\chi_{P}(x)d\mu(y)\right)^{\frac{1}{s'}}\\
= & \Vert k_{P}(x,y)\Vert_{L^{s'}\left(\alpha P,\frac{d\mu(y)}{\mu(\alpha P)}\right)}\leq1
\end{align*}
and we are done.
\end{proof}

\section{The Petermichl operator revisited. Haar multipliers with variable kernel}\label{sec:PetOp}

If $(X,d,\mu)$ is a homogeneous type space with dyadic family $\mathcal{D}$, Aimar, Bernardis and Iaffei \cite{ABI} showed that Haar-type wavelets associated to the dyadic family $\mathcal{D}$ can be constructed and that such systems of Haar functions turn out to be unconditional bases for Lebesgue spaces $L^p(X,\mu)$ with $1 < p < \infty$. More precisely, the authors proved the following result.

\begin{prop}
Let $\mathcal{D}$ be a dyadic family on the space of homogeneous type $(X,d,\mu)$. There exists a system $\mathcal{H}$ of simple Borel measurable real functions $h$ on $X$ that satisfies the following properties:
\begin{itemize}
\item[\textit{(h.1)}] \textit{ For each $h \in \mathcal{H}$ there exists a unique $j \in \mathbb{Z}$ and a cube
$Q(h) \in  \mathcal{\tilde{D}}^j$ such that $\{x \in X: h(x) \not=0\} \subseteq Q(h)$, and this property does not hold for any cube
in $\mathcal{D}^{j+1}$.}
\item[\textit{(h.2)}] \textit{ For every $Q \in \mathcal{\tilde{D}}$ there exist exactly $M_Q = \#(\mathcal{L}(Q)) - 1 \geq 1$ functions $h \in \mathcal{H}$ such
that (h.1) holds where $\mathcal{L}(Q)$ is the set of the dyadic cube children of $Q$. We denote with $\mathcal{H}(Q)$  the set
of all these functions $h$.}
\item[\textit{(h.3)}] \textit{ For each  $h \in \mathcal{H}$ we have that $\int_X h
d \mu = 0$.}
\item[\textit{(h.4)}] \textit{ For each $Q\in  \mathcal{\tilde{D}}$ let $V_Q$ denote the
vector space of all functions on $Q$ which are constant on each
$Q^{'} \in \mathcal{L}(Q)$. Then the system
$\{\frac{\chi_{_{Q}}}{(\mu(Q))^{1/2}}\ \}\ \bigcup \mathcal{H}(Q)$
is an orthonormal basis for $V_Q$.}
\item[\textit{(h.5)}] \textit{ There exists a positive constant $C$ such that the inequality $|h(x)| \leq C |h(y)|$
holds for almost every $x$ and $y$ in $Q(h)$ and every $h \in \mathcal{H}$.}
\item[\textit{(h.6)}] \textit{$\mathcal{H}$ is an orthonormal basis of $L^2(X,\mu)$.}
\item[\textit{(h.7)}] \textit{$\mathcal{H}$ is an unconditional basis to $L^p(X,\mu)$ with $1 < p < \infty$.}
\end{itemize}
\end{prop}

The system $\mathcal{H}$ in the previous result will be called the Haar system associated with the dyadic family $\mathcal{D}$. In the sequel we shall write $\mathcal{S(H)}$ to denote the set of lineal combination of Haar functions in $\mathcal{H}$ and $\mathcal{S(\vec{H})}$ to denote the set of lineal combination of  vectorial Haar functions $\vec{h} \in \vec{\mathcal{H}}$, that is $\vec{h} = (h_1,h_2,...,h_n)$ with $h_i \in \mathcal{H}$ for $i=1,...,n$.
\

On the other hand, given a dyadic family $\mathcal{D}$ in the space of homogeneous type $(X,d,\mu)$, we say that $X$ has a unique quadrant when there exist a dyadic cube $Q \in \mathcal{D}$ such that
$$X = \underset{\{Q^{'}\in \mathcal{D}: Q \subseteq Q^{'}\}}{\bigcup}Q^{'}.$$
In this case we can define a dyadic metric, $\delta_{dy}$, associated to $\mathcal{D}$ such that the $\delta_{dy}$-balls are precisely the dyadic cubes of $\mathcal{D}$ (see \cite{AG, ACN} for futher details). This dyadic metric $\delta_{dy}$ is defined in $X \times X$ given by
\begin{align}\label{ultrametrica}
\delta_{dy} (x,y)=
\begin{cases}
min\{\mu(Q): x,y \in Q, Q \in \tilde{\mathcal{D}}\}&\text{\,if\,} \ \  x \not= y \\
0 &\text{\,if\,} \ \ x=y.
\end{cases}
\end{align}

In what follows, in this section we will always assume that $(X,d,\mu)$ has a single quadrant and we will consider as natural dyadic environment the homogeneous type space $(X,\delta_{dy},\mu)$.  Also we will write $<f,g>$ to denote $\int_X f(x)g(x) d\mu(x)$.
First a few comments on ball dilation associated with dyadic metric $\delta_{dy}$. Since the $\delta_{dy}$-balls $B$ are precisely the dyadic cubes $Q$ of the family $\mathcal{D}$, if $\alpha$ is a positive integer we have that $\alpha B$ is precisely the $\alpha$-th ancestor of $B$. That is, $\alpha B \in \mathcal{\tilde{D}}^{\alpha + j}$ when $B \in \mathcal{\tilde{D}}^{j}$.
This observation about dilated balls associated with the dyadic metric will be useful in what follows.

We will study properties of Haar multiplier type operators with variable kernel 
$$T_\eta f(x) \ = \ \sum_{h \in \mathcal{H}} \eta(x,h)\left<f,h\right> h(x)$$
when the function $\eta: X \times \mathcal{H} \longrightarrow \mathbb{R}$ verify some regularity properties that we will specify later. As a particular case of such operators we can cite the Petermichl operator $\mathcal{P}$ (see \cite{P}) given by $$\mathcal{P} f(x) \ = \ \sum_{I \in \mathcal{D}}\langle f,h_I\rangle \left(h_{I^-}(x) - h_{I^+}(x)\right)$$ with $x \in \mathbb{R}^+$. Here, as usual, $h_I$ denote the Haar wavelets with support in the dyadic interval $I$ and $h_{I^-}$, $h_{I^+}$ the Haar wavelets in the left and right halves of the dyadic interval $I$.
The operator $\mathcal{P}$ can be written as (see \cite{ACN})
$$\mathcal{P} f(x)\ =\ \frac{1}{\sqrt[]{ 2}}\sum_{h \in \mathcal{H}} \eta(x,h) \langle f,h\rangle h(x)$$
with $\eta(x,h) = 1$ if $x \in I^{--}_h \cup I^{+-}_h$ and $\eta(x,h) = -1$ if $x \in I^{-+}_h \cup I^{++}_h$, where for $h\in \mathcal{H}$ we denote with $I(h)$ the interval support of $h$, and we consider as $I^{--}_h$ the left quarter of $I(h)$,  $I^{-+}_h$ as the second quarter, $I^{+-}_h$ as the third quarter and $I^{++}_h$ as the last quarter of $I(h)$. 

Now we establish the conditions and regularity of the function $\eta$ of the operator $T_\eta$: Let $(X,d,\mu)$ be a space of homogeneous type, $\mathcal{D}$ a dyadic system, $d_{dy}$ the associated dyadic metric, and $\mathcal{H}$ a Haar system. We consider functions $\eta:X\times \mathcal{H}\rightarrow \mathbb{R}$ such that

\begin{itemize}
	\item[(a)]
	$|\eta(x,h)| \leq K,$ for all $x \in X$ and $h \in \mathcal{H}$
	\item[(b)]
	$|\eta(x',h) - \eta(x,h)| \leq K \frac{d_{dy}(x,x')}{\mu(Q(h))}$, for all $h \in \mathcal{H}$ and $x, x' \in X$. 
\end{itemize}
where $K>0$ is a constant, and $Q(h)$ is the cube associated with the function $h$. 

The following result establishes that under these conditions, the operator $T_\eta$ satisfies conditions of Theorem \ref{thm:ThmSparse}.

\begin{lem}
	\label{Teorema 5}
	Let $(X,d_{dy},\mu)$ be a space of homogeneous type with $d_{dy}$ the metric asocciated to the dyadic system $\mathcal{D}$ with parameters $c_{0}$, $C_{0}$ and $\delta$. Then for every $d_{dy}$-ball $B$, the operator $T_\eta$ defined as 
	$$T_\eta f(x) \ = \ \sum_{h \in \mathcal{H}} \eta(x,h)\left<f,h\right> h(x)\, ,$$
	satisfies that there exists a non increasing function $\psi$ such that
	\begin{equation}\label{debil T}
		\mu\left(\{x\in B:|T_\eta(g\chi_{B})(x)|>\psi(\lambda)\langle g\rangle_{1,B}\}\right)\le\lambda\mu(B)\quad  \quad  (0<\lambda<1)
	\end{equation}
	for every $g\in L^1(X)$, and there exists a decreasing function $\phi$ such that			
	\begin{equation}\label{debilTmax}
		\mu(\{x\in B:\mathcal{M}_{T_\eta,\alpha}^{\#}(f\chi_{B})(x)>\phi(\lambda)\langle f\rangle_{1,B}\})\le\lambda\mu(B)\quad  \quad(0<\lambda<1)
	\end{equation}
	for every $f\in L^1(X)$.
\end{lem}

\begin{proof}
	By Theorem 4.1 in \cite{ACN} we have that the operator $T_\eta$ is a Calder\'on-Zygmund operator on the space $(X,d_{dy},\mu)$ and therefore is of weak type $(1,1)$. Then, applying Jensen's inequality and taking $\psi(\lambda)=\frac{1}{\lambda}\|T_{\eta}\|_{L^1\rightarrow L^{(1,\infty)}}$ we obtain	
	\begin{align*}
\mu\left(\left\{ x\in B:|T_{\eta}(g\chi_{B})|>\psi(\lambda)\langle g\rangle_{1,B}\right\} \right) & \leq\frac{\|T_{\eta}\|_{L^{1}\rightarrow L^{(1,\infty)}}}{\psi(\lambda)\langle g\rangle_{1,B}}\int_{B}|g|\\
 & =\frac{\lambda\|T_{\eta}\|_{L^{1}\rightarrow L^{(1,\infty)}}}{\|T_{\eta}\|_{L^{1}\rightarrow L^{(1,\infty)}}\langle g\rangle_{1,B}}\int_{B}|g|\\
 & =\lambda\mu(B)
\end{align*}
which proves \eqref{debil T}.

	Although \eqref{debilTmax} follows from the fact that $T_\eta$ is a Calder\'on-Zygmund operator on the space $(X,d_{dy},\mu)$ we provide a direct argument to settle that estimate. 

First we call
\[
N(x,y)=\sum_{h\in\mathcal{H}}\eta(x,h)h(x)h(y).
\]

\noindent Let $x\in X$ and $B$ a $d_{dy}$-ball such that $x\in B$,
let $\alpha\in\mathbb{N}$ and let $y\in B$. If $w\in X\setminus\alpha B$,
then,
\begin{itemize}
\item If $Q(h)\subset\alpha B$, then $h(w)=0$. Therefore, $h(y)h(w)=0$. 
\item If $Q(h)\cap\alpha B=\emptyset$, since $y\in B$, then $y\notin Q(h)$,
and thus $h(y)=0$. Consequently $h(y)h(w)=0$.
\end{itemize}
These are the only possible cases since $\alpha B$ and $Q(h)$ are
dyadic cubes and cannot intersect partially.

Then we are left with the non-zero terms and we have that

\[
N(y,w)=\sum_{h\in\mathcal{H}}\eta(y,h)h(w)h(y)=\sum_{\underset{Q(h)\supseteq\alpha B}{h\in\mathcal{H}}}\eta(y,h)h(w)h(y).
\]
At this point we recall that if $y\in X\setminus\alpha B$ and $x\in B$,
\[
\sum_{\underset{Q(h)\supseteq\alpha B}{h\in\mathcal{H}}}h(x)h(y)\leq\frac{1}{\mu(Q(x,y))}
\]
where $Q(x,y)$ is the smallest dyadic cube that contains $x$ and
$y$. Since $\eta$ is bounded this yields 
\[
|N(x,y)|\leq\frac{K}{\mu(Q(x,y))}
\]
for some $K>0$ depending on $\eta$.  Bearing this in mind we show
that $f\in L^{1}(X)$, if $B$ is a $d_{dy}$-ball and $x\in B$ then
\[
T_{\eta}(f\chi_{\{X\setminus\alpha B\}})(x)=\int_{X\setminus\alpha B}N(x,y)f(y)d\mu(y).
\]
Note that it suffices to show that
\[
y\longmapsto|N(x,y)||f(y)|\in L^{1}(X\setminus\alpha B).
\]
and, by the estimates above, in order to prove that, it suffices to
show that 
\[
\int_{X\setminus\alpha B}\frac{|f(y)|}{\mu(Q(x,y))}d\mu(y)<\infty.
\]
Let $B_{0}=\alpha B$ and and for every non negative integer $k$ let $B_k$ the $k$-th ancestor of $\alpha B$, we have that 
\[
X\setminus\alpha B=\bigcup_{k\geq0}B_{k+1}\setminus B_{k}.
\]
Hence
\begin{align*}
\int_{X\setminus\alpha B}\frac{|f(y)|}{\mu(Q(x,y))} & =\sum_{k=0}^{\infty}\int_{B_{k+1}\setminus B_{k}}\frac{|f(y)|}{\mu(Q(x,y))}d\mu(y)\leq\sum_{k=0}^{\infty}\int_{B_{k+1}\setminus B_{k}}\frac{|f(y)|}{\mu(Q(x,y))}d\mu(y)\\
 & \leq\sum_{k=0}^{\infty}\frac{1}{\mu(B_{k})}\int_{B_{k+1}\setminus B_{k}}|f(y)|d\mu(y)\leq\|f\|_{L^{1}}\sum_{k=0}^{\infty}\frac{1}{\mu(B_{k})}<\infty
\end{align*}
 since $\mu(B_k)\gtrsim c_\mu^{-k \log_2(\delta)}\mu(B_0)$ for every $k\geq0$.

Now we can continue our argument as follows. If $B$ is a ball and
$y,z\in B$ then we have that

\begin{eqnarray*}
 &  & \left|T_{\eta}(f\chi_{X\setminus\alpha B})(y)-T_{\eta}(f\chi_{X\setminus\alpha B})(z)\right|\\
 & = & \left|\int_{X\setminus\alpha B}\left(N(y,w)-N(z,w)\right)f(w)d\mu(w)\right|\\
 & = & \left|\int_{X\setminus\alpha B}\sum_{\underset{\alpha B\subset Q(h)}{h\in\mathcal{H}}}\left(\eta(y,h)h(y)-\eta(z,h)h(z)\right)h(w)f(w)d\mu(w)\right|\\
 & = & \left|\int_{X\setminus\alpha B}\sum_{\underset{\alpha B\subset Q(h)}{h\in\mathcal{H}}}\left(\eta(y,h)-\eta(z,h)\right)h(y)h(w)f(w)d\mu(w)\right|\\
 & \leq & \int_{X\setminus\alpha B}\sum_{\underset{\alpha B\subset Q(h)}{h\in\mathcal{H}}}\left|\left(\eta(y,h)-\eta(z,h)\right)h(y)h(w)f(w)\right|d\mu(w)\\
 & \leq & C\int_{X\setminus\alpha B}\sum_{\underset{\alpha B\subset Q(h)}{h\in\mathcal{H}}}\left|\eta(y,h)-\eta(z,h)\right|\frac{1}{\mu(Q(h))}|f(w)|d\mu(w)\\
 & \leq & KC\int_{X\setminus\alpha B}\sum_{\underset{\alpha B\subset Q(h)}{h\in\mathcal{H}}}\frac{d_{dy}(y,z)}{\mu(Q(h))}\frac{1}{\mu(Q(h))}|f(w)|d\mu(w)\\
 & = & \tilde{C}d_{dy}(y,z)\int_{X\setminus\alpha B}\sum_{\underset{\alpha B\subset Q(h)}{h\in\mathcal{H}}}\frac{1}{(\mu(Q(h)))^{2}}|f(w)|d\mu(w)\\
 & \leq & \tilde{C}d_{dy}(y,z)\int_{X\setminus\alpha B}\frac{|f(w)|}{(d_{dy}(z,w))^{2}}d\mu(w).
\end{eqnarray*}
Again, calling $B_{0}=\alpha B$ and 
$B_{k}$ the $k$-th ancestor of $\alpha B$ for every non negative integer $k$, we continue the
computation above as follows. 
\begin{eqnarray*}
 &  & \tilde{C}d_{dy}(y,z)\int_{X\setminus\alpha B}\frac{|f(w)|}{(d_{dy}(z,w))^{2}}d\mu(w)=\tilde{C}d_{dy}(y,z)\sum_{k=0}^{\infty}\int_{B_{k+1}\setminus B_{k}}\frac{|f(w)|}{(d_{dy}(z,w))^{2}}d\mu(w)\\
 &  & \leq\tilde{C}\mu(B_{0})\sum_{k=0}^{\infty}\frac{1}{(\mu(B_{k+1}))^{2}}\int_{B_{k+1}}|f(w)|d\mu(w)\lesssim C\sum_{k=1}^{\infty}\frac{\mu(B_{0})}{\mu(B_{k})}M^{\mathcal{D}}f(x)\\
 &  & \lesssim\sum_{k=1}^{\infty}    c_{\mu}^{k\log_{2}(\delta)}  M^{\mathcal{D}}f(x)=\sum_{k=1}^{\infty}C^{k}M^{\mathcal{D}}f(x)=\dfrac{C}{1-C}M^{\mathcal{D}}f(x).
\end{eqnarray*}
Summarizing, for $C=c_{\mu}^{\log_{2}\delta}$ we have that

\[
\left|T_{\eta}(f\chi_{X\setminus\alpha B})(y)-T_{\eta}(f\chi_{X\setminus\alpha B})(z)\right|\lesssim \dfrac{C}{1-C}M^{\mathcal{D}}f(x).
\]

From this it readily follows that $\mathcal{M}_{T_{\eta},\alpha}^{\#}f(x)\lesssim M^{\mathcal{D}}f(x)$,
and therefore $\mathcal{M}_{T_{\eta},\alpha}^{\#}$ is of weak type
$(1,1)$ and then (\ref{debilTmax}) holds with $\phi(\lambda)=\dfrac{1}{\lambda}\|\mathcal{M}_{T_{\eta},\alpha}^{\#}\|_{L^{1}\rightarrow L^{(1,\infty)}}$. 
\end{proof}

The result above allows us to apply Theorem \ref{thm:ThmSparse} to Haar multiplier operators obtaining the following Corollary.

\begin{cor}
	\label{Teorema 6}
	Let $(X,d_{dy},\mu)$ a space of homogeneous type with $d_{dy}$ the metric asocciated to the dyadic system $\mathcal{D}$ with parameters $c_{0}$, $C_{0}$ y $\delta$. Let $\alpha\geq\frac{3c_{d}^{2}}{\delta}$ and  $\vec{f}:X\rightarrow\mathbb{R}^{n}$ a boundedly supported funtion $|\vec{f}|\in L^{1}(X)$, then, there exists a $\frac{1}{2}$-sparse family $\mathcal{S}\subset\mathcal{D}$ such that
	\[
	T\vec{f}(x)\in c_{n,d}c_{T}\sum_{Q\in\mathcal{S}}\langle\langle\vec{f}\rangle\rangle_{s,\alpha Q}\chi_{Q}(x).
	\]
	In other words, we have that
	\begin{equation}
		\label{T1.a}
		T\vec{f}(x)=c_{n,d}c_{T}\sum_{Q\in\mathcal{S}}\frac{1}{\mu(\alpha Q)}\int_{\alpha Q}\vec{f}(y)k_{Q}(x,y)d\mu(y)\chi_{Q}(x)
	\end{equation}
	where
	\[
	\left\Vert k_{Q}\right\Vert _{L^{(\infty,\infty)}\left[(Q,\frac{\mu}{\mu(Q)}),(\alpha Q,\mu)\right]}\leq1
	\]
	furthermore, there exist $0<c_{0}\leq C_{0}<\infty$, $0<\delta<1$, $\gamma\geq1$ and $m\in\mathbb{N}$ such that there are dyadic systems $\mathcal{D}_{1},\dots,\mathcal{D}_{m}$ with parameters $c_{0},C_{0}$ and $\delta$, and $m$ sparse families $\mathcal{S}_{i}\subset\mathcal{D}_{i}$ such that
	\begin{equation}
		\label{T1.b}
		T\vec f(x)=c_{n,d}c_{T}\sum_{j=1}^{m}\sum_{Q\in\mathcal{S}_{j}}\frac{1}{\mu(Q)}\int_{Q}\vec{f}(y)k_{Q}(x,y)d\mu(y)\chi_{Q}(x)
	\end{equation}
	where $\left\Vert k_{Q}\right\Vert _{L^{(\infty,\infty)}\left[(Q,\frac{d\mu(y)}{\mu(Q)}),(Q,d\mu(x))\right]}\leq1$.
\end{cor}	
At this point it is worth noting that the weighted estimates obtained in Section \ref{sec:WEst} also work for $T_\eta$ under the conditions of the Theorem above, since the estimates provided there were obtained in terms convex body sparse operators. 

\section*{Declarations}
\subsection*{Ethical Approval} 
Not applicable
\subsection*{Funding}
The last author was partially supported by the Spanish Ministry of Science and Innovation
through the project PID2022-136619NB-I00 and by Junta de Andalucía through the project FQM-354.
\subsection*{Availability of data and materials}
Not applicable

\end{document}